\documentclass{amsart}
\usepackage{amssymb}
\usepackage{latexsym}
\usepackage{graphicx}
\usepackage{subfigure}
\usepackage{tikz}

\def\I{{\mathrm I}}
\def\II{{\mathrm{II}}}
\def\III{{\mathrm{III}}}

\def\L{{\Lambda}}
\def\t{{\theta}}
\def\l{{\lambda}}

\def\b{{\beta}}
\def\a{{\alpha}}
\def\e{{\varepsilon}}
\def\beq{\begin{equation}}
\def\eeq{\end{equation}}

\newcommand{\Z}{{\mathbb Z}}
\newcommand{\R}{{\mathbb R}}

\newcommand{\C}{{\mathbb C}}
\newcommand{\D}{{\mathbb D}}

\newcommand{\N}{{\mathbb N}}
\newcommand{\PP}{{\mathbb P}}

\newcommand{\CA}{{\mathcal A}}
\newcommand{\CC}{{\mathcal C}}
\newcommand{\CD}{{\mathcal D}}

\newcommand{\CB}{{\mathcal B}}
\newcommand{\CH}{{\mathcal H}}
\newcommand{\CI}{{\mathcal I}}

\newcommand{\CN}{{\mathcal N}}

\newcommand{\CU}{{\mathcal U}}

\newcommand{\CJ}{{\mathcal J}}

\newtheorem{theorem}{Theorem}
\newtheorem{remark}{Remark}
\newtheorem{lemma}{Lemma}
\newtheorem{defi}{Definition}
\newtheorem{prop}{Proposition}

\newtheorem{corollary}{Corollary}
\begin{document}

\title[Positivity and Continuity of Lyapunov Exponents]{Uniform Positivity and Continuity of Lyapunov Exponents for a class of $C^2$ Quasiperiodic Schr\"odinger Cocycles}

\author[Y.\ Wang]{Yiqian Wang}

\address{Department of Mathematics, Nanjing University, Nanjing 210093, China}

\email{yiqianw@nju.edu.cn}

\thanks{The first author was supported by the National Natural Science Foundation of China grant No. 11271183.}

\author[Z.\ Zhang]{Zhenghe Zhang}

\address{Department of Mathematics, Northwestern University, Evanston, IL~60208, USA}

\email{zhenghe@math.northwestern.edu}

\begin{abstract} We show that for a class of $C^2$ quasiperiodic
potentials and for any \emph{Diophantine} frequency, the Lyapunov exponents of the corresponding Schr\"odinger cocycles are uniformly positive and weak H\"older continuous as function of energies. As a corollary, we also obtain that the corresponding integrated density of states (IDS) is weak H\"older continous. Our approach is of purely dynamical systems, which depends on a detailed analysis of asymptotic stable and unstable directions. We also apply it to more general $\mathrm{SL}(2,\R)$ cocycles, which in turn can be applied to get uniform positivity and continuity of Lyapuonv exponents around unique nondegenerate extremal points of any smooth potential, and to a certain class of $C^2$ Szeg\H o cocycles.
\end{abstract}

\maketitle
\section{{\bf Introduction}}
Consider the family of Schr\"odinger operators $H_{\alpha,\l v,x}$ on $\ell^2(\Z)\ni u=(u_n)_{n\in\Z}$:
\beq\label{operator}
(H_{\alpha,\lambda v,x}u)_n=u_{n+1}+u_{n-1}+\lambda v(x+n\alpha)u_n.
\eeq
Here $v\in C^r(\R/\Z,\R),r\in\N\cup\{\infty,\omega\}$ is the potential, $\lambda\in\R$ coupling constant, $x\in\R/\Z$ phase, and $\alpha\in\R/\Z$ frequency. For simplicity, we may sometimes left $\a, \l, x$ in $H_{\alpha,\lambda v,x}$ implicit.  Let $\Sigma(H_{\a,\l v,x})$ be the spectrum of the operator. Then it is well-known that
\beq\label{interval-contains-spectrum}
\Sigma_{\a,\l v}\subset [-2+\l\inf v, 2+\l\sup v].
\eeq

Moreover, for irrational $\a$, due to a theorem of Johnson \cite{johnson}, $\Sigma(H_{\a,\l v,x})$ is phase-independent. This follows from minimality of the irrational rotation, see also \cite{zhang2} for a simple proof. Let $\Sigma_{\a,\l v}$ denote the common spectrum in this case.

Consider the eigenvalue equation $H_{\lambda,x}u=Eu.$ Then there is an associated cocycle map which is denoted as
$A^{(E-\lambda v)}\in C^{r}(\mathbb R/\mathbb Z, \mathrm{SL}(2,\mathbb R))$, and is given by
\beq\label{schrodinger-cocycle-map}
A^{(E-\lambda v)}(x)=\begin{pmatrix}E-\lambda v(x)& -1\\1& 0\end{pmatrix}.
\eeq
Then $(\alpha, A^{(E-\lambda v)})$ defines a family of dynamical systems on $(\R/\Z)\times\R^2$, which is given by $(x,w)\mapsto (x+\alpha, A^{(E-\l v)}(x)w)$ and is called the Schr\"odinger cocycle. The $n$th iteration of dynamics is denoted by $(\alpha, A^{(E-\l v)})^n=(n\alpha, A^{(E-\l v)}_n)$. Thus,
$$
A^{(E-\l v)}_n(x)=\begin{cases}A^{(E-\l v)}(x+(n-1)\a)\cdots A^{(E-\l v)}(x), &n\ge 1;\\ Id, &n=0;\\ [A^{(E-\l v)}_{-n}(x+n\a)]^{-1}, &n\le -1.\end{cases}
$$

The relation between operator and cocycle is the following. $u\in\C^{\Z}$ is a solution of the equation $H_{\lambda,x}u=Eu$ if and only if
$$
A^{(E-\l v)}_n(x)\binom{u_{0}}{u_{-1}}=\binom{u_{n}}{u_{n-1}},\ n\in\Z.
$$
This says that $A^{(E-\l v)}_n$ generates the $n$-step transfer matrices for the operator (\ref{operator}).

The Lyapunov Exponent (LE for short), $L(E,\l)$, of this cocycle is given by
$$
L(E,\l)=\lim\limits_{n\rightarrow\infty}\frac{1}{n}\int_{\R/\Z}\ln\|A^{(E-\l v)}_n(x)\|dx=\inf_n\frac{1}{n}\int_{\R/\Z}\ln\|A^{(E-\l v)}_n(x)\|dx\geq 0.
$$
The limit exists and is equal to the infimum since $\{\int_{\R/\Z}\ln\|A^{(E-\l v)}_n(x)\| dx\}_{n\geq1}$ is a subadditive sequence. Then by Kingman's subadditive ergodic theorem, we also have for irrational $\alpha$,
$$
L(E,\l)=\lim\limits_{n\rightarrow\infty}\frac{1}{n}\ln\|A^{(E-\l v)}_n(x)\| \mbox{ for } a.e.\ x\in\R/\Z.
$$

The integrated density of states (IDS for short), $N(E)$, is given by
$$
N(E)=\lim_{n\rightarrow\infty}\frac1n \mathrm{card}\left\{(-\infty,E)\cap \Sigma(H_{n,x})\right\}\mbox{ for } a.e.\ x\in\R/\Z.
$$
Here $H_{n,x}$ denote the restriction of the operator $H_{\l,x}$ to $[0,n]$ with Dirichlet boundary condition $u_{n+1}=0$, $\Sigma(H_{n,x})$ the set of eigenvalues of $H_{n,x}$, and $\mathrm{card}$ the cardinality of a set. It is well known that the convergence is independent of Lebesgue almost every $x\in\R/\Z$. Moreover, the Lyapunov exponent $L$ and the integrated density of states $N$ are related via the following famous Thouless' Formula
\beq\label{thouless}
L(E)=\int\log|E-E'|dN(E'),
\eeq
which basically says that $L$ is the Hilbert transform of $N$ and vice versa. It is well-known that Hilbert transform preserves H\"older or some weak H\"older continuity (e.g. the continuity results we obtained in Theorem~\ref{t.continuity} in Section 1.1), see \cite{goldstein} for some detailed description. In particular, H\"older and weak H\"older continuity pass from $L$ to $N$ and vice versa.

\subsection{Statement of Main Results}
In this paper, from now on, we assume $v\in C^2(\mathbb R/\mathbb Z,\mathbb R)$ satisfy the following conditions. Assume
$\frac{dv}{dx}=0$ at exactly two points, one is minimal and the other maximal, which are denoted by $z_1$ and $z_2$. Assume that these two extremals are non-degenerate. In other words, $\frac{d^2v}{dx^2}(z_j)\neq0$ for $j=1,2$.

Fix two positive constants $\tau, \gamma$. We say $\a$ satisfying a \emph{Diophantine} condition $DC_{\tau,\gamma}$ if
$$
|\a-\frac{p}{q}|\ge \frac{\gamma}{|q|^{\tau}} \mbox{ for all } p, q\in \Z \mbox{ with } q\neq 0.
$$
It is a standard result that for any $\tau>2$,
$$
DC_\tau:=\bigcup_{\gamma>0}DC_{\tau,\gamma}
$$
is of full Lebesgue measure. Let us fix an arbitrary $\tau>2$ and consider $\a\in DC_\tau$. Then, we would like to show the following results.
\begin{theorem}\label{t.main}
  Let $\a$ and $v$ be as above. Consider the Schr\"odinger cocycle with potential $v$ and coupling constant $\l$. Let $L(E,\l)$ be the associated Lyapunov exponents. Then for all $\e>0$, there exist a $\l_0=\l_0(\a,v,\e)>0$ such that
  \beq\label{lower-bound-LE}
  L(E,\l)>(1-\e)\log\l
  \eeq
  for all $(E,\l)\in\R\times[\l_0,\infty)$.
\end{theorem}

\begin{theorem}\label{t.continuity}
Let $\a$ and $v$ be in Theorem~\ref{t.main}. Consider the Schr\"odinger cocycle with potential $v$ and coupling constant $\l$. Then there exist a $\l_1=\l_1(\a,v)>0$ such that for any fixed $\l>\l_1$, if we let $L(E)$ be the Lyapunov exponents and $N(E)$ integrated density of states (IDS), then for all $E, E'\in [\l\inf v-2,\l\sup v+2]$, it holds that
\beq\label{log-holder}
|L(E)-L(E')|+|N(E)-N(E')|<Ce^{-c(\log|E-E'|^{-1})^\sigma},
\eeq
where $c, C>0$ depends on $\a,v,\l$, and $0<\sigma<1$ on $\a$.
\end{theorem}

By the discussion following (\ref{interval-contains-spectrum}), $\R\setminus [\l\inf v-2,\l\sup v+2]$ is a subset of the resolvent set, in which $N(E)$ clearly stays constant. Due to a theorem of Johnson \cite{johnson}, for irrational frequency, $(\a,A^{(E-\l v)})$ is uniform hyperbolic ($\CU\CH$ for short) if and only if $E$ is in the resolvent set. See again \cite{zhang2} for a simple proof. Then it is standard result that $L(E)$ is smooth in the $\CU\CH$ region, see e.g. \cite[Section 2.1]{avila}. Thus, in particular, for these $\a,v,\l$ as in Theorem~\ref{t.continuity}, $L(E)$ and $N(E)$ are weak H\"older continous functions of $E\in\R$.

\subsection{Remarks on Positivity of Lyapunov exponents}
Positivity of LE for Schr\"odinger cocycle is closely related to the spectral properties of the corresponding Schr\"odinger operators. For instance, by Ishii-Pastur-Kotani \cite{ishii, pastur,kotani2}, for general bounded ergodic potential, positivity of LE for almost every energy is equivalent to the absence of absolutely continuous spectrum for almost every phase.

Moreover, positivity of LE for all energies is closely related to the Anderson Localization phenomenon. In fact, for the type of potentials considered in Theorem~\ref{t.main}, Anderson Localization has been established by Sinai and  Fr\"ohlich-Spencer-Wittwer \cite{sinai,frohlich}. Note in \cite{frohlich}, the authors also assumed that the potentials are even functions. These authors developed some inductive multi-scale procedures to get exponentially decaying eigenstates. One could extract a similar result as Theorem~\ref{t.main}, that is, $L(E,\l)>\frac12\log\l$ for all $E\in\R$,  from the proofs in \cite{sinai,frohlich}. Very recently, Bjerkl\"ov also obtain among other things a similar result, $L(E,\l)>\frac23\log\l$ for all $E\in\R$, via his approach, see \cite{bjerklov1}.

Clearly, the estimate (\ref{lower-bound-LE}) obtained in this paper is stronger. Combined with some additional arguments, it actually leads to a version of Large deviation theorem (LDT for short) that is crucial for the proof of Theorem~\ref{t.continuity}, which is the first result of this kind. See Section 1.3 for the further remarks.

On the other hand, positivity of LE for Schr\"odinger cocycles, or more generally, $\mathrm{SL}(2,\mathbb R)$ cocycles, is one of the central topics in dynamical systems. Thus, it has been extensively studied by both dynamicists and mathematical physicists. For different base dynamics, both the mechanisms and phenomena are very different. Let us list some of the related results.

For the i.i.d. potentials, Furstenberg \cite{furstenberg} showed that, among other things, LE is uniformly positive for all energies. For ergodic potentials Kotani \cite{kotani} showed that LE is positive for almost every energy if the potential is non-deterministic. Moreover, Kotani \cite{kotani} showed that ergodic potential taking finitely many values is non-deterministic if it is aperiodic, hence the corresponding LE is positive for a.e. $E$. Based on this result of Kotani, together with some new interesting ingredients, Avila-Damanik \cite{aviladamanik} showed that for generic continuous potentials defined on compact metric spaces, if the ergodic measure of the base dynamics is non-atomic, then LE is positive for almost every energy. For doubling map on the unit circle or Anosov diffeomorphism on two dimensional torus, see Chulaevsky-Spencer and Bourgain-Schlag \cite{chulaevskyspencer, bourgainschlag}. For skew shifts, see Bourgain-Goldstein-Schalg, Bourgain, and Kr\"uger \cite{bourgaingoldsteinschlag, bourgain2, bourgain3, kruger1, kruger2}. For limit periodic potentials, of which the base dynamics is minimal translations on Cantor group, see Avila \cite{avila2}. Let us remark also that Avila \cite{avila} showed that positivity of LE is a dense phenomenon on any suitable base dynamics and in any usual regularity classes.

The most intensively studied cases are quasiperiodic potentials. For real analytic potentials, the first breakthrough is due to Herman \cite{herman}. By subharmonicity, among other things, the author showed that LE is uniformly positive for trigonometric polynomials. These techniques have been further developed by Sorets-Spencer \cite{sorets} for arbitrary one-frequency nonconstant real analytic potentials and for large disorders. Same results for \emph{Diophantine} multi-frequency were established by Bourgain-Schlag \cite{bourgaingoldstein} and Goldstein-Schalg \cite{goldstein}. Bourgain \cite{bourgain3} obtained the same results for any rational independent multi-frequency. Based on new results in \cite{avila3}, Zhang \cite{zhang} gave a different proof of the \cite{sorets} results. He also applied it to a certain class of analytic Szeg\H o cocycle and obtained the uniform positivity of the associated LE.
%For almost Mathieu operators, see Bourgain and Jitomirskaya \cite{bourgainjitomirskaya}.

All results mentioned in the above paragraph do not require the \emph{Diophantine} type of conditions for frequency since one
has subharmonicity. For a class of Gevrey potentials and strong \emph{Diophantine} frequencies, see Klein \cite{klein}. Eliasson \cite{eliasson} also gets some related results for a certain class of Gevrey potentials and for some strong \emph{Diophantine} frequencies. For smooth potentials, it seems a complicated induction and some \emph{Diophantine} type of conditions are necessary to take care of the small divisor type of problems. Other than works in \cite{frohlich, sinai}, some recent works can be found in \cite{bjerklov, chan} for more general smooth potentials. In \cite{bjerklov}, the author used techniques that are close in spirit to \cite{benedickscarleson} and a positive measure of frequencies and energies are excluded. In \cite{chan}, the author used multi-scale analysis, and uniform positivity of LE for some $C^3$ potentials is obtained by excluding a positive measure of frequencies and by varying the potentials in some typical way.

The method used in this paper is of purely dynamical systems, which is from Young \cite{young} and close also in spirit to Benedicks-Carleson \cite{benedickscarleson}. The techniques in \cite{young} have been applied to Schr\"odinger cocycles by Zhang \cite{zhang}, and some positivity results for general smooth potentials and for fixed \emph{Brjuno} frequencies have been obtained. Roughly speaking, these techniques based on some detailed analysis of asymptotic stable and unstable directions. The key idea is to classify the ways that they intersect with each other. Then, one need to develop some induction schemes to show that these ways are all the possibilities of intersection between them.

In those cases considered in \cite{young, zhang}, one again needs to exclude a positive measure of energies to get the nonresonance condition. Then it is showed that under the nonresonance condition, the $n$-step stable and $n$-step unstable directions always intersect in a transversal way, which makes the induction easier. And in this case, a $C^1$ type of estimates of the asymptotic stable and unstable direction is sufficient. The nonresonance condition also makes sure that for the survived parameters, the dynamical systems are \emph{nonuniformly hyperbolic} ($\CN\CU\CH$ for short). This is due to the fact that the intersection between asymptotic stable and unstable directions persist in larger and larger time scale, which eventually implies the intersection of stable and unstable directions, hence, $\CN\CU\CH$. Back to the model in Theorem~\ref{t.main}, while the statement of Theorem~\ref{t.main} does not necessarily distinguish energies between the spectrum and the resolvent set, we actually have the following Corollary of \cite[Theorem B$'$]{zhang}.

\begin{corollary}\label{c.large-spectrum}
Let $\a$ and $v$ be as in Theorem~\ref{t.main}. Then for each $\l>\l_0$, there exists a $\Omega_{\a,\l v}\subset\Sigma_{\a,\l v}$ such that
$$
\lim_{\l\rightarrow\infty}\frac{\mathrm{Leb}(\Omega_{\a,\l v})}{\l(\sup v-\inf v)}=1,
$$
and for each $E\in\Omega_{\a,\l v}$, there exists some $x\in\R/\Z$ such that the eigenvalue equation $H_{\a,\l v,x}u=Eu$ admits some exponentially decaying eigenvectors.
\end{corollary}
\begin{remark}
The last statement of Corollary~\ref{c.large-spectrum} implies that $|u_n|<Ce^{-L|n|}$ for all $n\in\Z$ for some constant $C,L>0$, which is kind of Anderson Localization phenomenon. Also, in Corollary~\ref{c.large-spectrum}, we can actually relax the \emph{Diophantine} condition to the \emph{Brjuno} condition, see \cite[Theorem B$'$]{zhang}.
\end{remark}
In this paper, we will not exclude any parameter. Thus, the main difficulty of the cases considered in this paper is the occurrence of `resonance'. This leads to bifurcation of the way $n$-step stable direction intersecting with $n$-step unstable direction: our analysis shows that `resonance' leads to some tangential way of intersection or even separation of $n$-step stable and unstable directions, which leads to \emph{uniformly hyperbolic} ($\CU\CH$) systems, see figure (d)--(f). In fact, to start with, one encounters with tangential intersections of first step stable and unstable directions. Thus one needs some nondegenerate conditions to get start. And a new induction scheme that includes both `nonresonance' and `resonance' cases needs to be introduced. Moreover, to deal with the tangential type of intersection, a $C^2$ type of estimate of the asymptotic stable and unstable directions is required.

\subsection{Remarks on Regularity of Lyapunov exponents}
Much work has been devoted to the regularity properties of Lyapunov exponents (LE) and integrated density of states (IDS). By the discussion following (\ref{thouless}), we focus on the regularity of LE here.

On the regularity of LE for real analytic quasi-periodic potentials, a series of positive results have been obtained in the 2000s. It starts with the work of Goldstein-Schlag \cite{goldstein} where they obtained some sharp version of large deviation theorems (LDT) for real analytic potentials with strong \emph{Diophantine} frequency, developed a powerful tool, the Avalanche Principle, and proved H\"older or weak H\"older continuity of $L(E)$ in the regime of positive LE. Notice that LDT for real analytic potentials with \emph{Diophantine} frequency was first established in \cite{bourgaingoldstein} in order to get Anderson Localization.  This also illustrates the power and importance of LDT.

Avalanche Principle involves only long finite products of matrices, see Section~\ref{sec-continuity}.2. Thus, the key to apply the method in \cite{goldstein} to other cases is to establish LDT. For other type of base dynamics, Bourgain-Goldstein-Schalg \cite{bourgaingoldsteinschlag} obtained the results for skewshift base dynamics, Bourgain-Schlag \cite{bourgainschlag} for doubling map and Anosov diffeomorphism. For lower regularity case, \cite{klein} got the results for a class of Gevrey potentials. These results concern continuity with respect to energies. For wider class of cocycle maps, Jitomirskaya-Koslover-Schulteis \cite{JKS1} get the continuity of LE for a class of analytic quasiperiodic $\mathrm{M}(2, \C)$ cocycles which is applicable to general quasi-periodic Jacobi matrices or orthogonal polynomials on the unit circle in various parameters. Jitomirskaya-Marx \cite{JM1} later extended it to all (including singular) $\mathrm{M}(2,\C)$ cocycles. H\"older continuity for $\mathrm{GL}(d,\C)$ cocycles, $d\ge 2$,  was recently obtained in Schlag \cite{schlag} and Duarte-Klein \cite{duarteklein}. All the results stated above, except strongly mixing cases, require \emph{Diophantine} condition.

An arithmetic version of large deviations and an arithmetic inductive scheme were developed in \cite{bourgainjitomirskaya} allowing to obtain joint continuity of LE for $\mathrm{SL}(2,\C)$ cocycles, in frequency and cocycle map, at any irrational frequencies. This result has been crucial in later proofs of the Ten Martini problem \cite{AJ}, Avila's global theory of one-frequency cocycles \cite{avila3, avila3.2}, and other important developments. It was extended to multi-frequency case by Bourgain \cite{bourgain3} and to general $\mathrm{M}(2,\C)$ case in \cite{JM2}. More recently, a completely different proof, not using LDT or Avalanche principle, and extending to the general $\mathrm{M}(d,\C)$, $d\ge 2$, case was developed in Avila-Jitomirskaya-Sadel \cite{avilajitosadel}. All these results however rely heavily on analyticity of the cocycle map.

Thus, Theorem~\ref{t.continuity} in this paper is striking in the sense that it provides the first positive result on the continuity of LE and weak H\"older continuity of IDS on $E$ (note log-H\"older continuity of IDS on $E$ holds for general ergodic bounded potentials, see \cite{craigsimon}) for $C^r$, $r\le\infty$, quasi-periodic potentials. More concretely, surprisingly, it turns out that some version of LDT follows naturally from our induction scheme, see Section~\ref{sec-continuity} for details. Thus, combined with the Avalanche Principle, Theorem \ref{t.continuity} follows essentially from the same argument as in \cite{goldstein}.

We remark that Theorem \ref{t.continuity} has an analog as in \cite{JKS1}, that is, we can prove continuity of LE with respect to the $C^2$ $cos$-type of potentials.

For other related results, Avila-Krikorian \cite{avilakrikorian} recently studied so-called monotonic cocycles which are a class of smooth or analytic cocycles non-homotopic to constant. They proved that the LE is smooth or even analytic, respectively. In comparison, the regularity of LE cannot be better (as far as the modulus of continuity is concerned) than $1/2$-H\"older continuous for cocycles homotopic to constant which automatically includes the category of Schr\"odinger cocycles. However, Avila \cite{avila3} recently showed that if one stratifies the energies or some real analytic family of real analytic potentials in some natural way, then the LE is in fact real analytic.

There are many negative results on the positivity and continuity of LE for non-analytic cases. It is well known that in $C^0$-topology, discontinuity of LE holds true at every non-uniformly hyperbolic cocycle, see \cite{Furman,Knill,Th}. Moreover, motivated by Ma$\tilde{n}\acute{e}$ \cite{Mane1,Mane2}, Bochi \cite{Bochi1,Bochi2} proved that with an ergodic base system, any non-uniformly hyperbolic $\mathrm{SL}(2,\R)$-cocycle can be approximated by cocycles with zero LE in the $C^0$ topology.

Based also on the the method of Young\cite{young}, Wang-You \cite{youwang1} constructed examples to show that LE can be discontinuous even in the space of $C^{\infty}$ Schr\"odinger cocycles. Recently, Wang-You \cite{youwang2} has improved the result in \cite{youwang1} by showing that in $C^r$ topology, $1\le r\le\infty$, there exists Schr\"odinger cocycles with positive LE that can be approximated by ones with zero LE. The example in \cite{youwang2} also showed that the nondegenerate condition for the potential in Theorem \ref{t.main} and \ref{t.continuity} is necessary for positivity and continuity of LE. Jitomirskaya-Marx \cite{JM2} constructed examples showing that LE of $\mathrm{M}(2,\C)$ cocycles is discontinuous in $C^\infty$ topology.

Finally, let us remark that the continuity of LE for Schr\"odinger cocycles is also expected to play important roles in  studying Cantor spectrum, typical localization length, phase transition, etc, for quasi-periodic Schr\"odinger operators.

\subsection{Generalization and Further Developments}
Though Theorem~\ref{t.main} is our primary interest, our method is not restricted to Schr\"odinger cocycles. What we proved is actually a more general version concerning smooth quasiperiodic $\mathrm{SL}(2,\R)$ cocycles, see Corollary~\ref{c.general} of Appendix Section~\ref{application}. In particular, we obtain the following corollary of Corollary~\ref{c.general} and \cite[Theorem B$'$]{zhang}. We say $v\in C^2(\R/\Z,\R)$ has a unique maximal point if the set $\{x: v(x)=\max_{y\in\R/\Z} v(y)\}$ consists of a single point (for simplicity, we state only the maximal point case. The minimal point case can be stated similarly). Then we have the following corollary.

\begin{corollary}\label{c.partial}
Let $\a$ be as in Theorem~\ref{t.main}. Assume $v\in C^2(\R/\Z,\R)$ has a unique nondegenerate maximal point which is denoted by $x_0$. Then there exists a $r>0$ such that for each $\e>0$, there exists a $\l_0=\l_0(\a,v,\e,r)$ such that for all $(E,\l)\in \l [v(x_0)-r,v(x_0)+r]\times (\l_0,\infty)$,
$$
L(E,\l)>(1-\e)\log\l.
$$
Moreover, for any fixed $\l>\l_0$ and for all $E,E'\in \l [v(x_0)-r,v(x_0)+r]$, it holds that
$$
|L(E)-L(E')|+|N(E)-N(E')|<Ce^{-c(\log|E-E'|^{-1})^\sigma},
$$
where $c,C>0$ depends on $v,\a,\e,r,\l$, and $0<\sigma<1$ on $\a$. Finally, we have
$$
\lim_{\l\rightarrow\infty}\frac{1}{\lambda r}\mathrm{Leb}\left\{\Sigma_{\alpha,\lambda v}\cap \l [v(x_0)-r,v(x_0)]\right\}=1.
$$
\end{corollary}

In other words, the LE is positive and continuous for all energies around the unique non-degenerate extremals of potentials for large disorders. This corollary says that, in some sense, the positivity of LE is a local property with respect to the initial `critical interval' of the potential. Corollary~\ref{c.general} can also be applied to a certain class of quasiperiodic Szeg\H o cocycles, see Corollary~\ref{c.szego} of Section~\ref{application}. For details and other applications, see Section~\ref{application}.

To sum up, we believe that our method may have the following further development. Firstly, although the computation will be much more complicated, it is possible that our techniques can be used to analyze more general smooth potentials. For instance, instead of $C^2$ estimate of  Lemma~\ref{l.nstep-deri-estimate}, we may need $C^r$ for $r<\infty$. Moreover, we may need to deal with the new types of resonance, e.g. resonance between the type $\I$ and type $\II$ functions of Definition~\ref{d.type-of-functions}.

Secondly, since our method is based on a detailed analysis of asymptotic stable and unstable directions, it has the advantage in showing the occurrence of $\CU\CH$, see, for example, Remark~\ref{r.uh-stopping}. Hence, it is possible to show Cantor spectrum for the type of potentials in Theorem~\ref{t.main}, or even for more possible potentials. We will come back to this topic elsewhere.

Thirdly, it is possible to relax \emph{Diophantine} condition to \emph{Brjuno} or even weak \emph{Liouville} conditions in Theorem~\ref{t.main} and \ref{t.continuity}. Moreover, it is also possible to improve the index $\sigma$ in (\ref{log-holder}) to $1$ which is nothing other than the H\"older continuity. We do not pursue these goals here in order to keep this paper to a reasonable length.

Finally, the idea of analyzing the asymptotic stable and unstable directions is probably not restricted to one-frequency quasiperiodic case. These techniques are also considered to be promising in \cite{avila3, avilakrikorian}.

\subsection{Structure of the Paper and Acknowledgements}
The structure of the remaining part of this papers is as follows. In Section~\ref{preliminaries}, we state a series of technical lemmas. We first reduce the Schr\"odinger cocycles to its polar decomposition form so that we can get started with our induction. Then we state the series of Lemmas that will be used to control the derivatives of asymptotic stable and unstable directions and the norms the iteration of cocycles, and concatenation of sequence of matrix-maps. Then, we classify the types of functions that will be used to describe all possible ways the $n$-step stable direction intersecting with $n$-step unstable direction. Finally, we state and prove a easy corollary which actually builds the bridge of concatenation of sequence of matrix-maps and our classification of the intersection between asymptotic stable and unstable directions . The proof of our induction and Theorem~\ref{t.main} and \ref{t.continuity} are just some repeated applications of these lemmas.

In Section~\ref{getting-started}, we will get started with our induction. We will start with step $1$ and move one step forward to step $2$. So we get to know all possible cases that will occur in our induction. In Section~\ref{induction}, we state and prove our induction. In Section~\ref{sec-positivity}, we prove theorem \ref{t.main} by induction. In Section~\ref{sec-continuity}, we first show a version of LDT. Then, we prove Theorem~\ref{t.continuity}. In appendix Section~\ref{a}, we prove Lemma~\ref{l.reduce-form}--\ref{l.s-deri-resonance} that are given in Section~\ref{preliminaries}. In Section~\ref{application}, we state a more general version of Theorem~\ref{t.main} and \ref{t.continuity}, and give some applications.

\vskip0.4cm
\noindent{\bf\textit{Acknowledgments.}} Z.Z. would like to thank his advisors Artur Avila and Amie Wilkinson for suggesting the project of uniformly positive Lyapunov exponents for smooth potentials, for some helpful discussions, and for their encouragement and continuous support. He also would like to thank Vadim Kaloshin for suggesting this project, for some helpful discussion, and for showing a note joint with Anton Gorodetski that gives some helpful hints. We are grateful to Michael Goldstein for some helpful discussion and for showing us a manuscript of him which gives us some positive hints. It is our pleasure to thank Jiangong You for some helpful discussions. We also owe our thanks to Svetlana Jitomirskaya for detailed comments and suggestions.

\section{{\bf Preliminaries}}\label{preliminaries}

From now on, if not stated otherwise, let $C,\ c$ be some universal positive constants depending only on $v$ and $\a$, where $C$ is large and $c$ small. Let $\frac{p_s}{q_s}$ be the $s$th continued fraction approximants of frequency $\alpha$. Then it is a standard result that $\a\in DC_\tau$ if and only if there is some $c>0$ such that $q_{s+1}<cq_{s}^{\tau-1}$ for all $s\geq1$. We will sometimes use this equivalent condition. Finally, for two positive real number $a,b>0$, by $a \gg b $ or $b\ll a$, we mean that $a$ is sufficiently larger than $b$.

For $\theta\in \R/(2\pi \Z)$, let
$$
R_\t=\begin{pmatrix}\cos\t&-\sin\t\\ \sin\t&\cos\t\end{pmatrix}\in\mathrm{SO}(2,\mathbb R).
$$
 Then, instead of proving Theorem~\ref{t.main} directly, we will use the following equivalent  form of the cocycle map (\ref{schrodinger-cocycle-map}).
\begin{lemma}\label{l.reduce-form}
 Let $\CI\subset\R$ be any compact interval. For $x\in\mathbb R/\mathbb Z$ and $t\in\CI$, define  the following cocycles map
\beq\label{polar-decom}
A(x)=\Lambda(x)\cdot R_{\phi(x,t)}:=\begin{pmatrix}\l(x)&0\\0&\l^{-1}(x)\end{pmatrix}\cdot\begin{pmatrix}\frac{t-v(x)}{\sqrt{(t-v(x))^2+1}}&\frac{-1}{\sqrt{(t-v(x))^2+1}}\\
\frac{1}{\sqrt{(t-v(x))^2+1}}&\frac{t-v(x)}{\sqrt{(t-v(x))^2+1}}\end{pmatrix},
\eeq
where $\cot\phi(x,t)=t-v(x)$. Assume
\beq\label{norm1-deri-control}
\l(x)>\l,\ \left|\frac{d^m\l(x)}{dx^m}\right|<C\l,\ m=1,2.
\eeq
Then to prove Theorem~\ref{t.main}, it is enough to prove the corresponding results for (\ref{polar-decom}).
\end{lemma}
The proof of Lemma~\ref{l.reduce-form} will be given in Section~\ref{a.1}. From now on, $A$ will denote the cocycle map in (\ref{polar-decom}).

The following notations will be used throughout this paper. Let $B(x,r)\subset\R/\Z$ be the ball centered around $x\in\R/\Z$ with radius $r$. For a connected interval $I\subset\R/\Z$ and constant $0<a\le1$, let $aI$ be the subinterval of $I$ with the same center and whose length is $a|I|$. Define the map
$$
s:\mathrm{SL}(2,\R)\rightarrow\R\PP^1=\R/(\pi\Z)
$$
so that $s(A)$ is the most contraction direction of $A\in\mathrm{SL}(2,\R)$. Let $\hat s(A)\in s(A)$ be an unit vector. Thus, $\|A\cdot\hat s(A)\|=\|A\|^{-1}$. Abusing the notation a little, let
$$
u:\mathrm{SL}(2,\R)\rightarrow\R\PP^1=\R/(\pi\Z)
$$
be that $u(A)=s(A^{-1})$. Then for $A\in\mathrm{SL}(2,\R)$, it is clear that
\beq\label{polar-form}
A=R_u\cdot\begin{pmatrix}\|A\| &0\\ 0 &\|A\|^{-1}\end{pmatrix}\cdot R_{\frac\pi2-s},
\eeq
where $s,u\in [0,2\pi)$ are some suitable choices of angles correspond to the directions $s(A),u(A)\in \R/(\pi\Z)$. It can also be deduced from the polar decomposition procedure of $A$, see Section~\ref{a.1}.

The following series of Lemmas will be quite involved in the our induction scheme. Basically, under suitable conditions, they deals with the concatenation of sequence of $\mathrm{SL}(2,\R)$ matrices maps that are defined on small intervals of $\mathbb R/\mathbb Z$. To get exponential growth of norm of the products for larger and larger time scale, on one hand we need to control the geometrical properties of the forward and backward most contraction directions. On the other hand, we also need to control the derivatives of the norms with respect to the phase. And we need to deal with both resonance and nonresonance cases. Roughly speaking, if $\|E_i\|\gg 1$ for $i=1, 2$ and we want to concatenate $E_2\cdot E_1$, then in nonresonance case we have
$$
|s(E_2)-u(E_1)|^{-1}\ll \min\{\|E_1\|,\ \|E_2\|\}.
$$
Otherwise, we are in resonance case. Proofs of Lemmas \ref{l.ess-change}--\ref{l.s-deri-resonance} can be found in Section~\ref{b} and \ref{c}.

Let us start with a lemma that reduce the estimate of the most contraction directions in case of concatenation of two matrices to the estimate of some simple functions.
\begin{lemma}\label{l.ess-change}
Consider the function $s(x)=s[E(x)], u(x)=u[E(x)]:I\rightarrow\mathbb R\mathbb P^1$, where $I\subset\mathbb R\mathbb P^1$ is a connected interval and
$$
E(x):=\begin{pmatrix}e_2(x)&0\\0&e_2^{-1}(x)\end{pmatrix}R_{\t(x)}\begin{pmatrix}e_1(x)&0\\0&e_1^{-1}(x)\end{pmatrix}.
$$

Let $f_1(x)=\frac{1}{2}(e_1^2\cot\t+e_1^2e_2^{-4}\tan\t)$ and $f_2(x)=\frac{1}{2}(e_2^2\cot\t+e_2^2e_1^{-4}\tan\t)$. Then for each $m=0,1,2$ and each $x\in I$, we have the following.
\begin{itemize}
\item If $e_2(x)>e_1(x)\gg1$, then we have
\beq\label{s-e2-great-e1}
c<\left|\frac{d^m s}{dx^m}/\frac{d^m \tan^{-1}(e_1^2\cot\t)}{dx^m}\right|,\ \left|\frac{d^m u}{dx^m}/\frac{d^m\cot^{-1}(\sqrt{f_2^2+1}+f_2)}{dx^m}\right|<C.
\eeq
\item If $e_1(x)>e_2(x)\gg1$, then we have
\beq\label{s-e1-great-e2}
c<\left|\frac{d^m s}{dx^m}/\frac{d^m\tan^{-1}(\sqrt{f_1^2+1}+f_1)}{dx^m}\right|,\ \left|\frac{d^m u}{dx^m}/\frac{d^m \cot^{-1}(e_2^2\cot\t)}{dx^m}\right|<C.
\eeq
\item If $e_1(x)=e_2(x)\gg1$, then we have
\beq\label{s-e1-equal-e2}
c<\left|\frac{d^m s}{dx^m}/\frac{d^m\tan^{-1}(\cot\t\sqrt{e_1^4+\tan^2\t})}{dx^m}\right|<C.
\eeq
\end{itemize}
Moreover in the above we can replace $(s, \tan^{-1})$ by $(u, \cot^{-1})$ to get the estimates for $u$.
\end{lemma}

Lemma \ref{l.ess-change} contains information for both resonance and nonresonance case. Let us first consider the nonresonance case. We again start with the concatenation of two matrices.

\begin{lemma}\label{l.1step-deri-estimate}
Let $E(x)$, $e_0=\min\{e_1,e_2\}$ be as in the Lemma~\ref{l.ess-change} and $e_3(x)=\|E(x)\|$. Assume $0<\eta\ll 1$. Suppose that for all $x\in I$, $j,m=1,2$, we have
$$
\left|\frac{d^me_j}{dx^m}(x)\right|< Ce_j^{1+m\eta};\ \left|\frac{d^m\t}{dx^m}\right|<Ce_0^{\eta},\ |\t-\frac\pi2|^{-1}>ce_0^{-\eta}.
$$
Then we have
\beq\label{1step-deri-estimate4}
\left\|s-\frac\pi2\right\|_{C^2}<Ce_1^{-(2-5\eta)},\ \left\|u\right\|_{C^2}<Ce_2^{-(2-5\eta)};
\eeq
\beq\label{1step-deri-estimate5}
\left|\frac{d^me_3}{dx^m}(x)\right|<Ce_3^{1+m\eta}\mbox{ for all } x\in I \mbox{ and }m=1,2.
\eeq
\end{lemma}

Then we move Lemma~\ref{l.1step-deri-estimate} forward to the concatenation of $n$ matrices in nonresonance case for some big, which is as follows.

Consider a sequence of map
$$
E^{(\ell)}\in C^2(I,\mathrm{SL}(2,\mathbb R)),\ 0\le \ell\le n-1.
$$
Let $s^{(\ell)}=s(E^{(\ell)})$, $u^{(\ell)}=u(E^{(\ell)})$, $\l_\ell=\|E^{(\ell)}\|$, and $\Lambda^{(\ell)}=\begin{pmatrix}\l_\ell&0\\0&\l_\ell^{-1}\end{pmatrix}$. By (\ref{polar-form}), we clearly have
$$
E^{(\ell)}=R_{u^{(\ell)}}\Lambda^{(\ell)}R_{\frac\pi2-s^{(\ell)}}.
$$
Set $E_k(x)=E^{(k-1)}(x)\cdots E^{(0)}(x)$, $1\le k\le n$. Let
$$
s_{k}=s(E_{k}),\ u_{k}=u(E_{k}),\ l_{k}=\|E_{k}\| \mbox{ and } L_{k}=\begin{pmatrix}l_{k}&0\\0&\l_{k}^{-1}\end{pmatrix}.
$$
Again from (\ref{polar-form}), we have
$$
E_{k}=R_{u_{k}}L_{k}R_{\frac\pi2-s_{k}}.
$$

Then the following lemma will be play the key role in dealing with the nonresonance case.

\begin{lemma}\label{l.nstep-deri-estimate}
Let $E^{(\ell)}$ and $E_k$ be as above. Let $0<\eta\ll 1\ll \l':=\min_{0\le\ell\le n-1}\{\l_\ell\}$. We further assume that $n<C\l'^{\frac12}$, and for any $x\in I$, $m=1,2$ and $0\le\ell\le n-1,$
$$
\left|\frac{d^m\l_\ell}{dx^m}(x)\right|< C\l_\ell^{1+m\eta};\ \left|\frac{d^ms^{(\ell)}}{dx^m}\right|,\ \left|\frac{d^mu^{(\ell)}}{dx^m}\right|<C\l'^{\eta},\ |s^{(\ell)}-u^{(\ell-1)}|>c\l'^{-\eta}.
$$
Then we have that

\beq\label{nstep-deri-estimate1}
\left\|u^{(n-1)}-u_{n}\right\|_{C^2}<C\l_{n-1}^{-(2-5\eta)},\ \left\|s^{(0)}-s_n\right\|_{C^2}<C\l_0^{-(2-5\eta)};
\eeq
\beq\label{nstep-deri-estimate2}
\left|\frac{d^ml_n}{dx^m}(x)\right|< Cl_n^{1+m\eta},\ m=1,2;
\eeq
\beq\label{nstep-deri-estimate3}
l_n>\left(\prod^{n-1}_{\ell=0}\l_\ell\right)^{1-\eta}.
\eeq
\end{lemma}

\begin{remark}\label{C0-or-C1}
By the proof of Lemma~\ref{l.1step-deri-estimate} and \ref{l.nstep-deri-estimate} in Section~\ref{b}, it is not difficult to see that in order to get (\ref{nstep-deri-estimate3}) and a $C^0$ version of (\ref{nstep-deri-estimate1}), one only needs to assume that the norm of the sequence of matrices are large and $|s^{(\ell)}-u^{(\ell-1)}|^{-1}$ is not large with respect to norms. If in addition, one needs $C^1$ version of (\ref{nstep-deri-estimate1}), then one just needs to add the corresponding $C^1$ control of the norm maps, $s$ and $u$. In particular, the $C^1$ version of Lemma~\ref{l.nstep-deri-estimate} is essentially the same with \cite[Lemma 3]{young}.
\end{remark}

By Lemma~\ref{l.nstep-deri-estimate}, we will see that we can reduce the model to the concatenation of two matrices to deal with the resonance case. In other words, we only need to consider $E_2\cdot E_1$. However, $s(E_2)-u(E_{1})$ may pass through $0$. We will show that in the resonance case, with the help of Lemma~\ref{l.ess-change}, some good estimate still holds true if $\|E_2\|\gg \|E_{1}\|$ or $\|E_1\|\gg \|E_{2}\|$. We first estimate the derivatives of the norm functions, and give the upper-bound of the most contraction direction.
\begin{lemma}\label{l.norm-deri}
Let $E(x)=E_2(x)E_1(x)$. Define $e_3(x)=\|E(x)\|$ and $e_0=\min\{e_1,e_2\}$. Assume $0<\eta\ll 1\ll e_0$ and $0<\b\ll1$.  Suppose $e_1\le e_2^\b$ or $e_2\le e_1^\b$,  and for $\t(x)=s[E_2(x)]-u[E_1(x)]$ and each $x\in I$, $j,m=1,2$, it holds that
$$
\left|\frac{d^me_j}{dx^m}(x)\right|< Ce_j^{1+m\eta};\ \left|\frac{d^m\t}{dx^m}\right|<Ce_0^{\eta}.
$$
Then we have for $m=1,2$,
\beq\label{norm-deri2}
\left|\frac{d^ms[E(x)]}{dx^m}\right|<Ce_1^{4+2\eta},\ \left|\frac{d^mu[E(x)]}{dx^m}\right|<Ce_3^{-\frac32} \mbox{ if } e_1\le e_2^\b;
\eeq
\beq\label{norm-deri3}
\left|\frac{d^mu[E(x)]}{dx^m}\right|<Ce_2^{4+2\eta},\ \left|\frac{d^ms[E(x)]}{dx^m}\right|<Ce_3^{-\frac32} \mbox{ if } e_2\le e_1^\b;
\eeq
\beq\label{norm-deri1}
\left|\frac{d^me_3}{dx^m}(x)\right|<Ce_3^{1+m\eta+2m\eta\b}.
\eeq
\end{lemma}

However, in the resonance case, we also need a $C^2$ lower bound near $C^1$ degenerate points. Instead of estimating the derivatives of the most contraction directions directly, let us consider the following three types of functions, which basically classify all the possible ways that the $n$-step stable directions intersecting with unstable directions. In particular, the type $\III$ functions are going to describe the resonance case, from which we also have a bifurcation procedure.

Let $I\subset\mathbb R/\mathbb Z$ be a connected interval. Without loss of generality, let $I=B(0,r)$ and $l$ satisfy $l\gg r^{-1}\gg 1$. For the given $I$ and $l$, we define the following types of functions.
\begin{defi}\label{d.type-of-functions}
Let $I$ and $l$ be as above. Let $f\in C^2(I,\mathbb R\mathbb P^1)$. Then
\vskip0.2cm
\noindent {\bf $f$ is of type $\mathrm{I}$} if we have the following. $\|f\|_{C^2}<C$ and $f(x)=0$ has only one solution, say $x_0$, which is contained in $\frac{I}{3}$; $\frac{df}{dx}=0$ has at most one solution on $I$; $|\frac{df}{dx}|>r^2$ for all $x\in B(x_0,\frac r2)$; Let $J\subset I$ be the subinterval such that $\frac{df}{dx}(J)\cdot\frac{df}{dx}(x_0)\le 0$, then $|f(x)|>cr^3$ for all $x\in J$. Let $\mathrm{I}_{+}$ denotes the case $\frac{df}{dx}(x_0)> 0$ and $\mathrm{I}_{-}$ for $\frac{df}{dx}(x_0)<0$. See figure $(a)$.

\vskip0.2cm
\noindent {\bf $f$ is of type $\mathrm{II}$} if we have the following. $\|f\|_{C^2}<C$ and $f(x)=0$ has at most two solutions; $\frac{df}{dx}(x)=0$ has one solution; All of these solutions are contained in $\frac{I}{2}$; $f(x)=0$ has one solution if and only if it is the $x$ such that $\frac{df}{dx}(x)=0$; Finally, $\left|\frac{d^2f}{dx^2}\right|>c$ whenever $|\frac{df}{dx}|<r^2$. See figure $(b)$.

\vskip0.2cm
\noindent {\bf $f$ is of type $\mathrm{III}$} if
\beq\label{type3}
f=\tan^{-1}(l^2[\tan f_1(x)])-\frac\pi2+f_2,
\eeq
where either $f_1$ is of type $\mathrm{I}_+$ and $f_2$ of type $\mathrm{I}_-$, or $f_1$ is of type $\mathrm{I}_-$ and $f_2$ of type $\mathrm{I}_+$. See figure $(c)$.

\end{defi}
A simple case of type $\mathrm{I}$ function is that $\left|\frac{df}{dx}(x)\right|>r^2$ for all $x\in I$. The form of type $\III$ function in (\ref{type3}) actually follows from the first estimate of (\ref{s-e2-great-e1}) and the second estimate of (\ref{s-e1-great-e2}).

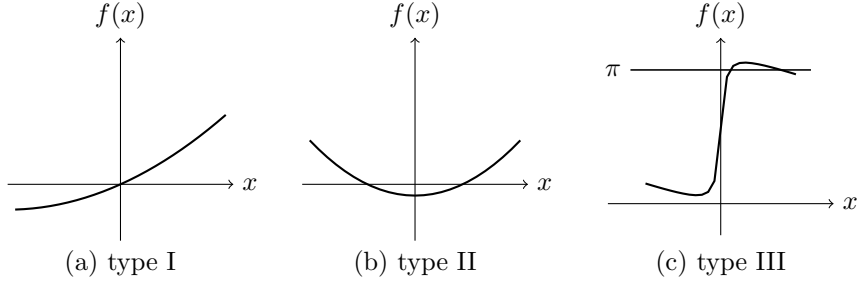
\begin{center}
\begin{tikzpicture}[yscale=1.5]
\draw [->] (-1.5,0) -- (1.5,0);
\draw [->] (0,-0.5) -- (0,1.3);
\draw [thick,domain=-1.4:1.4] plot (\x, {0.1*\x*\x+0.3*\x});
\node [above] at (0,1.3) {$f(x)$};
\node [right] at (1.5,0) {$x$};
\node[align=left, below] at (0,-0.5)%
{(a) type $\I$};
\end{tikzpicture}\ \ \ \
\begin{tikzpicture}[yscale=1.5]
\draw [->] (-1.5,0) -- (1.5,0);
\draw [->] (0,-0.5) -- (0,1.3);
\draw [thick,domain=-1.4:1.4] plot (\x, {0.25*\x*\x-0.1});
\node [above] at (0,1.3) {$f(x)$};
\node [right] at (1.5,0) {$x$};
\node[align=left, below] at (0,-0.5)%
{(b) type $\II$};
\end{tikzpicture}\ \ \ \
\begin{tikzpicture}[yscale=0.7]
\draw [->] (-1.5,0.6) -- (1.5,0.6);
\draw [->] (0,0) -- (0,3.75);
\draw [thick,domain=-1:1] plot (\x, {1/180*pi*atan(20*tan(\x r))+0.5*pi-0.5*(\x-0.9)});
\draw [semithick,domain=-1.2:1.2] plot (\x, {0*\x+pi});
\node [above] at (0,3.75) {$f(x)$};
\node [right] at (1.5,0.6) {$x$};
\node [left] at (-1.2,3.14) {$\pi$};
\node[align=left, below] at (0,-0.1)%
{(c) type $\III$};
\end{tikzpicture}
\end{center}

The following lemma for $f$ of type $\mathrm{III}$ actually plays the key role for the lower-bound estimate of the geometric properties of most contraction directions in resonance case. Without loss of generality, let $f$ be as in (\ref{type3}) with $f_1$ be type $\mathrm{I}_+$ and $f_2$ type $\mathrm{I}_-$ throughout this section. We may further assume that $f_1(0)=0$ and $f_2(d)=0$ with $0\le d\le \frac 23r$. Let
$$
X=\{x\in I: \mathbb R\mathbb P^1\ni |f(x)|=\min_{y\in I}|f(y)|\}.
$$
Then it is easy to see that $X$ contains at most two points, say $X=\{x_1,x_2\}$ with $x_1\le x_2$. Then we have the following lemma.
\begin{lemma}\label{l.s-deri-resonance}
Let $f$ be of type $\mathrm{III}$. Let $\eta_j$ be constants satisfying $r^{2}\le \eta_j\le r^{-2}$, $0\le j\le 4$. Then
\beq\label{type3-zero}
|x_1|<Cl^{-\frac 34},\ |x_2-d|<Cl^{-\frac 34}.
\eeq
In particular, if $f(x_1)=f(x_2)=0$, then
\beq\label{type3-zero-order}
0<x_1\le x_2<d;
\eeq
if $f(x_1)=f(x_2)\neq 0$, then
\beq\label{type3-zero-1}
x_1=x_2.
\eeq
Moreover there exist two distinct points $x_3,\ x_4\in B(x_1,\eta_0l^{-1})$ such that $\frac{df}{dx}(x_j)=0$ for $j=3,4$, and $x_3$ is
 a local minimum with \beq\label{type3-nonzero-minimum}
f(x_3)>\eta_1l^{-1}-\pi.
\eeq
 Moreover, we have the following.

\vskip0.2cm
If $d\ge \frac r3$, then we have
\beq\label{type3-1}
|f(x)|>cr^3,\ x\notin B(x_1,Cl^{-\frac14})\cup B(x_2,\frac r4);\ \|f-f_2\|_{C^1}<Cl^{-\frac32},\ x\in B(x_2,\frac r4).
\eeq

If $d<\frac r3$, then we have
\beq\label{type3-nondegeneracy}
\left|\frac{d^2f}{dx^2}(x)\right|>c \mbox{ whenever } \left|\frac{df}{dx}(x)\right|\le r^2,\ \forall x\in B(X,\frac r6)
\eeq
and $|f(x)|>cr^3$ for all $x\notin B(X,\frac r6)$.\\

Finally, we have the following bifurcation as $d$ varies. There is a $d_0=\eta_2 l^{-1}$ such that:
\begin{itemize}
\item if $d>d_0$, then $f(x)=0$ has two solutions. See figure $(d)$;
\item if $d=d_0$, then $f(x)=0$ has exactly one tangential solution. In other words, $x_1=x_2=x_4$ and $f(x_4)=0$. See figure $(e)$;
\item if $0\le d<d_0$, then $f(x)\neq 0$ for all $x\in I$. See figure $(f)$. Moreover, we have
      $$
      \min_{x\in I}|f(x)|=-\eta_3l^{-1}+\eta_4d.
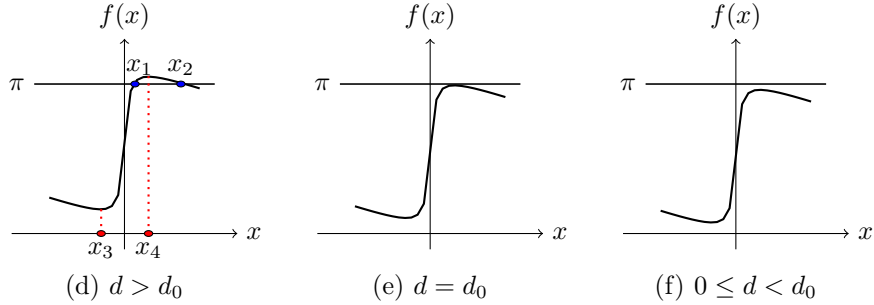
      $$
\end{itemize}
\end{lemma}

\begin{center}
\begin{tikzpicture}[yscale=0.7]
\draw [->] (-1.5,0.3) -- (1.5,0.3);
\draw [->] (0,0) -- (0,4);
\draw [thick,domain=-1:1] plot (\x, {1/180*pi*atan(20*tan(\x r))+0.5*pi-0.5*(\x-0.9)});
\draw [semithick,domain=-1.2:1.2] plot (\x, {0*\x+pi});
\draw [dotted, thick, red] (-0.31,0.3) -- (-0.31,0.8);
\draw [dotted, thick, red] (0.32,0.3) -- (0.32,3.3);
\draw[fill=red] (-0.31,0.3) circle [radius=0.055];
\node [below] at (-0.31,0.3) {$x_3$};
\draw[fill=blue] (0.14,pi) circle [radius=0.055];
\node [above] at (0.2,pi) {$x_1$};
\draw[fill=red] (0.32,0.3) circle [radius=0.055];
\node [below] at (0.32,0.3) {$x_4$};
\draw[fill=blue] (0.75,pi) circle [radius=0.055];
\node [above] at (0.75,pi)  {$x_2$};
\node [above] at (0,4) {$f(x)$};
\node [right] at (1.5,0.3) {$x$};
\node [left] at (-1.2,3.14) {$\pi$};
\node[align=left, below] at (0,-0.3)%
{(d)\ $d>d_0$};
\end{tikzpicture}\ \ \ \
\begin{tikzpicture}[yscale=0.7]
\draw [->] (-1.5,0.3) -- (1.5,0.3);
\draw [->] (0,0) -- (0,4);
\draw [thick,domain=-1:1] plot (\x, {1/180*pi*atan(20*tan(\x r))+0.5*pi-0.5*(\x-0.57)});
\draw [semithick,domain=-1.2:1.2] plot (\x, {0*\x+pi});
\node [above] at (0,4) {$f(x)$};
\node [right] at (1.5,0.3) {$x$};
\node [left] at (-1.2,3.14) {$\pi$};
\node[align=left, below] at (0,-0.3)%
{(e)\ $d=d_0$};
\end{tikzpicture}\ \ \ \
\begin{tikzpicture}[yscale=0.7]
\draw [->] (-1.5,0.3) -- (1.5,0.3);
\draw [->] (0,0) -- (0,4);
\draw [thick,domain=-1:1] plot (\x, {1/180*pi*atan(20*tan(\x r))+0.5*pi-0.5*(\x-0.4)});
\draw [semithick,domain=-1.2:1.2] plot (\x, {0*\x+pi});
\node [above] at (0,4) {$f(x)$};
\node [right] at (1.5,0.3) {$x$};
\node [left] at (-1.2,3.14) {$\pi$};
\node[align=left, below] at (0,-0.3)%
{(f)\ $0\le d<d_0$};
\end{tikzpicture}
\end{center}

\noindent The proof will be given in the Appendix~\ref{c}.
\begin{remark}\label{r.type-of-functions}
Let $I'=B(x_0,r')$ with $r'<\frac r2$, then the restriction of type $\I$ $f$ on $I'$ is still of type $\I$ for $I'$. Let $I''=B(0,r'')$ with $r''\le r$, then the restriction of type $\II$ $f$ on $I''$ is still of type $\II$ for $I''$ if all solutions to $f(x)=0$ and $\frac{df}{dx}=0$ are contained in $\frac{I''}{2}$. Let $I'''\subset I$ be any connected interval containing $B(x_1, Cl^{-\frac14})$. We also call the restriction of the type $\mathrm{III}$ $f$ on $I'''$ is again of type $\mathrm{III}$ for $I'''$. Note that if $d$ is sufficiently close to $d_0$ and $r'''$ is sufficiently small, then the restriction of type $\III$ function $f$ to $B(x_1,r''')$ may become type $\II$. However, for the sake of simplicity, we still call this restriction is of type $\III$. In any case, we have the following corollary which is important for the  application of Lemma~\ref{l.nstep-deri-estimate}--\ref{l.s-deri-resonance}.
\end{remark}

\begin{corollary}\label{c.distance-away-from-zero}
Let $f:I\rightarrow\R\PP^1$ be of type $\I$, $\II$ or $\III$. Define
$$
X=\{x\in I: |f(x)|=\min_{y\in I}|f(y)|\}=\begin{cases}\{x_0\}, &\mbox{if f is of type }\I\\\{x_1,x_2\}, &\mbox{if f is of type }\II\mbox{ or }\III.\end{cases}
$$
In case $f$ is of type $\III$, we further assume $d:=|x_1-x_2|<\frac r3$. Then for any $0<r'<r$, we have that
\beq\label{distance-away-from-zero}
|f(x)|>cr'^3,\mbox{ for all } x\notin B(X,r').
\eeq
For the case that $f$ is of type $\III$, we have the same estimate (\ref{distance-away-from-zero}) for $Cl^{-\frac14}<r'<r$ if $d\ge\frac r3$.
\end{corollary}
\begin{proof}
Let us consider the case that $f$ is of type $\I$ first. If $r>\frac r2$, then (\ref{distance-away-from-zero}) is obtained by definition. If $r'\le \frac r2$, then for all $x\notin B(X,r')$, it holds that
$$
|f(x)|>r^2|x-x_0|>r^2r'>cr'^3.
$$

If $f$ is of type $\II$, then clearly for some $d',\ d''\ge0$ satisfying $d'+d''=|x-x_{j}|$, we have that for all $x\notin B(X,r')$, it holds that
$$
|f(x)|>cd'^2+r^2|d''|>c|x-x_j|^3>cr'^3.
$$

If $f$ is of type $\III$ and $d<\frac r3$. Then (\ref{distance-away-from-zero}) follows from Lemma~\ref{l.s-deri-resonance} directly for $r'>\frac r6$. If $|x_1-x_3|\le r'\le \frac r6$, we partition $B(X,r')$ as
 $$
 B(X,r')=[x_1-r',x_3]\cup [x_3,x_1]\cup J.
 $$

Then for the part $J$, the corresponding growth of $f$ follows from (\ref{type3-nondegeneracy}) of Lemma~\ref{l.s-deri-resonance} and the same argument for type $\II$ functions. For the part $[x_3,x_1]$, the issue is that $f(x)$ may increase too fast from near $-\pi$ to near $0$. However, by (\ref{type3-nonzero-minimum}) of Lemma~\ref{l.s-deri-resonance}, we have $|x_3-x_1|=\eta_0l^{-1}$ and
$$
|f(x_3)|>\eta_1l^{-1}-\pi>c|x_3-x_1|^3-\pi,
$$
which is also the local minimal. Hence, we have the corresponding growth of $f(x)$.  For the part $[x_1-r',x_3]$, again by (\ref{type3-nondegeneracy}) of Lemma~\ref{l.s-deri-resonance}, we have corresponding growth as those for type $\II$ function .

If $0<r'<|x_1-x_3|$, we partition $B(X,r')$ as $B(X,r')=[x_1-r', x_1]\cup J$. Then it can be treated similarly as the case $|x_1-x_3|\le r'\le \frac r6$.

If $d\ge \frac r3$ and $Cl^{-\frac14}<r'<r$, then it follows from Lemma~\ref{l.s-deri-resonance} and the same argument as the one for type $\I$ functions.
\end{proof}

\section{{\bf Getting started}}\label{getting-started}
Consider the sequence $\{\l_{n}\}_{n\ge N}^{\infty}$ by $\log\l_{n}=\log\l_{n-1}-C\frac{\log q_{n}}{q_{n-1}}\log\l_{n-1}$ with $\l_N=\l$. It is easy to see that for all $\e$, there exists a $\lambda$ such that $\l_n$ decreases to some $\l_{\infty}$ with $\l_{\infty}>\l^{1-\e}$.
For two finite sets $C_j\subset\R/\Z$, $j=1,2$, we define $|C_1-C_2|=\min_{c_1\in C_1, c_2\in C_2}|c_1-c_2|$.

For $n\ge 1$, let $s_n(x)=s[A_n(x)]$ and $u_n(x)=s[A_{-n}(x)]$. Note they may depend on the parameter $t$. These two functions will play the role of $n$-step stable and unstable directions. We call them $n$-step stable and unstable directions, respectively, since it is not very difficult to see that they converge to the stable and unstable directions in case one has a positive Lyapunov exponent, see, for example, the proof of \cite[Theorem 1]{zhang2}. Obviously, we have that $u_1(x)=0$ and
$$
s^t_1(x)=\frac\pi2-\phi(x,t)=\frac\pi2-\cot^{-1}[t-v(x)]=\tan^{-1}[t-v(x)] .
$$
Let us define the following function, $g^t_1$, which is the difference between the first step stable and unstable direction:
\beq\label{g1}
g^t_1(x):=s_1(x)-u_1(x)=\tan^{-1}[t-v(x)].
\eeq

It is not difficult to see that we only need to consider
$$
t\in \CI:=[\inf v-\frac 2\l_0, \sup v+\frac 2\l_0] \mbox{ for all }\l>\l_0,
$$
see, for example, Lemma 11 of \cite{zhang}. From now on, let us restrict $t$ to this interval, and the dependence of $g_1^t$ on $t\in\CI$ will be left implicit.

By (\ref{g1}), it is a straightforward computation to see that for all $t\in\CI$,
\beq\label{g1-bound}
\|g_1\|_{C^2}\le C,\mbox{ and }c\le\left|\frac{d^mg_1}{dx^m}/\frac{d^m(t-v)}{dx^m}\right|\le C
\eeq
for $m=0,1$ and all $x\in\R/\Z$. Thus, for all $t\in\CI$, $\frac{dg_1}{dx}=0$ have the same solution with and $\frac{dv}{dx}=0$, which are $z_1$ and $z_2$. Moreover, it is a straightforward calculation to see that $|\frac{d^2g_1}{dx^2}(z_j)|>c$ for all $t\in\CI$. Clearly, there exists a $r>0$ such that on $B(x_j,r)$, we have for all $t\in\CI$,
\beq\label{deri2-1}
g_1(x)=g_1(z_j)+\frac{d^2g_1}{dx^2}(z_j)(x-z_j)^2+o(x^2)
\eeq
We also assume that, for the above $r$ and for all $t\in\CI$,
\beq\label{deri1-1}
\left|\frac{dg_1}{dx}(x)\right|>cr,\mbox{ for all } x\notin B(z_j,r).
\eeq
By choosing $N$ sufficiently large, we may assume that $q_N^{-2\tau}\ll r$. Let
$$
C_1=\{y: |g_1(y)|=\min_{x\in\R/\Z}|g_1(x)|\} \mbox{ and }I_1=B(C_1,\frac1{2q_N^{2\tau}}).
$$
Clearly, $C_1$ contains at most two points. So we may let
$$
C_1=\{c_{1,1}, c_{1,2}\}\mbox{ and }I_{1,j}=B(c_{1,j},\frac1{2q_N^{2\tau}}),\ j=1,2.
$$
Note it is possible that $c_{1,1}=c_{1,2}$.

\subsection{Step 1}
For each $t\in\CI$, we have one of the following cases.
\vskip 0.2cm
\noindent $(1)_\I$\  {\bf(Type $\mathrm{I}$)}\
$I_1$ consists of two disjoint connected intervals. In other words, $I_{1,1}\cap I_{1,2}=\varnothing$. Then, it is easy to see that $g_1(c_{1,j})=0$ and $|c_{1,1}- c_{1,2}|\geq\frac{1}{q_N^{2\tau}}$. Then by (\ref{deri2-1}) and (\ref{deri1-1}), it is straightforward that $g_1$ is of type $\mathrm{I}$ on $I_{1,1}$ and $I_{1,2}$. Furthermore, if $g_1$ is of type $\mathrm{I_+}$ on $I_{1,1}$, then it is of type $\mathrm{I_-}$ on $I_{1,2}$, vice versa. Let $(1)_\I$ denotes this case.
\vskip 0.2cm
\noindent $(1)_\II$\ {\bf( Type $\mathrm{II}$)} $I_1$ consists of one connected interval. Hence $0\le |c_{1,1}- c_{1,2}|< \frac{1}{q_N^{2\tau}}$. Clearly, $g_1$ is of type $\mathrm{II}$ on $I_1$ in this case. Let $(1)_\II$ denote this case.

\vskip 0.2cm
Thus, by Corollary~\ref{c.distance-away-from-zero}, we have that for each $t$ and each $x\notin I_1$, $|g_1(x)|>cq_N^{-6\tau}$. Let $\eta'_N=\frac{C\log q_N}{\log\l_N}\ll \frac{C\log q_{N+1}}{q_N}\ll1$. Fix a connected interval $I\subset\mathbb R/\mathbb Z$ and $\ell<\l_N^\frac12$. Assume that $x+j\a\notin I_1$ for all $x\in I$ and for all $1\le j\le \ell-1$. Then by Lemma~\ref{l.reduce-form} and Lemma~\ref{l.nstep-deri-estimate}, we have that for $x\in I$, it holds that
\beq\label{s1-c2closeto-sn}
\|s_\ell(x)-s_1(x)\|_{C^2},\ \|u_\ell(x+\ell\a)-u_1(x+\ell\a)\|_{C^2}\le C\l_N^{-\frac32},
\eeq
\beq\label{norm1-growth-deri}
\|A_\ell(x)\|\ge \l_N^{(1-\eta'_N)\ell},\ \left|\frac{d^m\|A_\ell(x)\|}{dx^m}\right|<\|A_\ell(x)\|^{1+m\eta'_N},\ m=1,2.
\eeq

\subsection{From step 1 to step 2}
Define $q_N-1<r^{\pm}_1:I_1\rightarrow\Z^+$ to be the smallest positive number $j$ such that $j>q_N-1$ and $T^{\pm j}x\in I_1$  for $x\in I_1$, respectively. Thus, there exist three possible cases:

\begin{itemize}
\item $I_1$ is in case $(1)_\II$. Thus $r^{\pm}_1(x)$ is the actual first return time by the \emph{Diophantine} condition and we call this the non-resonance case;
\item $I_1$ is in case $(1)_\I$ and $r^{\pm}_1(x)$ is the actual first return time for all $x\in I_1$. We also call this the non-resonance case and denote it by $(1)_{\I,NR}$;
\item $I_1$ is in case $(1)_\I$ and $r^{\pm}_1(x)$ is the second return time for some $x\in I_1$. We call this the resonance case and denote it by $(1)_{\I,R}$.
\end{itemize}
Let $r^{\pm}_1=\min_{x\in I_1}r^{\pm}_1(x)$ and $r_1=\min\{r^{+}_1,r^-_1\}$.

\subsubsection{{\bf Nonresonance case}}
By the \emph{Diophantine} condition, it is easy to see that in cases $(1)_\II$, it holds that $r_1\ge q_N^2$. First note that in this case $I_1$ is a connected interval of length at most $\frac{2}{q_N^{2\tau}}$. Thus, assume $x+n\a\in I_{1}$ for some $n$ and $x\in I_{1}$, then by the \emph{Diophantine} condition $\gamma n^{-\tau+1}<\|n\a\|_{\R/\Z}<2q_N^{-2\tau}$, where $\|\cdot\|_{\R/\Z}$ denotes the distance to the nearest integers. Hence $n>cq_N^{\frac{2\tau}{\tau-1}}>q_N^2$. In case $(1)_{\I,NR}$, by definition, it holds that $r_1\ge q_N$.

On the other hand, we may choose $\l_N$ sufficiently large so that $r^\pm_1(x)\ll\l_N^{\frac12}$. Indeed, we only need to set $\l_N>q_s^2$ for some $q_N^{\tau^2}>q_s>q_N^{2\tau}$. The reason is that  for this $q_s$, we have $\|q_s\a\|_{\R/\Z}<\gamma q_s^{-\tau+1}<q_N^{-2\tau}$ which implies $r^\pm_1(x)\le q_s$. By the \emph{Diophantine} condition, there exists such a $q_s$. Thus in this case, (\ref{s1-c2closeto-sn}) and (\ref{norm1-growth-deri}) can be applied directly to this case with $I=I_1$ and $\ell=r^\pm_1$. Then we get

\beq\label{nonresonance-s1-c2closeto-sn}
\|s_{r^+_1}(x)-s_1(x)\|_{C^2},\ \|u_{r^-_1}(x)-u_1(x)\|_{C^2}\le C\l_N^{-\frac32}
\eeq
\beq\label{nonresonance-norm1-growth-deri}
\|A_{\pm r^\pm_1}(x)\|\ge \l_N^{(1-\eta'_N)r^\pm_1},\ \left|\frac{d^m\|A_{\pm r^\pm_1}(x)\|}{dx^m}\right|<\|A_{\pm r^\pm_1}(x)\|^{1+m\eta'_N},\ m=1,2.
\eeq

Now we consider the function $s_{r^+_1}$, $u_{r^-_1}:I_1\rightarrow\R\PP^1$ and define $g_2=s_{r^+_1}-u_{r^-_1}:I_1\rightarrow\R\PP^1$. Thus we get that as a functions on $I_1$,
\beq\label{Cr-close}
\left\|g_2-g_1\right\|_{C^2}<C\l_{N+1}^{-\frac32}.
\eeq
Combined with our assumption, it is clear that in cases $(1)_{\I,NR}$ and $(1)_\II$, we must be in the following cases.
\vskip 0.2cm
\noindent $(2)_\I$\  {\bf(Type $\mathrm{I}$)}\
$I_1$ consists of two disjoint connected intervals. In other words, $I_{1,1}\cap I_{1,2}=\varnothing$. Furthermore $g_2$ is of type $\mathrm{I}$ on $I_{1,1}$ and $I_{1,2}$. In addition, if $g_2$ is of type $\mathrm{I_+}$ on $I_{1,1}$, then it is of type $\mathrm{I_-}$ on $I_{1,2}$, vice versa. Let $(2)_\I$ denote this case.
\vskip 0.2cm

\noindent $(2)_\II$\ {\bf( Type $\mathrm{II}$)} $I_1$ consists of one connected interval.  $g_2$ is of type $\mathrm{II}$. Let $(2)_\II$ denote this case.
\vskip 0.2cm
In both cases, if we let $C_2=\{y: |g_2(y)|=\min_{x\in\R/\Z}|g_2(x)|\}=\{c_{2,1}, c_{2,2}\}$, then we clearly have $|c_{1,j}-c_{2,j}|<C\l_N^{-1}$.

\subsubsection{{\bf Resonance case}}

Now we consider the case $(1)_{\I,R}$, let $0<k<q_N$ be the actual first return time for some $x$. Note in this case, we must have that $(I_{1,1}\pm k\a)\cap I_{1,2}\neq\varnothing$. Without loss of generality, assume $(I_{1,1}+k\a)\cap I_{1,2}\neq\varnothing$. Also, by the \emph{Diophantine} condition and by the same argument as in nonresonance case, it is not difficult to see that once we have resonance, $(I_{1,1}+j\a)\cap I_{1,2}=\varnothing$ for all $j$ such that $|j|<q_N^2$ and $|j|\neq k$, and $r^\pm_1(x)\ll\l^{\frac12}$ for all $x\in I_1$. Hence, in this case, no matter whether $r^\pm_1(x)$ is the first return or the second return time for $x\in\I_1$, we have $q_N^2<r^{\pm}_1(x)<\l_N^{\frac12}$ for all $x\in I_1$.

Now, for $x\in I_{1,1}$, let us consider
$$
A_{r^+_1}(x)=A_{r^+_1-k}(x+k\a)A_k(x)\mbox{ and } A_{-r^-_1}(x).
$$
Clearly, for $A_{r^+_1-k}(x+k\a)$, $A_k(x)$ and $A_{-r^-_1}(x)$, (\ref{s1-c2closeto-sn}) and (\ref{norm1-growth-deri}) can be applied. Then we get the following facts. First we have
$$
\|u_1(x)-u_{r^-_1}(x)\|_{C^2},\ \|s_k(x)-s_1(x)\|_{C^2},\ \|u_k(x+k\a)-u_1(x+k\a)\|_{C^2} \mbox{ and }
$$
$$
\|s_{r^+_1-k}(x+k\a)-s_1(x+k\a)\|_{C^2}<C\l_{N}^{-\frac32};
$$

Secondly, let $\nu=r^+_1-k$, $-r^-_1$ or $k$. Then
$$
\left|\frac{d^m\|A_{\nu}(x)\|}{dx^m}\right|< \|A_{\nu}(x)\|^{1+m\eta'_N}\mbox{ for }m=1,2.
$$

Finally, again for $\nu=r^+_1-k$, $-r^-_1$ or $k$,
$$
\|A_\nu(x)\|>\l_{N+1}^{|\nu|}.
$$

Now, let us focus on
$$
A_{r^+_1}(x)=A_{r^+_1-k}(x+k\a)A_k(x).
$$
Let $l_k=\|A_k(x)\|$ and $l'=\|A_{r^+_1-k}(x+k\a)\|$. Then by definition, we have
$$
A_{r^+_1}(x)=R_{u_{r^+_1-k}(x+r^+_1\a)}\begin{pmatrix}l'&0\\0&l'^{-1}\end{pmatrix}R_{\frac\pi2-s_{r^+_1-k}(x+k\a)+u_k(x+k\a)}
\begin{pmatrix}l_k&0\\0&l_k^{-1}\end{pmatrix}R_{\frac\pi2-s_k(x)}.
$$
Clearly, we have
$$
\l_{N}^{(1-\eta'_N)k}\le l_k<\l_N^k<\l_N^{q_N}\ll\l_{N}^{(1-\eta'_N)(q_N^2-k)}\le\l_{N}^{(1-\eta'_N)(r^+_1-k)}<l',
$$
which implies that $\|A_{r^{+}_1}(x)\|>\l_{N+1}^{r^{+}_1}\ $.
Moreover, by Lemma~\ref{l.norm-deri}, it is easy to see that for sufficiently large $\l$, we have for all $x\in I_1$ and $m=1,2$,
\beq\label{resonance-norm1-growth-deri}
\left|\frac{d^m\|A_{r^{+}_1}(x)\|}{dx^m}\right|<C\|A_{r^{+}_1}(x)\|^{1+m\eta'_N+Cm\frac{\eta'_N}{q_N}},
\eeq
where $\log\l_{N+1}>(1-\frac{C}{q_N})\log\l_N$.

Now, let us consider the function $g_2$. It is enough to consider $g_2:I_{1,1}\rightarrow\mathbb R\mathbb P^1$. Clearly, $s_{r^+_1}(x)=s[B(x)]:I_1\rightarrow\R\PP^1$ for the following $B$.
$$
B(x)=\begin{pmatrix}l'&0\\0&l'^{-1}\end{pmatrix}R_{\frac\pi2-s_{r^+_1-k}(x+k\a)+u_k(x+k\a)}\begin{pmatrix}l_k&0\\0&l_k^{-1}\end{pmatrix}R_{\frac\pi2-s_k(x)}.
$$
Let $g'_{1,1}=s_k-u_{r^-_1}:I_{1,1}\rightarrow\R\PP^1$ and $g'_{1,2}=s_{r^+_1-k}-u_k:I_{1,2}\rightarrow\R\PP^1$. Thus we have
\beq\label{g'1-close-to-g1}
\left\|g'_{1,j}-g_1\right\|_{C^2,I_{1,j}}<C\l_{N}^{-\frac32},\ j=1,2.
\eeq
Thus $g'_{1,j}$ is of the same type as $g_1$ on $I_{1,j}$. Let $\bar c_{1,j}\in I_{1,j}$ be the zero of $g'_{1,j}$.
Note as in the nonresonance case, we have
\beq\label{1st-intermediate-critical}
|\bar c_{1,j}-c_{1,j}|<C\l_{N}^{-\frac32}\mbox{ for } j=1,2.
\eeq

By the first estimate of (\ref{s-e2-great-e1}) of Lemma~\ref{l.ess-change}, to do the $C^2$ estimate of $g_2$, it suffices to take
\beq\label{g2-on-I11}
g_2(x)=\tan^{-1}\left(l_k^2\tan[g'_{1,2}(x+k\a)]\right)-\frac\pi2+g'_{1,1}(x),\ x\in I_{1,1}.
\eeq
Similarly, for $x\in I_{1,2}$, by considering
$$
A_{r^+_1}(x)\mbox{ and } A_{-r^-_1}(x)=A_{-r^-_1+k}(x-k\a)A_{-k}(x),
$$
we may take
\beq\label{g2-on-I12}
g_2(x)=\tan^{-1}\left(l_k^2\tan[g'_{1,1}(x-k\a)]\right)-\frac\pi2+g'_{1,2}(x),\ x\in I_{1,2}.
\eeq
Thus, by Lemma~\ref{l.s-deri-resonance}, $g_2$ are of type $\mathrm{III}$ on $I_{1,1}\cup (I_{1,2}-k\a)$ and $I_{1,2}\cup (I_{1,1}+k\a)$. Let $d:=\bar c_{1,1}+k\a-\bar c_{1,2}$ and assume without loss of generality $d\ge 0$. Then, depending on the size of $d$, we get the following for step 2.

\vskip0.2cm
\noindent ({\bf Nonresonance}) There is a $d_0$ close to $\frac{q_N^{-2\tau}}{2}$ such that for $d>d_0$, we have
$$
\|g_2-g_1\|_{C^2}<\|g_2-g'_{1,j}\|_{C^2}+\|g'_{1,j}-g_1\|_{C^2}<Cl_k^{-\frac32}+C\l_N^{-\frac32}<C\l_N^{-\frac32}.
$$
by (\ref{type3-1}) and (\ref{g'1-close-to-g1}). Thus, this basically goes to the case $(2)_\I$. In other words, $g_2$ is of the same type as $g_1$ on $I_{1,j}$, $j=1,2$.

\vskip0.2cm
\noindent ({\bf Weak resonance}) If $d\le d_1$, then the drastic change part of graph of $\tan^{-1}(l_k^2\tan[g'_{1,j_1}(x\pm k\a)])$  gradually enters $I_{1,j_2}$, where $j_1\neq j_2\in\{1,2\}$. Hence $g_2$ is of type $\III$ by Remark~\ref{r.type-of-functions}. By Lemma~\ref{l.s-deri-resonance}, there is a $d_0=\eta_1l_k^{-1}\le d_1$ with $q_N^{-4\tau}<\eta_1<q_N^{-4\tau}$ such that for $d>d_0$, $g_2(x)=0$ has two solutions. We say this comes from the case $(1)_{\I,WR}$.

\vskip0.2cm
\noindent ({\bf Strong resonance}) If $0\le d\le d_0$, then $g_2(x)=0$ has only one or even no solution. Here are the essential differences between the resonance case and the nonresonance case. The reason is that we have tangential intersections or even separation of $r_1$-step stable and unstable directions. The separation may also lead to $\mathcal U\mathcal H$ (see Remark \ref{r.uh-stopping}). We say this comes from the case $(1)_{\I,SR}$.\\

We say the last two cases are in case $(2)_\III$ and come from case $(1)_{\I,R}$. Now let us focus on the case $(2)_\III$ and consider without loss of generality $g_2$ on $I_{1,1}$, which is given by (\ref{g2-on-I11}). By the (\ref{type3-zero-order}) and (\ref{type3-zero-1}) of Lemma~\ref{l.s-deri-resonance}, we may let $c_{2,1}\in I_{1,1}$ be the minimal point of $g_2$ that is closer to $\bar c_{1,1}$ than other ones. Thus, by (\ref{type3-zero}), we have $|\bar c_{1,1}-c_{2,1}|<Cl_k^{-\frac34}<C\l_N^{-\frac34}$. This together with (\ref{1st-intermediate-critical}), clearly implies that
$$
|c_{1,1}-c_{2,1}|<C\l_N^{-\frac34}.
$$
Similarly, by considering (\ref{g2-on-I12}), we find a minimal point, $c_{2,2}\in I_{1,2}$, of $g_2$ such that
$$
|c_{1,2}-c_{2,2}|<C\l_N^{-\frac34}.
$$
We say $c_{2,j}$ comes essentially from $c_{1,j}$ for $j=1,2$. It is possible that $g_2$ has one or two minimum points on $I_{1,j}$. Let $C'_2=\{c'_{2,1}, c'_{2,2}\}$ with $c'_{2,1}\in I_{1,2}$ and $c'_{2,2}\in I_{1,1}$ be the possible extra minimum points of $g_2$. By (\ref{g2-on-I11}), (\ref{g2-on-I12}), (\ref{type3-zero-order}) and (\ref{type3-zero-1}), we have that $c'_{2,2}$ is closer to $\bar c_{1,2}-k\a$ than $c_{2,1}$, and $c'_{2,1}$ closer to $\bar c_{1,1}+k\a$ than $c_{2,2}$. On the other hand, $c'_{2,j}$ is essentially on the $k$-orbit of $c_{2,j}$, $j=1,2$, which is illustrated by the following lemma.
\begin{lemma}\label{critical-points-orbit}
Assume we have either $|g_2(c_{2,1})|<C\lambda_{N+1}^{-\frac1{10}r_1}$ or $|g_2(c_{2,2})|<C\lambda_{N+1}^{-\frac1{10}r_1}$, then
$$
|c_{2,1}+k\a-c'_{2,1}|,\ |c_{2,2}-k\a-c'_{2,2}|<C\lambda_{N+1}^{-\frac1{30} r_1}.
$$
\end{lemma}

To prove Lemma~\ref{critical-points-orbit}, we need the following lemma and corollary.
\begin{lemma}\label{l.almost-inv-s-direction}
Let $E=E_2E_1\in\mathrm{SL}(2,\mathbb R)$ such that $\|E_2\|\gg \|E_1\|\gg 1$. Then we have
\beq\label{almost-inv-s-direction}
|E_1^{-1}\cdot s(E_2)-s(E)|<C\|E\|^{-2},\ |s(E_2)-E_1\cdot s(E)|<C\|E_2\|^{-2}.
\eeq
\end{lemma}
\begin{proof}
Let $\hat s\in s$ be an unit vector. By polar decomposition, it suffices to consider the case
$$
E_2=\begin{pmatrix}e_2&0\\0&e_2^{-1}\end{pmatrix}R_{\t},\ E_1=\begin{pmatrix}e_1&0\\0&e_1^{-1}\end{pmatrix}.
$$
Then $s(E_2)=\frac\pi2-\t$ and $\tan [E_1^{-1}\cdot s(E_2)]=e_1^2\cot\t$. Assume $\t\neq0$, otherwise it is trivial. Let $w\in E_1^{-1}\cdot s(E_2)$ be a unit vector. Then we have
\begin{align*}
\|Ew\|&=\frac1{\sqrt{1+e_1^4\cot^2\t}}\left\|E\binom{1}{e_1^2\cot\t}\right\|=\frac{e_1e_2^{-1}|\sin\t|^{-1}}{\sqrt{1+e_1^4\cot^2\t}}\\&
=\frac1{\sqrt{e_1^{-2}e_2^2\sin^2\t+e_1^2e_2^2\cot^2\t}}=C\|E\|^{-1},
\end{align*}
which clearly implies the first inequality of (\ref{almost-inv-s-direction}).

For the proof of the second inequality, let $w'\in E_1\cdot s(E)$ be a unit vector. Then we have
$$\|E_2w'\|=\frac{1}{\|E_1\hat s(E)\|}\|E_2E_1\hat s(E)\|=\frac{1}{\|E\|\cdot\|E_1\hat s(E)\|}.
$$
Since $\|E\|\gg \|E_1\|$, by the first inequality of (\ref{almost-inv-s-direction}), we could replace $\hat s(E)$ by a unit vector, $w''$, in $E_1^{-1}\cdot s(E_2)$ in the above estimate. Thus we get
\begin{align*}
\|E_2w'\|&=\frac{C}{\|E\|\cdot\|E_1w''\|}=\frac{C\sqrt{1+e_1^4\cot^2\t}}{\|E\|}\left\|E_1\binom{1}{e_1^2\cot\t}\right\|^{-1}\\&
=\frac{C\sqrt{e_1^{-2}\sin^2\t+e_1^2\cos^2\t}}{\sqrt{e_1^{-2}e_2^2\sin^2\t+e_1^2e_2^2\cot^2\t}}=Ce_2^{-1}\\&=C\|E_2\|^{-1},
\end{align*}
which implies the second inequality of (\ref{almost-inv-s-direction}).
\end{proof}

The following corollary is an immediate consequence of Lemma~\ref{l.almost-inv-s-direction}.
\begin{corollary}\label{c.almost-inv-u-direction}
Let $E_1$ and $E_1$ be as in Lemma~\ref{l.almost-inv-s-direction}. Let $E=E_1\cdot E_2$, then we have
\beq\label{almost-inv-u-direction}
|E_1\cdot u(E_2)-u(E)|<C\|E\|^{-2},\ |u(E_2)-E_1^{-1}\cdot u(E)|<C\|E_2\|^{-2}.
\eeq
\end{corollary}
\begin{proof}
Note that $u(E)=s(E^{-1})$, $u(E_2)=s(E_2^{-1})$ and $E^{-1}=E_2^{-1}\cdot E_1^{-1}$. This reduces the proof to the case in Lemma~\ref{l.almost-inv-s-direction}.
\end{proof}

\begin{remark}
It is probably interesting to point out that Lemma~\ref{l.almost-inv-s-direction} and Corollary~\ref{c.almost-inv-u-direction} are one more precise version of the first estimate of (\ref{s-e2-great-e1}) and the second estimate of (\ref{s-e1-great-e2}) in case of $C^0$ estimate.
\end{remark}

Now we are ready to prove Lemma~\ref{critical-points-orbit}.
\begin{proof}({\bf Proof of Lemma~\ref{critical-points-orbit}}). For $x\in I_{1,1}$, let us consider the following
$$
A_{r^+_1}(x+k\a)\cdot A_k(x)\cdot A_{r^-_1}(x-r^-_1\a).
$$
By Lemma~\ref{l.almost-inv-s-direction} and Corollary~\ref{c.almost-inv-u-direction}, we obtain
$$
|s_{r^+_1}(x)-A_k(x)^{-1}\cdot s_{r^+_1}(x+k\a)|,\ |u_{r^-_1}(x)-A_k(x)^{-1}\cdot u_{r^-_1}(x+k\a)|<C\l_{N+1}^{-2r_1}.
$$
Hence, we have
\beq\label{almost-inv-gap-function1}
\left|g_2(x)-[A_k(x)^{-1}\cdot s_{r^+_1}(x+k\a)-A_k(x)^{-1}\cdot u_{r^-_1}(x+k\a)]\right|<C\l_{N+1}^{-2r_1}.
\eeq
For $A\in\mathrm{SL}(2,\mathbb R)$, consider the induced map $A:\mathbb R\mathbb P^1\rightarrow \mathbb R\mathbb P^1=\R/(\pi\Z)$. For $\t\in\mathbb R\mathbb P^1$, let $\hat\t\in \t$ be a unit vector. Then a direct computation shows that $\frac{dA}{d\t}(\t)=\|A\hat\t\|^{-2}$. Together with (\ref{almost-inv-gap-function1}) and fact that $l_k=\|A_k(x)\|<\l^{k}\ll \l_{N+1}^{r_1}$, this clearly implies that
\beq\label{almost-inv-gap-function2}
|g_2(x)-M\cdot g_2(x+k\a)|<C\l_{N+1}^{-2r_1},\ l_k^{-2}\le M\le l_k^{2}.
\eeq
Note here $g_2$ is a function on $I_{1,1}$. Similarly, we get that, as a function on $I_{1,2}$, $g_2$ satisfies
$$
\left|g_2(x)-[A_k(x)\cdot s_{r^+_1}(x-k\a)-A_k(x)\cdot u_{r^-_1}(x-k\a)]\right|<C\l_{N+1}^{-2r_1}.
$$
Hence,
\beq\label{almost-inv-gap-function3}
|g_2(x)-M\cdot g_2(x-k\a)|<C\l_{N+1}^{-2r_1},\ l_k^{-2}\le M\le l_k^{2}.
\eeq
Without loss of generality, assume that $|g_2(c_{2,1})|<\l_{N+1}^{-\frac1{10}r_1}$. Then, by (\ref{almost-inv-gap-function3}), we have
$$
|g_2(c_{2,1}+k\a)|<Cl_k^{2}\l_{N+1}^{-\frac1{10}r_1}<C\l_{N+1}^{-\frac1{15} r_1}.
$$
Clearly, $0\le |g_2(c'_{2,1})|=|g_2(c_{2,2})|\le |g_2(c_{2,1}+k\a)|<C\l_{N+1}^{-\frac1{15} r_1}$. Now $c'_{2,1}$ is always the minimal point of $g_2$ on $I_{1,2}$ that is closer to $\bar c_{1,1}+k\a$, hence $c_{2,1}+k\a$, than other ones. Thus, by Lemma~\ref{l.s-deri-resonance}, we get that
$$
|c_{2,1}+k\a-c'_{2,1}|<C\l_{N+1}^{-\frac1{30}r_1}.
$$
Note that if $g_2(c_{2,1})\neq0$, then $c_{2,1}=c_{2,2}'$. If $g_2(c_{2,1})=0$, then $g_2(c'_{2,2})=0$. In any case, we get a similar relation between $c'_{2,2}$ and $c_{2,2}$, concluding the proof.
\end{proof}
Finally, note that in case $(2)_\III$, we also have the following estimate by (\ref{norm-deri2}) and (\ref{norm-deri3}) of Lemma~\ref{l.norm-deri}
$$
\|g_2\|_{C^2}<Cl_k^{5}<C\l^{5q_N}.
$$

\subsection{The starting lemma}
To conclude, at step 2, we have the following lemma.
\begin{lemma}[The starting Lemma]\label{l.starting}\quad
Let $g_1=s_1-u_1=\tan^{-1}(t-v):\R/\Z\rightarrow\R\PP^1$. Define
$$
C_1=\{c_{1,1},c_{1,2}\}=\{y:|g_1(y)|=\min_{x\in\R/\Z}|g_1(x)|\},\ I_{1,j}=\{x:|x-c_{1,j}|\le \frac{1}{2q_N^{2\tau}}\}
$$
and $I_1=I_{1,1}\cup I_{1,2}$. Let $q_{N}-1<r^{\pm}_1(x):I_1\rightarrow\mathbb Z^+$ be the first return time after time $q_N-1$, where $r^+_1$ is the forward return and $r^-_1$ backward. Let $r^{\pm}_1=\min_{x\in\R/\Z}r^{\pm}_1(x)$, $r_1=\min\{r^{+}_1,r^-_1\}$ and $r_0=1$. Let $g_2=s_{r^+_1}-u_{r^-_1}:I_1\rightarrow\R\PP^1$ and define
 $\eta_N=\frac{\log q_{N+1}}{q_N}$ and $\log\l_{N+1}>(1-C\eta_N)\log\l_N$. Then there exists a set
$$
C_2=\{c_{2,1},c_{2,2}\}\subset\{y:|g_2(y)|=\min_{x\in I_1}|g_2(x)|\}\mbox{ such that }
$$
\beq\label{c1-close-c2}
|c_{1,j}-c_{2,j}|<C\l_{N}^{-\frac34},\ j=1,2;
\eeq
Moreover, for all $x\in I_1$ and $m=1,2$, it holds that
\beq\label{norm-1}
\|A_{\pm r^{\pm}_1}(x)\|>\l_{N+1}^{r^{\pm}_1},\ \frac{d^m\|A_{\pm r^{\pm}_1}(x)\|}{dx^m}<\|A_{\pm r^{\pm}_1}(x)\|^{1+m\eta_N};
\eeq
if $I_1$ consists of one connected interval, then $g_2$ is of type $\II$ on $I_1$; if $I_{1,1}\cap I_{1,2}=\varnothing$, then in nonresonance case, $g_2$ is either of type $\I$; in resonance case $g_1$ is of type $\III$ on each $I_{1,j}$. In other words, we have three different cases: $(2)_\I$, $(2)_\II$ and $(2)_\III$.
\vskip 0.2cm

In case $(2)_\I$ and $(2)_\II$, there is no other minimal point of $g_2$ than those in $C_2$, and we have that
\beq\label{g2-closeto-g1}
\|g_2-g_1\|_{C^2}\le C\l_{N}^{-\frac32} \mbox{ and } \|g_2\|_{C^2}<C.
\eeq
Moreover, in case $(2)_\I$, if $g_2$ is of type $\I_+$ on $I_{1,1}$, then it is of type $\I_-$ on $I_{1,2}$, vice versa.

In case $(2)_\III$, there are two more minimal points $c'_{2,1}, c'_{2,2}$ with $c'_{2,1}\in I_{1,2}$, $c'_{2,2}\in I_{1,1}$ such that $g_2(c'_{2,1})=g_2(c_{2,2})$ and $g_2(c'_{2,2})=g_2(c_{2,1})$. However, $c_{2,j}$ is always the one that comes essentially from $c_{1,j}$ for $j=1,2$ while it is possible that $c'_{2,1}=c_{2,2}$ and $c'_{2,2}=c_{2,1}$.

Moreover, if $|g_2(c_{2,j})|<C\l_{N+1}^{-\frac1{10}r_1}$ for $j=1$ or $j=2$, then we have
\beq\label{2-orbit-critical-points}
|c_{2,1}+k\a-c'_{2,1}|,\ |c_{2,2}-k\a-c'_{2,2}|<C\lambda_{N+1}^{-\frac1{30} r_1}
\eeq
with $1\le k< q_N$. In this case, we also have
\beq\label{g2-upbound}
r_1\ge q_N^2 \mbox{ and } \|g_2\|_{C^2}\le C\l^{5q_N}.
\eeq
\end{lemma}

\begin{remark}\label{r.small-error}
In Section~\ref{getting-started}, instead of $r^\pm_{1}(x)$, sometimes, we used $r^\pm_{1}$, $r_1$ and even $r^\pm_1\pm k$, $1\le k\le q_N-1$. The difference between the usage of $r^\pm_1(x)$, $r^\pm_1$, $r_1$ and $r_1\pm k$ are negligible. In fact, by the second estimate of (\ref{norm-deri2}) and (\ref{norm-deri3}) in Lemma~\ref{l.norm-deri}, it is easy to see that these differences produces errors of order at most $\l_{N+1}^{-\frac 32 r_1}$, which is clearly not important in all the necessary estimates. Similarly, in the following discussions, we will not distinguish the difference between $r^\pm_i(x)$, $r^\pm_i$, $r_i$ and $r_i\pm k$, where $i\ge1$ and $1\le k\le q_{N+i-1}$.
\end{remark}

\section{{\bf The Induction}}\label{induction}
Now we are ready to do the induction.
\subsection{Statement of the Induction Theorem}
We formulate our induction as the following theorem.
\begin{theorem}[Iteration Lemma]\label{t.iteration}
{\bf Step (i+1).} Let $g_i=s_{r^+_{i-1}}-u_{r^-_{i-1}}:I_{i-1}\rightarrow\R\PP^1$. Assume we have
$$
C_i=\{c_{i,1}, c_{i,2}\}\subset\{y:|g_i(y)|=\min_{x\in I_{i-1}}|g_i(x)|\} \mbox{ and } I_{i,j}=\{x:|x-c_{i,j}|\le \frac{1}{2^iq_{N+i-1}^{2\tau}}\}
$$
and $I_i=I_{i,1}\cup I_{i,2}$. Let $q_{N+i-1}<r^{\pm}_i(x):I_i\rightarrow\Z^+$ be the first return time after time $q_{N+i-1}$. Let $r_i=\min\{r^+_{i},r^-_{i}\}$ with $r^{\pm}_i=\min_{x\in I_i}r^{\pm}_i(x)$. Let $g_{i+1}=s_{r^+_i}-u_{r^-_i}:I_i\rightarrow\R\PP^1$ and assume we have
$$
C_{i+1}=\{c_{i+1,1}, c_{i+1,2}\}\subset\{y:|g_{i+1}(y)|=\min_{x\in I_i}|g_{i+1}(x)|\} \mbox{ such that}
$$
\beq\label{ci-close-ci+1}
|c_{i,j}-c_{i+1,j}|<C\l_{N+i}^{-\frac 34 r_{i-1}},\ j=1,2,
\eeq
where  $\log\l_{N+i}>(1-C\eta_{N+i-1})\log\l_{N+i-1}$ with $\eta_{N+i-1}=\frac{\log q_{N+i}}{q_{N+i-1}}$. Assume
 for all $x\in I_i$, $m=1,2$, it holds that
\beq\label{norm-i+1}
\|A_{\pm r^{\pm}_i}(x)\|>\l_{N+i}^{r^{\pm}_i},\ \frac{d^m\|A_{\pm r^{\pm}_i}(x)\|}{dx^m}<\|A_{\pm r^{\pm}_i}(x)\|^{1+mC\sum^{N+i-1}_{j=N}\eta_{j}},
\eeq
and $g_{i+1}:I_i\rightarrow\mathbb R\mathbb P^1$ is of type $\I$, $\II$ or $\III$, which are denoted as cases $(i+1)_\I$, $(i+1)_\II$ and $(i+1)_\III$.

Assume that in case $(i+1)_\I$ and $(i+1)_\II$, there is no other minimal point of $g_{i+1}$ than those in $C_{i+1}$, and we have
\beq\label{gi+1-closeto-gi}
\|g_{i+1}-g_i\|_{C^2}\le C\l_{N+i-1}^{-\frac32 r_{i-1}}\mbox { and } \|g_{i+1}\|_{C^2}\le C.
\eeq
Moreover, in case $(i+1)_\I$, assume that if $g_{i+1}$ is of type $\I_+$ on $I_{i,1}$, then it is of type $\I_-$ on $I_{i,2}$, vice versa.

Assume that in case $(i+1)_\III$, there are two more minimal points $c'_{i+1,1}, c'_{i+1,2}$ with $c'_{i+1,1}\in I_{1,2}$, $c'_{i+1,2}\in I_{1,1}$ such that $g_{i+1}(c'_{i+1,1})=g_{i+1}(c_{i+1,2})$ and $g_{i+1}(c'_{i+1,2})=g_{i+1}(c_{i+1,1})$. Assume $c_{i+1,j}$ is always the one comes essentially from $c_{i,j}$ while it is possible that $c'_{i+1,1}=c_{i+1,2}$ and $c'_{i+1,2}=c_{i+1,1}$.

If $|g_{i+1}(c_{i+1,j})|<C\l_{N+i}^{-\frac1{10}r_{i}}$ for $j=1$ or $j=2$, assume we have
\beq\label{i+1-orbit-critical-points}
|c_{i+1,1}+k\a-c'_{i+1,1}|,\ |c_{i+1,2}-k\a-c'_{i+1,2}|<C\l_{N+i}^{-\frac1{30} r_i}.
\eeq
with $1\le k< q_{N+i-1}$. In this case, we also assume that
\beq\label{gi+1-upbound}
r_i\ge q_{N+i-1}^2\mbox{ and }\|g_{i+1}\|_{C^2}\le C\l^{5q_{N+i-1}}.
\eeq

Then we have the following.
\vskip 0.4cm

\noindent {\bf Step (i+2).} Let $I_{i+1,j}=\{x:|x-C_{i+1,j}|\le \frac{1}{2^{i+1}q_{N+i}^{2\tau}}\}$ and $I_{i+1}=I_{i+1,1}\cup I_{i+2,j}$. Let $q_{N+i}<r^{\pm}_{i+1}(x):I_{i+1}\rightarrow\Z^+$ be the first return time after time $q_{N+i}$. Let $r^{\pm}_{i+1}=\min_{x\in I_{i+1}}r^{\pm}_{i+1}(x)$ and $r_i=\min\{r^+_i,r^-_i\}$. Let $g_{i+2}=s_{r^+_{i+1}}-u_{r^-_{i+1}}:I_{i+1}\rightarrow\R\PP^1$.
Define $\eta_{N+i}=\frac{\log q_{N+i+1}}{q_{N+i}}$ and $\log\l_{N+i+1}>(1-C\eta_{N+i})\log\l_{N+i}$. Then there exists a set
$$
C_{i+2}=\{c_{i+2,1},\ c_{i+2,2}\}\subset\{y:|g_{i+2}(y)|=\min_{x\in I_{i+1}}|g_{i+2}(x)|\}
$$ such that
\beq\label{ci-close-ci+2}
|c_{i+1,j}-c_{i+2,j}|<C\l_{N+i}^{-\frac34r_{i}},\ j=1,2;
\eeq
and for all $x\in I_{i+1}$ and $m=1,2$, it holds that
\beq\label{norm-i+2}
\|A_{\pm r^{\pm}_{i+1}}(x)\|>\l_{N+i+1}^{r^{\pm}_{i+1}},\ \frac{d^m\|A_{\pm r^{\pm}_{i+1}}(x)\|}{dx^m}<\|A_{\pm r^{\pm}_{i+1}}(x)\|^{1+mC\sum^{N+i}_{j=N}\eta_{j}}
\eeq
and $g_{i+2}:I_{i+1}\rightarrow\mathbb R\mathbb P^1$ is of types $\I$, $\II$ or $\III$, which are denoted as case $(i+2)_\I$, $(i+2)_\II$ and $(i+2)_\III$. In addition, in case $(i+2)_\I$ and $(i+2)_\II$, there is no other minimal point of $g_{i+2}$ than those in $C_{i+2}$, and we have
\beq\label{gi+2-closeto-gi+1}
\|g_{i+2}-g_{i+1}\|_{C^2}\le C\l_{N+i}^{-\frac32 r_{i}} \mbox{ and }\|g_{i+2}\|_{C^2}\le C.
\eeq
Moreover, in case $(i+2)_\I$, if $g_{i+2}$ is of type $\I_+$ on $I_{i+1,1}$, then it is of type $\I_-$ on $I_{i+1,2}$, vice versa.

In case $(i+2)_\III$, there are two more minimal points $c'_{i+2,1},c'_{i+2,2}$ such that $g_{i+2}(c'_{i+2,1})=g_{i+2}(c_{i+2,2})$ and $g_{i+2}(c'_{i+2,2})=g_{i+2}(c_{i+2,1})$. However, $c_{i+2,j}$ is always the one that comes essentially from $c_{i+1,j}$ while it is possible that $c'_{i+2,1}=c_{i+2,2}$ and $c'_{i+2,2}=c_{i+2,1}$. Moreover, it holds that
\begin{itemize}
\item if $|g_{i+1}(c_{i+1,j})|>C\l_{N+i}^{-\frac1{10}r_{i}}$, $j=1,2$, then so are $|g_{i+2}(c_{i+2,j})|$, $j=1,2$;
\item if $|g_{i+2}(c_{i+2,j})|<C\l_{N+i+1}^{-\frac1{10}r_{i+1}}$ for $j=1$ or $j=2$, then
      \beq\label{i+2-orbit-critical-points}
      |c_{i+2,1}+k\a-c'_{i+2,1}|,\ |c_{i+2,2}-k\a-c'_{i+2,2}|<C\l_{N+i+1}^{-\frac1{30} r_{i+1}}
      \eeq
      with $1\le k< q_{N+i}$.
\end{itemize}
In this case, we also have
\beq\label{gi+2-upbound}
r_{i+1}\ge q_{N+i}^2 \mbox{ and } \|g_{i+2}\|_{C^2}\le C\l^{5q_{N+i}}.
\eeq
\end{theorem}
\begin{remark}\label{r.returntime-not-large}
By the same argument that $r^\pm_1(x)<\l_N^\frac12$, it is easy to see that $r^\pm_{i}\ll\l_{N+i-1}^{\frac12 q_{N+i-2}}$ for all $i\ge 2$. Indeed, $r^\pm_i$ is no larger than the first return time of $x\in I_{i,j}$ back to $I_{i,j}$ for $j=1$ or $2$. Note $|I_{i,j}|=q_{N+i-1}^{-2\tau}$. Then it is sufficient to take the first $q_n$ such that $q_n>q_{N+i-1}^{2\tau}>q_{N+i-2}^{2\tau(\tau-1)}$. Then $\|q_n\a\|<cq_n^{-\tau+1}<q_{N+i-1}^{-2\tau}$. Thus, $r^\pm_i(x)<q_n$, which by the \emph{Diophantine} condition is polynomially large in $q_{N+i-2}$ with polynomial depending only on $\a$.
\end{remark}

\subsection{Proof of the Induction}
To study the $g_{i+2}:I_{i+1}\rightarrow\R\PP^1$, we need to consider $g_{i+1}:I_{i+1}\rightarrow\R\PP^1$. Let us start with the case $(i+1)_\I$. Clearly, $g_{i+1}:I_{i+1}\rightarrow\R\PP^1$ inherited all the conditions of $g_{i+1}:I_{i}\rightarrow\R\PP^1$ that are assumed in the induction step $(i+1)_\I$. As in Section 2.1, we divide it into cases $(i+1)_{\I,NR}$ and $(i+1)_{\I,R}$. We first consider the nonresonance case $(i+1)_{\I,NR}$. \\

\noindent ${\bf {Case }\ (i+1)_{\I,NR}}$. In this case, $(I_{i+1,1}+l\a)\cap I_{i+1,2}=\varnothing$ for all $|l|\le q_{N+i}-1$. Hence $r^{\pm}_{i+1}(x)$ is the actually first return time for each $x\in I_{i+1}$. If $r^+_{i+1}=r^+_{i}$, then $g_{i+2}=g_{i+1}$ and we have nothing to say. Otherwise, let $0\le j_m\le r^+_{i+1}$, $0\le m\le p$ be the times such that
\beq\label{i+1nr-notin-Ii+1}
j_{m+1}-j_{m}\ge q_{N+i-1},\ 0\le m\le p-1 \mbox{ and } x+j_m\a\in I_i\setminus I_{i+1},\ 1\le m\le p-1,
\eeq
where $j_0=0$ and $j_p=r^+_{i+1}$. Now consider the sequence
\beq\label{i+1nr-sequence}
A_{j_1}(x), A_{j_2}(x+j_1\a), \ldots, A_{j_p-j_{p-1}}(x+j_{p-1}\a).
\eeq
Then, we have the following.
\vskip0.2cm
First, for $1\le m\le p$, $\|A_{j_m-j_{m-1}}(x+j_{m-1}\a)\|\ge\l_{N+i}^{j_m-j_{m-1}}\ge\l_{N+i}^{q_{N+i-1}}$ and $\left|\frac{d^m\|A^+_{r_{i}}(x)\|}{dx^m}\right|<\|A^+_{r_{i}}(x)\|^{1+mC\sum^{N+i-2}_{j=N}\eta_j}$ by induction assumption (\ref{norm-i+1}).

\vskip0.2cm
Secondly, $|g_{i+1}(x+j_{m}\a)|\ge 2^{-2(i+1)}q_{N+i}^{-4\tau}>\l_{N+i}^{-\eta_{N+i-1} q_{N+i-1}}$ by induction assumption on the $g_{i+1}$, (\ref{i+1nr-notin-Ii+1}) and Corollary~\ref{c.distance-away-from-zero}.

\vskip0.2cm
Finally, from (\ref{gi+1-closeto-gi}), we have $\|g_{i+1}\|_{C^2}<C.$\\

Thus, the sequence (\ref{i+1nr-sequence}) satisfies all the conditions in Lemma~\ref{l.nstep-deri-estimate}. Hence, we get
\begin{itemize}
\item $\|A_{r^+_{i+1}}(x)\|\ge\l_{N+i+1}^{r^+_{i+1}}$ with $\log\l_{N+i+1}=(1-C\eta_{N+i-1})\log\l_{N+i}$;
\item $\left|\frac{d^m\|A^+_{r_{i+1}}(x)\|}{dx^m}\right|<\|A^+_{r_{i+1}}(x)\|^{1+mC\sum^{N+i-1}_{j=N}\eta_j}$;
\item $\|s_{r^+_{i+1}}-s_{r^+_{i}}\|_{C^2}\le C\l_{N+i}^{-\frac32r_i}$ on $I_{i+1}$.
\end{itemize}
Similarly, we get the estimate for backward sequences. Thus,  we also get that
$$
\|g_{i+2}-g_{i+1}\|_{C^2}\le C\l_{N+i}^{-\frac 32 r_i} \mbox{ on } I_{i+1},
$$
which together with $\|g_{i+1}\|_{C^2}<C$ clearly implies (\ref{gi+2-closeto-gi+1}). This also clearly implies the existence of two minimal points of $|g_{i+2}(x)|$, denoted by $c_{i+2, 1}$ and $c_{i+2, 2}$ such that
$$
|c_{i+2,j}-c_{i+1,j}|<C\l_{N+i}^{-r_i}<C\l_{N+i}^{-\frac 34r_i}.
$$
By (\ref{gi+2-closeto-gi+1}) and the induction assumption for $g_{i+1}:I_{i+1}\rightarrow\mathbb R\mathbb P^1$, $g_{i+2}:I_{i+1}\rightarrow\mathbb R\mathbb P^1$ is clearly in case $(i+2)_\I$. Moreover, if $g_{i+2}$ is of type $\I_+$ on $I_{i+1,1}$, then it is of type $\I_-$ on $I_{i+1,2}$, vice versa.\\

\noindent ${\bf {Case}\ (i+1)_{\I,R}}$. In this case, there exists a $r_{i}\le k\le q_{N+i}-1$ such that
$$
x+k\a\in I_{i+1,2} \mbox{ for some } x\in I_{i+1,1}.
$$
By the \emph{Diophantine} condition, $r_{i+1}\ge q_{N+i}^2$. Again, let $1\le j_m\le r_{i+1}$, $0\le m\le p$ be the times such that
$$
j_{m+1}-j_{m}\ge q_{N+i-1},\ 0\le m\le p-1 \mbox{ and } x+j_m\a\in I_i\setminus I_{i+1},\ 1\le m\le p-1,
$$
where we set $j_0=0$ and $j_p=r_{i+1}$. Let $l\in\{1,\ldots, p-1\}$ be that $j_l<k<j_{l+1}$. Consider for $x\in I_{i+1,1}$,
$$
A_{r^+_{i+1}-k}(x+k\a)\cdot A_k(x)\mbox{ and } A_{-r^-_{i+1}}(x),\mbox{ where } A_k(x)=A_{k-j_l}(x+j_l\a)\cdots A_{j_1}(x)
$$
$$
\mbox{ and } A_{r^+_{i+1}-k}(x+k\a)=A_{j_p-j_{p-1}}(x+j_{p-1})\cdots A_{j_{l+1}-k}(x+k\a).
$$
Clearly, for $A_k(x)$, $A_{r^+_{i+1}-k}(x+k\a)$ and $A_{-r^-_{i+1}}(x)$, the same argument of $(i+1)_{\I,NR}\rightarrow (i+2)_\I$ is applicable. Thus, by Lemma~\ref{l.nstep-deri-estimate}, we get the following. Let $\nu=r^+_{i+1}-k$, $-r^-_{i+1}$ or $k$. Then we have

\vskip0.2cm
First, $\|A_{\nu}(x)\|\ge\l_{N+i+1}^{|\nu|}$ with $\log\l_{N+i+1}=(1-C\eta_{N+i-1})\log\l_{N+i}$.

\vskip0.2cm
Secondly, $\left|\frac{d^m\|A_{\nu}(x)\|}{dx^m}\right|<\|A_{\nu}(x)\|^{1+mC\sum^{N+i}_{j=N}\eta_j}$.

\vskip0.2cm
Finally, $\|s_k(x)-s_{r^+_i}(x)\|_{C^2}$, $\|u_k(x+k\a)-u_{r^-_i}(x+k\a)\|_{C^2}$ and
$$
\|s_{r^+_{i+1}-k}(x+k\a)-s_{r^+_i}(x+k\a)\|_{C^2}<C\l_{N+i}^{-\frac32 r_i}.
$$

By the same reason, for $x\in I_{i+1,2}$, we could get a similar result when we consider
$$
A_{r^+_{i+1}}(x)\mbox{ and } A_{-r^-_{i+1}}(x)=A_{-r^-_{i+1}+k}(x-k\a)\cdot A_{-k}(x).
$$
Then everything follows from the same argument as of the case $(1)_{\I,R}$ in Section 2.1. More concretely, let $g'_{i+1,1}=s_k-u_{r^-_{i+1}}:I_{i+1,1}\rightarrow\mathbb R\mathbb P^1$, $g'_{i+1,2}=s_{r^+_{i+1}-k}-u_k:I_{i+1,2}\rightarrow\mathbb R\mathbb P^1$ and
$$
\bar c_{i+1,1}\in I_{i+1,1},\bar c_{i+1,2}\in I_{i+1,2} \mbox{ where } g'_{i+1,1}(\bar c_{i+1,1})=0, \ g'_{i+1,2}(\bar c_{i+1,2})=0.
$$
Clearly, we have for $j=1,2$,
\beq\label{i+1r-1}
\left\|g'_{i+1,j}-g_{i+1}\right\|_{C^2,I_{i+1,j}}<C\l_{N+i}^{-\frac 32r_i}\mbox{ and } |\bar c_{i+1,j}-c_{i+1,j}|<C\l_{N+i}^{-r_i}.
\eeq
This implies that $g'_{i+1,j}$ are of type $\I$ on $I_{i+1,j}$, one of which is of type $\I_+$ and the other $\I_-$. Let $d=\bar c_{i+1,1}+k\a-\bar c_{i+1,2}$. To estimate the geometric properties of $g_{i+2}$, it is sufficient to take
$$
g_{i+2}(x)=\begin{cases}\tan^{-1}(l_k^2\tan [g'_{i+1,2}(x+k\a)])-\frac\pi2+g'_{i+1,1}(x),\ &x\in I_{i+1,1},\\ \tan^{-1}(l_k^2[\tan g'_{i+1,1}(x-k\a)])-\frac\pi2+g'_{i+1,2}(x),\ &x\in I_{i+1,2}.\end{cases}
$$
Then depending on the size of $d$, we have the following.
\begin{itemize}
\item $(i+1)_{\I,R}\rightarrow (i+2)_{\I}$ if $d_0\le d<\frac{2}{2^{i+1}q_{N+i}^{2\tau}}$ for some $d_0$ close to $\frac{1}{2^{i+1}q_{N+i}^{2\tau}}$.
\item $(i+1)_{\I,R}\rightarrow (i+2)_{\III}$ if $0\le d<d_0$.
\end{itemize}
By (\ref{type3-zero}) of Lemma~\ref{l.s-deri-resonance}, there is a minimal point, $c_{i+2,j}$, of $g_{i+2,j}$ that is always closer to $\bar c_{i+1,j}$ than other ones such that
\beq\label{i+1r-2}
|c_{i+2,j}-\bar c_{i+1,j}|<Cl_k^{-\frac34}<C\l_{N+i}^{-\frac34 r_i}
\eeq
for $j=1,2$. (\ref{i+1r-1}) and (\ref{i+1r-2}) clearly implies that
$$
|c_{i+2,j}-c_{i+1,j}|<C\l_{N+i}^{-\frac34 r_i},\ j=1,2.
$$
Finally, in this case, by (\ref{norm-deri2}), (\ref{norm-deri3}) of Lemma~\ref{l.norm-deri} and (\ref{gi+1-closeto-gi}), (\ref{gi+1-upbound}) of Theorem~\ref{t.iteration}, we get
$$
\|g_{i+2}\|_{C^2}<Cl_k^{-4.5}<C\l^{5q_{N+i}}.
$$
\vskip 0.2cm
\noindent ${\bf Case\ (i+1)_\II}$. In this case, let $d=c_{i+1,1}-c_{i+1,2}$. By definition, $|d|<\frac{2}{2^iq_{N+i-1}^{2\tau}}$. Then we have the following.

\vskip0.2cm
 If $|d|\ge\frac{1}{2^{i+1}q_{N+i}^{2\tau}}$, $I_{i+1,1}\cap I_{i+1,2}=\varnothing$. Then, by induction assumption, it is not difficult to see that $g_{i+1}:I_{i+1}\rightarrow\mathbb R\mathbb P^1$ are essentially as in the case $(i+1)_\I$. Thus we can do it as in the previous step and get $(i+2)_\I$ or $(i+2)_\III$.

\vskip0.2cm
If $d<\frac{1}{2^{i+1}q_{N+i}^{2\tau}}$, then $I_{i+1}$ consists of one interval. Thus $r^\pm_{i+1}>q_{N+i}^2$ is the actually the first return. Moreover there exist two minimal points $c_{i+2, 1}$ and $c_{i+2, 2}$ of $g_{i+2}$ such that
$$
\|g_{i+2}-g_{i+1}\|_{C^2}<C\l_{N+i}^{-\frac32r_i}\mbox{ and } |c_{i+2,j}-c_{i+1,j}|<C\l_{N+i}^{-\frac34 r_i}.
$$
The estimate of $\|A_{\pm r^\pm_{i+1}}(x)\|$ together with its derivatives follows from the same argument as in case $(i+1)_\I$. So we will be in the case $(i+2)_\II$.

\vskip 0.2cm
\noindent ${\bf Case\ (i+1)_\III}$. In this case, other than $c_{i+1,j}\in I_{i,j}$, $j=1,2$, we have minimal points, $c'_{i+1,1}\in I_{i,2}$ and $c'_{i+1,2}\in I_{i,1}$, of $g_{i+1}$ such that if $|g_{i+1}(c_{i+1,j})|<C\l_{N+i}^{-\frac{1}{10}r_i}$, then
$$
|c'_{i+1,2}+k\a-c_{i+1,2}|,\ |c_{i+1,1}+k\a-c'_{i+1,1}|<\l_{N+i}^{-\frac1{30}r_i},\ 0<k\le q_{N+i-1}-1.
$$
Let $d=c_{i+1,1}-c'_{i+1,2}$. Then there exists a $d_0$ close to $\frac{2}{2^{i+1}q_{N+i}^{2\tau}}$ such that the following holds true.

\vskip0.3cm
If $d>d_0$, then it is easy to see that $g_{i+1}:I_{i+1}\rightarrow\mathbb R\mathbb P^1$ are essentially in the case $(i+1)_\I$. So we get the corresponding case in step $(i+2)$. Note in this case it is necessary that $g_{i+1}(c_{i+1,j})=0$.

\vskip0.3cm
If $0<d\le d_0$, then we know that $(I_{i+1,1}+k\a)\cap I_{i+1,2}\neq\varnothing$ for $k<q_{N+i-1}$. Let us first assume $d$ is not too small so that $|g_{i+1}(c_{i+1,j})|<C\l_{N+i}^{-\frac{1}{10}r_i}$. The other case will be dealt with later. In any case, by the \emph{Diophantine} condition, we have that $r^\pm_{i+1}>q_{N+i}^2$. Hence, for any $x\in I_{i+1}$, whenever $x+l\a\in I_{i}$ we must have $x+l\a\notin I_{i+1}$ for any $1\le |l|\le r_{i+1}-k$. Thus, by Lemma~\ref{l.nstep-deri-estimate}, it is not difficult to see that, as functions on $I_{i+1}$, $\|g_{i+2}-g_{i+1}\|_{C^2}<\l_{N+i}^{-\frac32r_i}$. This implies the existence of two minimal points $c_{i+2, 1}$ and $c_{i+2, 2}$ of $g_{i+2}$ such that
$$
|c_{i+1,j}-c_{i+2,j}|<C\l_{N+i}^{-\frac34r_i};
$$
and together with (\ref{gi+1-upbound}), we have
$$
\|g_{i+2}\|_{C^2}<C\l^{5q_{N+i-1}}<C\l^{5q_{N+i}}.
$$
The estimate of $\|A_{\pm r^\pm_{i+1}}(x)\|$ together with its derivatives comes from the same argument as previous cases. Thus, we are in case $(i+2)_\III$.

\vskip0.3cm
In addition, from $(i+1)_\III$ to $(i+2)_\III$ that, we need to show that if $|g_{i+1}(c_{i+1,j})|>C\l_{N+i}^{-\frac1{10}r_{i}}$, $j=1,2$, then so are $|g_{i+2}(c_{i+2,j})|$, $j=1,2$. This is a consequence of Lemma~\ref{l.nstep-deri-estimate}. Indeed, fix arbitrary $x\in I_{i+1}$, let $0\le j_m\le r^+_{i+1}$, $0\le m\le p$ be as in case $(i+1)_{\I,NR}$. Consider the following sequence
$$
A_{j_1}(x),\ldots, A_{j_{m+1}-j_m}(x+j_m\a),\ldots, A_{j_{p}-j_{p-1}}(x+j_{p-1}\a)
$$
and we clearly have
\begin{itemize}
\item $\|A_{j_{m+1}-j_m}(x+j_m\a)\|\ge \l_{N+i}^{r_i}$ and
\item $|g_{i+1}(x+j_m\a)|=C|s_{j_{m+1}-j_m}(x+j_m\a)-u_{j_m-j_{m-1}}(x+j_m\a)|>C\l_{N+i}^{-\frac1{10}r_{i}}.$
\end{itemize}
Then, by Remark~\ref{C0-or-C1} following the proof of Lemma~\ref{l.nstep-deri-estimate} in Section~\ref{b} in case of $C^0$ estimate, we get as functions on $I_{i+1}$,
$$
\|s_{r^+_{i+1}}-s_{r^+_i}\|<C\l_{N+i}^{-\frac32r_i}.
$$
Similarly, we get $\|u_{r^-_{i+1}}-u_{r^-_i}\|<C\l_{N+i}^{-\frac32r_i}$. Thus, we get $\|g_{i+2}-g_{i+1}\|_{I_{i+1}}<C\l_{N+i}^{-\frac32r_i}$. This clearly implies that
$$
\|g_{i+2}(x)\|>C\l_{N+i}^{-\frac1{10}r_i} \mbox{ for all } x\in I_{i+1}.
$$

Finally, whenever we are in case $(i+2)_\III$, let $c'_{i+2, 1}$ and $c'_{i+2, 2}$ be the two extra minimal points of $g_{i+2}$. We need to show that if $|g_{i+2}(c_{i+2,j})|<C\l_{N+i+1}^{-\frac1{10}r_{i+1}}$ for $j=1$ or $j=2$, then
      $$
      |c_{i+2,1}+k\a-c'_{i+2,1}|,\ |c_{i+2,2}-k\a-c'_{i+2,2}|<C\l_{N+i+1}^{-\frac1{30} r_{i+1}}
      $$
for some $1\le k< q_{N+i}$. As the proof of Lemma~\ref{critical-points-orbit}, this is just another an application of Lemma~\ref{l.almost-inv-s-direction} and Corollary~\ref{c.almost-inv-u-direction} to the following
$$
A_{r^+_{i+1}}(x+k\a)\cdot A_{k}(x)\cdot A_{r^-_{i+1}}(x-r^-_{i+1}\a).
$$
Then everything follows the same proof of Lemma~\ref{critical-points-orbit}. This concludes the proof of our induction.
\begin{remark}\label{r.uh-stopping}
This above inequality essentially shows that $|g_{n+1}(x)|>C\l_{N+i}^{-\frac1{10}r_i}$ for all $n\ge i$ and for all $x\in I_n$. Since it is clear that the $\|g_{n+1}-g_{n}\|_{I_n}<C\l_{N+n-1}^{-\frac32r_{n-1}}$, $n\ge i+1$,  we  have
$$
C\sum^{\infty}_{n=i+1}\l_{N+n-1}^{-\frac32r_{n-1}}\le C\l_{N+i}^{-\frac32r_{i}}.
$$
Thus, we will not need to worry about these parameters in the future since Lemma~\ref{l.nstep-deri-estimate} will be applicable for all $n\ge i$. In fact, the reason we have this property is that, by Lemma 11 of \cite{zhang}, it is not difficult to see that $(\a,A)\in\mathcal U\mathcal H$ for these parameters. Hence there must exist a $N_0\in\Z^+$ and $\gamma_0>0$ such that $|s_n(x)-u_n(x)|>\gamma_0$ for all $n\ge N_0$ and for all $x\in\R/\Z$. In particular, whenever we have $|g_{i+1}(x)|>C\l_{N+i}^{-\frac1{10}r_i}$ for all $x\in I_i$ for some $i$, then the induction can basically stop at this step $i$. Thus, in our induction, we essentially only need to deal with the case $|g_{i+1}(x)|<C\l_{N+i}^{-\frac1{10}r_i}$ for some $x\in I_i$.
\end{remark}

\begin{remark}\label{r.no-new-resonance}
As we see in the proof of induction, the case $(i+2)_\III$ may come from $(i+1)_\III$. In fact, it is possible that the strong resonance occurs at step $i_0\ll i$ so that $(i+1)_\III$ comes from $(i_0)_\III$. Then, as pointed out in Remark~\ref{r.type-of-functions}, $g_{i+1}:I_{i,j}\rightarrow\R\PP^1$, $j=1,2$, may become type $\II$ and we may have $c_{i+1,j_1}=c'_{i+1,j_2}$, $j_1\neq j_2\in\{1,2\}$ and $\l_N^{-k}>\frac{1}{2^{i+1}q_{N+i}^{2\tau}}$. Thus, at step $(i+2)$, one may worry that if there is resonance between two type $\II$ functions, which may lead to some new bifurcation that is not included in our induction. However, the last part of Theorem~\ref{t.iteration} and Remark~\ref{r.uh-stopping} imply that either we have $|g_{n}(c_{n,j})|>C\l_{N+i}^{-\frac1{10}r_{i}}$ for all $n\ge i+1$, hence, we do not need to worry about this case in the future. Or we have
$$
|c'_{i+1,1}-k\a+c_{i+1,1}|,\ |c'_{i+1,2}+k\a-c_{i+1,2}|<C\l_{N+i}^{-\frac1{30}r_{i}}\ll |I_{i+1,j}|,
$$
which implies that there will be no new type of bifurcation within time $q_{N+i}$.
\end{remark}

\section{\bf{Uniform Positivity: Proof of Theorem~\ref{t.main}}}\label{sec-positivity}
We are going to do it by induction. Let $D_0=\R/\Z$ and $k'_0=0$. For $i\ge1$, let
$$
k'_i=\begin{cases}0, &\mbox{ for cases } (i+1)_\I,\ (i+1)_\II;\\ k_i, &\mbox{ for the case } (i+1)_\III,\end{cases}
$$
where $0<k_i<q_{N+i-1}$ is the time such that $(I_{i,1}+k_i\a)\cap I_{i,2}\neq\varnothing$. Let
$$
B^+_i=\left(\bigcup^{q_{N+i}-1}_{l=1} R_\a^{-l}I_{i+1,1}\right)\cup \left(\bigcup^{q_{N+i}+k'_i-1}_{l=1}R_\a^{-l}I_{i+1,2}\right),
$$
$$
B^-_i=\left(\bigcup^{q_{N+i}-1}_{l=1} R_\a^{l}I_{i+1,2}\right)\cup \left(\bigcup^{q_{N+i}+k'_i-1}_{l=1}R_\a^{l}I_{i+1,1}\right)\mbox{ and }
$$
$$
D_{i+1}=D_i\setminus(B_i^+\cup B_i^-).
$$
It is easy to see that $D_\infty=\bigcap_{i\ge1}D_i$ is of positive Lebesgue measure. Indeed, since $\{q_n\}_{n\ge 0}$ grows at least exponentially fast, by choosing $N$ large, we get
$$
\mathrm{Leb}(D_\infty)>1-\sum_{n\ge N}\frac{6q_n}{q^{2\tau}_n}>1-\sum_{n\ge N}\frac{1}{q^{2}_n}>0.
$$
Set $q_{N-1}=1$, $t^\pm_0=1$. Inductively, we will show that for all $i\ge1$ and $x\in D_i$, let $t^\pm_i=t^\pm_i(x)\ge q_{N+i-1}$ be the first nonzero time such that $x\pm t^\pm_i\a\in I_i$, then
\beq\label{LE-stepi}
\frac1{t^\pm_i}\log \|A_{\pm t^\pm_i}(x)\|\ge\log\l_{N+i};
\eeq
\beq\label{1st-chain1}
x\in I_i \mbox{ or } \left|s_{t^+_{i-1}}(x)-u_{t^-_{i-1}}(x)\right|>c(2^{-i}q_{N+i-1}^{-2\tau})^3;
\eeq
\beq\label{1st-chain2}
\left|s_{t^+_i}(x)-s_{t^+_{i-1}}(x)\right|,\ \left|u_{t^-_i}(x)-u_{t^-_{i-1}}(x)\right|<C\l_{N+i-1}^{-q_{N+i-2}}.
\eeq
Let $\l_\infty=\lim_{i\rightarrow\infty}\l_{N+i}$. It is easy to see that
$$
\log\l_\infty>(1-\e)\log\l_N=(1-\e)\log\l
$$
for $N$ sufficiently large. Theorem~\ref{t.main} is immediate since we then have
$$
\lim_{n\rightarrow\infty}\frac1n\log\|A_n(x)\|=\limsup_{n\rightarrow\infty}\frac1n\log\|A_n(x)\|\ge\log\l_\infty \mbox{ for all } x\in D_\infty.
$$
Now to show (\ref{LE-stepi})--(\ref{1st-chain2}), let us start with $i=1$. It follows immediately from (\ref{s1-c2closeto-sn}) and (\ref{norm1-growth-deri}). Because in this case, we have $x+l\a\notin I_1$ for all $l\in [-t^-_1-1,-1]\cup [1,t^+_1-1]$.

Assume (\ref{LE-stepi})--(\ref{1st-chain2}) holds for $i=n$ and we want to show the case for $n+1$. Fix a
$$
x\in D_{n+1}=D_n\setminus(B_n^+\cup B_n^-).
$$
Let us first consider (\ref{LE-stepi}). We only need to consider the forward sequence since the argument for backward sequence is similar. Now if $t^+_{n+1}(x)=t^+_{n}(x)$, then it follows from the induction assumption. Otherwise, let $j_0=0$ and $0<j_m<t^+_{n+1}$, $1\le m\le p$ be all the times such that
\begin{center}
$x+j_m\a\in I_n$, $1\le m\le p$; $j_{m+1}-j_{m}=r^+_n(x+j_{m}\a)\ge q_{N+n-1}$, $1\le m\le p-1$.
\end{center}
Note $j_1=t^+_n$. Then by definition, we have the following facts.

\vskip 0.3cm
\noindent  If $k_n'=0$, then we claim that
      \beq\label{apply-to-c1-1}
      x+j_m\a\notin I_{n+1},\ 1\le m\le p.
      \eeq
In fact, if $j_m<q_{N+n}$, then it follows from the definition of $D_{n+1}$. If $q_{N+n}\le j_m<t^+_{n+1}$, then it follows from the definition of $t_{n+1}$.
\vskip0.3cm
\noindent If $k'_n=k_n$, then for each $1\le m\le p-1$, we claim that
      \beq\label{apply-to-c1-2}
        x+j_m\a\in I_{n,1}\setminus \left[I_{n+1,1}\cup (I_{n+1,2}-k_n\a)\right] \mbox{ or } x+j_m\a\in I_{n,2}\setminus I_{n+1,2}.
      \eeq
Indeed, for the `either' part, if $j_m<q_{N+n}$, then it follows from the definition of $D_{n+1}$. If $j_m\ge q_{N+n}$, then $q_{N+n}<j_m+k_n<j_{m+1}<t_{n+1}$. By definition of $t_{n+1}$, we must have that $x+(j_m+k_n)\a\notin I_{n+1,2}$. For the `or' part of the claim, if $j_m<q_{N+n}$, then again it follows from the definition of $D_{n+1}$. If $j_m\ge q_{N+n}$, then it  follows from the fact that $j_m<t_{n+1}$ and the definition of $t_{n+1}$. For the time $j_p$, then either it falls into the two cases above; or $j_p+k_n=t_{n+1}$. Note in the latter case, we have $j_p-j_{p-1}>q_{N+n-1}^2\gg q_{N+n-1}>k_n.$ Then it is easy to see that
$$
\|A_{t_{n+1}-j_{p-1}}(x+j_{p-1}\a)\|\ge \l_{N+n}^{t_{n+1}-j_{p-1}} \mbox{ and }
$$
$$
|s_{t_{n+1}-j_{p-1}}(x+j_{p-1})-s_{j_p-j_{p-1}}(x+j_{p-1})|<C\l_{N+n}^{-r^+_n}.
$$
Note in this case, either $c'_{n+1,2}=c_{n+1,1}$, then we do not really need to consider $I_{n+1,2}-k_n\a$ in the subsequent estimate; or by (\ref{i+1-orbit-critical-points}), $I_{n+1,2}-k_n\a$ is almost centered around $c'_{n+1,2}$. Also, we have either $k_n\ge r_{n-1}\ge q_{N+n-2}$ and $l_{k_n}\gg |I_{n+1,j}|$; or $k_n$ comes from resonance of previous step $s<n$. Then we have $|c'_{s,2}-c_{s,1}|<\frac{1}{2^nq_{N+n-1}}<\frac{1}{4}|I_{s,j}|$. In any case, Corollary~\ref{c.distance-away-from-zero} will be applicable in the subsequent estimate.

\vskip0.3cm
Next we claim that we have
\beq\label{1st-chain4}
|u_{t^+_n}(x+t^+_n\a)-u_{r^-_n}(x+t^+_n\a)|\le c\l_{N+n}^{-q_{N+n-1}}.
\eeq
Indeed, if $x\in I_n$, then $j_1=r^-_n$ and we have nothing to say. Otherwise, by (\ref{1st-chain1}) and (\ref{1st-chain2}) in case of $i=n$ and the fact $\l_{N+i-1}^{-q_{N+i-2}}\ll c(2^{-i}q_{N+i-1}^{-2\tau})^3$, we clearly have
\beq\label{1st-chain3}
\left|s_{t^+_{n}}(x)-u_{t^-_{n}}(x)\right|>c(2^{-i}q_{N+n-1}^{-2\tau})^3.
\eeq
Now we consider the product
\beq\label{1st-chain8}
A_{-r^-_n}(x+t^+_n\a)=A_{-t^-_n}(x)\cdot A_{-t^+_n}(x+t^+_n\a)
\eeq
Note that we have
$$
s_{t^+_{n}}(x)=u[A_{-t^+_n}(x+t^+_n\a)],\ u_{t^-_{n}}(x)=s[A_{-t^-_n}(x)],\mbox{ and }
$$
$$
u_{t^+_n}(x+t^+_n\a)=s[A_{-t^+_n}(x+t^+_n\a)],\ u_{r^-_n}(x+t^+_n\a)=s[A_{-r^-_n}(x+t^+_n\a)].
$$
By (\ref{LE-stepi}) in the case of $i=n$, we have
\beq\label{1st-chain7}
\|A_{-t^-_n}(x)\|,\ \|A_{-t^+_n}(x+t^+_n\a)\|=\|A_{t^+_n}(x)\|>\l_{N+n}^{q_{N+n-1}}\gg C(2^{i}q_{N+n-1}^{2\tau})^3.
\eeq

Thus, by (\ref{1st-chain3}) and (\ref{1st-chain7}), we can apply Lemma~\ref{l.1step-deri-estimate} to (\ref{1st-chain8}). Then by the first estimate of (\ref{1step-deri-estimate4}), we obtain (\ref{1st-chain4}).

Now we are ready to show (\ref{LE-stepi}) for $i=n+1$. We need to concatenate either the sequence
\beq\label{1st-chain5}
A_{j_1}(x),\ldots, A_{j_{m+1}-j_m}(x+j_m\a),\ldots, A_{t_{n+1}-j_{p-1}}(x+j_{p-1}\a)
\eeq
if $j_p+k_n=t^+_{n+1}$; or the sequence
\beq\label{1st-chain6}
A_{j_1}(x),\ldots, A_{j_{m+1}-j_m}(x+j_m\a),\ldots, A_{t_{n+1}-j_{p}}(x+j_{p}\a),
\eeq
otherwise. In any case, we get the following facts.

\vskip0.2cm
$\|A_{j_{m+1}-j_m}(x+j_m\a)\|\ge \l_{N+n}^{r_n}\ge \l_{N+n}^{q_{N+n-1}}$. For $m=1$, this follows from the induction assumption (\ref{LE-stepi}) in case of $i=n$. For $1<m\le p$, it is by (\ref{norm-i+2}) of Theorem~\ref{t.iteration}.

\vskip0.2cm
$|s_{r^+_n}(x+j_m\a)-u_{r^-_n}(x+j_m\a)|=c|g_{n+1}(x+j_m\a)|>c|I_{n+1}|^3=c(2^{n+1}q_{N+n}^{2\tau})^{-3}$. For $m=1$, this follows from (\ref{1st-chain4}), (\ref{apply-to-c1-1}), (\ref{apply-to-c1-2}), Theorem~\ref{t.iteration} and Corollary~\ref{c.distance-away-from-zero}. For $2\le m\le p$, it follows directly from the (\ref{apply-to-c1-1}), (\ref{apply-to-c1-2}), Theorem~\ref{t.iteration} and Corollary~\ref{c.distance-away-from-zero}. If we set $c(2^{n+1}q_{N+n}^{2\tau})^{-3}=\l_{N+n}^{-\eta q_{N+n-1}}$, then it is clear that $\eta<\eta_{N+n-1}=C\frac{\log q_{N+n}}{q_{N+n-1}}$.

\vskip0.2cm
Thus, by Remark~\ref{r.returntime-not-large}, Lemma~\ref{l.nstep-deri-estimate} can be applied to the sequence (\ref{1st-chain5}) or (\ref{1st-chain6}). Then, by (\ref{nstep-deri-estimate3}) of Lemma~\ref{l.nstep-deri-estimate}, we get
\begin{align*}
\|A_{t_{n+1}}(x)\|> \CA_{1} \mbox{ or } \CA_{2}
\ge \l_{N+n}^{t_{n+1}(1-\eta_{N+n-1})}=\l_{N+n+1}^{t_{n+1}},
\end{align*}
where $\CA_{1}=\left(\|A_{t_{n+1}-j_{p}}(x+j_{p}\a)\|\cdot\prod^p_{m=1}\|A_{j_m-j_{m-1}}(x+j_{m-1}\a)\|\right)^{1-\eta_{N+n-1}}$
and $\CA_{2}=\left(\|A_{t_{n+1}-j_{p-1}}(x+j_{p-1}\a)\|\cdot\prod^{p-1}_{m=1}\|A_{j_m-j_{m-1}}(x+j_{m-1}\a)\|\right)^{1-\eta_{N+n-1}}.$ This is (\ref{LE-stepi}) for $i=n+1$.

Now we need to show (\ref{1st-chain1}) and (\ref{1st-chain2}) for $i=n+1$. For (\ref{1st-chain1}), if $x\in I_n$, then it follows from (\ref{1st-chain3}) and the fact $q_{N+n}>q_{N+n-1}$. If $x\in I_n\setminus I_{n+1}$, then it follows from Theorem~\ref{t.iteration} and Corollary~\ref{c.distance-away-from-zero} since $t^\pm_n=r^\pm_n$ in this case. If $x\in I_{n+1}$, we have nothing need to say.

For (\ref{1st-chain2}), again it suffices to consider the forward sequence which corresponds to the estimate $|s_{t^+_{n+1}}(x)-s_{t^+_{n}}(x)|<C\l_{N+n}^{-q_{N+n-1}}$. As the proof of (\ref{LE-stepi}) for $i=n+1$,  this is actually just another consequence (\ref{nstep-deri-estimate1}) of Lemma~\ref{l.nstep-deri-estimate} when we apply the induction fact $\|A_{t^+_n}(x)\|>\l_{N+n}^{q_{N+n-1}}$ to the sequence (\ref{1st-chain5}) or (\ref{1st-chain6}). This concludes the proof of Main Theorem.

\section{{\bf Large Deviation and Continuity: Proof of Theorem~\ref{t.continuity}}}\label{sec-continuity}
In this section we will first prove a version of large deviation theorem (LDT) for Lyapunov exponents. Then, combine with the so-called Avalanche Principle, our LDT will imply the weak H\"older continuity of Lyapunov exponents and integrated density of states (IDS) with respect to the energy. Let us start with the large deviation theorem.
\subsection{Large Deviation Theorem}
Again we only need to consider energies in a compact interval, e.g. $E\in [\l\inf v-C, \l\sup v+C]$. Thus for any $\e>0$, we have for large $\l$
\beq\label{uniform-bound}
\log \|A^{(E-\l v)}(x)\|<(1+\frac\e2)\l
\eeq
for all $E$ in question and all $x\in\R/\Z$.

Let $\frac{p_s}{q_s}$ be the $s$th continued fraction approximants of $\a$, which is from our Main Theorem. Note we have by \emph{Diophantine} condition
\beq\label{diophantine}
q_{s+1}<cq_{s}^{\tau-1},\ s\in \Z_+.
\eeq

We need the following proposition which is a standard result for \emph{Diophantine} translation on torus. For a simple proof, see \cite[Lemma 6]{aviladamanikzhang}.
\begin{prop}
For arbitrary fixed \emph{Diophantine} frequency, there is a polynomial $P=P(\a)\in \R[X]$
such that, for any interval $I\subset\R/\Z$ and for each $x\in\R/\Z$,
$$
x+\ell\a\in I \mbox{ for some } 0< \ell<P(|I|^{-1}).
$$
\end{prop}

Consider $I_n=I_{n,1}\cup I_{n,2}$ with $I_{n,j}=B(c_{n,j}, 2^{-n}q_{N+n-1}^{-2\tau})$, $n\ge 1$, $j=1,2$, which is from the induction Theorem~\ref{t.iteration}. Recall by Theorem~\ref{t.iteration}, there are associated times $r^\pm_n(x)$, $r_n=\min_{x\in I_n}{r^\pm_n(x)}\ge q_{N+n-1}$, such that the following holds.
\begin{itemize}
\item $\|A_\nu(x)\|\ge \l_{N+n}^{|\nu|}$ for each $x\in I_n$, $\nu=r^\pm_n(x)\mbox{ or } r_n$.
\item $|s_{r^+_n(x)}(x)-s_{r_n}(x)|$, $|u_{r^-_n(x)}(x)-u_{r_n}(x)|<\l_{N+n}^{-\frac32 r_n}$ for each $x\in I_n$ , see, e.g. Remark~\ref{r.small-error}.
\item $g_{n+1}:I_n\rightarrow\R\PP^1$ is of type $\I$, $\II$ or $\III$, where $g_{n+1}(x)=s_{r_n}(x)-u_{r_n}(x)$.
\end{itemize}

Recall $C_{n+1}=\{c_{n+1,1}, c_{n+1,2}\}$ is a set of minimal points of $g_{n+1}$, which is also from Theorem~\ref{t.iteration}.
Without loss of generality, let $R_n=P(|I_n|^{-1})= q_{N+n-1}^C$ for some constant $C$ depending only on $\a$. By definition, $R_n\ge \max_{x\in I_n}\{r^\pm_n(x)\}\ge r_n$.  Note that by (\ref{diophantine})
\beq\label{large-forward-time}
R_n^{2\tau}=q_{N+n-1}^{2\tau C}>q_{N+n}^{2C}=R_{n+1}^2>R_n^{2}.
\eeq
Let $\CI_n=B(C_{n+1}, e^{-\delta q_{N+n-1}})=B(c_{n+1,1}, e^{-\delta q_{N+n-1}})\bigcup B(c_{n+1,2}, e^{-\delta q_{N+n-1}})$,
\beq\label{poly-equidistributed}
\CD_n=\bigcup_{\ell=1}^{R_n^{2\tau}}(\CI_n-\ell\a) \mbox{ and } \CB_n=\bigcup_{\ell=1}^{R_n}(\CD_n-\ell\a).
\eeq
Clearly, if we set $\sigma=\frac{1}{2\tau C}$, then for sufficiently large $n$ and for each $i\in [R_n^2, R_n^{2\tau}]$
\beq\label{bad-set-estimate}
\mathrm{Leb}(\CB_n)\le 4q_{N+n-1}^{(2\tau+1) C}e^{-\delta q_{N+n-1}}\le e^{-\frac12\delta q_{N+n-1}}\le e^{-\frac12\delta i^{\sigma}}.
\eeq

Now we claim the following. Let be $v,\a$ as in Theorem~\ref{t.main}. Then for all $\e>0$, there exists $\l_2=\l_2(v,\a,\e)$, $\delta=\delta(\e)>0$, $0<\sigma=\sigma(\a)<1$, and $n_0=n_0(\a,\e)\in\Z^+$, such that for each $\l>\l_2$, each $i\in [R_n^2, R_n^{2\tau}]$ with $n\ge n_0$, and each $x\in (\R/\Z)\setminus \CB_n$, it holds that
\beq\label{large-LE-mostx}
\|A_i(x)\|\ge \l_{N+n+1}^i\ge \l^{(1-\frac\e2)i}.
\eeq
The following LDT is an easy consequence of (\ref{large-LE-mostx}) and Theorem~\ref{t.main}.

\begin{theorem}\label{t.large-deviation}
Let $v$, $\a$ be as in Main Theorem. Then for each $\e>0$, there exists $\l_1=\l_1(v,\a,\e)$, $\delta=\delta(\e)>0$, $0<\sigma=\sigma(\a)<1$, and $i_0=i_0(\a,\e)\in\Z^+$, such that for each $\l>\l_1$ and each $i\ge i_0$, it holds that
\beq\label{large-deviation}
\mathrm{Leb}\{x\in \R/\Z| \left|\frac{1}{i}\log\|A_i(x)\|-L(E)\right|>\e\log\l\}<e^{-\frac12\delta i^\sigma}.
\eeq
\end{theorem}
\begin{proof}
Choose $\l_1$ so that for all $\l>\l_1$, we have (\ref{uniform-bound}), (\ref{large-LE-mostx}) and $L(E)>(1-\frac\e2)\log\l$. Let $\delta=\delta(\e)$, $\sigma=\sigma(\a)$, and $n_0=n_0(\a,\e)\in\Z^+$ be as in the statement preceding (\ref{large-LE-mostx}).

By (\ref{large-forward-time}), we clearly have $\bigcup_{n\ge n_0}[R_n^2, R_n^{2\tau}]=[R_{n_0},+\infty)$. Set $i_0=R_{n_0}^2$ and consider $i\ge i_0$. Assume $i\in [R_n^2, R_n^{2\tau}]$ for a some $n\ge n_0$. Then for each $x\in (\R/\Z)\setminus \CB_n$, we have
$$
(1-\frac\e2)\log\l-(1+\frac\e2)\log\l\le \frac1i\log \|A_i(x)\|-L(E)\le(1+\frac\e2)\log\l-(1-\frac\e2)\log\l.
$$
This clearly implies that
\beq\label{deviation}
\left|\frac1i\log \|A_i(x)\|-L(E)\right|\le\e\log\l.
\eeq
Thus $\left|\frac1i\log \|A_i(x)\|-L(E)\right|>\e\log\l$ implies that $x\in \CB_n$. This together with (\ref{bad-set-estimate}) clearly imply our Theorem~\ref{t.large-deviation}.
\end{proof}

Let us prove (\ref{large-LE-mostx}). Consider $x\in (\R/\Z)\setminus \CB_n$. Then by the choice of $R_n$, we have that $x_\ell=x+\ell\a\in I_n$ for some $0<\ell<R_n$. Then write $A_i(x)$ as
$$
A_i(x)=A_{i-\ell}(x_\ell)\cdot A_\ell(x).
$$
Now let us focus on $A_{i-\ell}(x_\ell)$ with $x_\ell\in I_n$. Let $j_0=0$ and $j_m\le i-\ell$, $1\le m\le p$, be the all the possible times such that
$$
j_{m}-j_{m-1}=r^+_n(x_\ell+j_{m-1}\a)\ge q_{N+n-1}.
$$
Note $i-\ell-j_p\le R_n$.

By our choice of $x$ and the definition of $r^+_n:I_n\rightarrow\Z^+$, it holds that
$$
x_\ell+j_m\a\in I_n\setminus\CI_n,\ 0\le \ell\le p.
$$
In other words, we always have $|x+j_m\a-C_{n+1}|>e^{-\delta q_{N+n-1}}$. Thus for a suitable choice of $\delta$, we are able to apply Lemma~\ref{l.nstep-deri-estimate} to concatenate the sequence
\beq\label{sequence-in-LDT}
A_{j_1}(x_\ell),\ldots, A_{j_{m+1}-j_m}(x_\ell+j_m\a),\ldots, A_{j_{p}-j_{p-1}}(x_\ell+j_{p-1}\a).
\eeq
For the given small $\e>0$, apply Theorem~\ref{t.iteration} to $\frac\e4$. Then the analysis above implies the following facts.
\vskip0.2cm
First, for $1\le m\le p$, it holds
\beq\label{large-norm-in-LDT}
\|A_{j_{m}-j_{m-1}}(x_\ell+j_{m-1}\a)\|\ge \l_{N+n}^{r_n}\ge \l_{N+n}^{q_{N+n-1}}\ge \l^{(1-\frac\e4)q_{N+n-1}}.
\eeq
\vskip0.2cm
Secondly, for $1\le m\le p$,
\beq\label{large-dist-to-critical-in-LDT}
|x+j_m\a-C_{n+1}|>e^{-\delta q_{N+n-1}}.
\eeq
Then we claim for $1\le m\le p-1$, it holds that
\beq\label{large-angle-in-LDT}
|s_{r_n}(x_\ell+j_m\a)-u_{r_n}(x_\ell+j_m\a)|=|g_{n+1}(x_\ell+j_m\a)|>e^{-3\delta q_{N+n-1}}.
\eeq
As the proof of Theorem~\ref{t.main}, in case $(n+1)_\I$ and $(n+1)_\II$, (\ref{large-angle-in-LDT}) follows directly from (\ref{large-dist-to-critical-in-LDT}) since there is no extra minimal points of $g_{n+1}:I_n\rightarrow\R\PP^1$. In the resonance case $(n+1)_\III$, we need to worry about the possibility that
\beq\label{resonance-in-LDT}
|x_\ell+j_m\a-c_{n+1,2}'|<ce^{-\delta q_{N+n-1}}.
\eeq
However, either $|c_{n+1,2}'-c_{n+1,1}|\ll e^{-\delta q_{N+n-1}}$ so that (\ref{resonance-in-LDT}) contradicts with (\ref{large-dist-to-critical-in-LDT}). Or by (\ref{i+2-orbit-critical-points}), for some $k<q_{N+n-1}$,
$$
|c_{n+1,2}'+k\a-c_{n+1,2}|<\l_{N+n}^{-\frac1{30}r_n}\ll e^{-\delta q_{N+n-1}},
$$
which together with (\ref{resonance-in-LDT}) implies
$$
|x_\ell+(j_m+k)\a-c_{n+1,2}|\le e^{-\delta q_{N+n-1}}.
$$
This again contradicts with our choice of $x\notin \CB_n$ since $j_m+k<j_p\le i-\ell$ for $1\le m\le p-1$.

Finally, by choosing $n$ large, we can of course assume that for each $1\le m\le p$,
\beq\label{time-lessthan-norm-in-LDT}
p<R_n^{2\tau}=q_{N+n-1}^{2\tau C}<\l^{\frac12 (1-\frac\e4)q_{N+n-1}}<\|A_{j_{m}-j_{m-1}}(x_\ell+j_{m-1}\a)\|^{\frac12}.
\eeq

Now by suitable choice of $\delta$, e.g. $\delta<c\e$, we may assume
$$
e^{3\delta q_{N+n-1}}< \l^{\frac\e 4(1-\frac\e4)q_{N+n-1}}.
$$
Thus, by (\ref{large-norm-in-LDT}), (\ref{large-angle-in-LDT}) and (\ref{time-lessthan-norm-in-LDT}), we can apply Lemma~\ref{l.nstep-deri-estimate} and Remark~\ref{C0-or-C1} to the sequence (\ref{sequence-in-LDT}) and obtain
\begin{align*}
\|A_{j_p}(x_\ell)\|&\ge\left(\prod^p_{m=1}\|A_{j_m-j_{m-1}}(x_\ell+j_{m-1}\a)\|\right)^{1-\frac\e4}\\
&\ge \l^{(1-\frac\e4)^2j_p}\\
&\ge \l^{(1-\frac\e2)j_p}.
\end{align*}
Notice we have $\ell<R_n$, $i-\ell-j_p<R_n$ and $i\ge R_n^2$. Hence,
$$
j_p\ge i-\ell-R_n>i-2R_n.
$$
By choosing $n$ large, we also get $\frac{4R_n}{i}<\frac{4}{R_n}<\frac\e2$. Thus, it holds that
\begin{align*}
\|A_i(x)\|&=\|A_{i-\ell-j_p}(x_\ell+j_p\a)\cdot A_{j_p}(x_\ell)\cdot A_\ell(x)\|\\
&\ge \|A_{i-\ell-j_p}(x_\ell+j_p\a)\|^{-1}\cdot \|A_{j_p}(x_\ell)\|\cdot \|A_\ell(x)\|^{-1}\\
&\ge \l^{-2R_n}\cdot \l^{(1-\frac\e2)j_p}\\
&\ge \l^{-2R_n}\cdot \l^{(1-\frac\e2)(i-2R_n)}\\
&\ge \l^{(1-\frac\e2)i}\cdot \l^{-4R_n}\\
&\ge \l^{(1-\e)i},
\end{align*}
concluding the proof of (\ref{large-LE-mostx}).

\subsection{Continuity of Lyapunov exponents}
Once we have large deviation Theorem~\ref{t.large-deviation}, the weak H\"older continuity of Lyapunov exponents and of IDS follows essentailly from the approach that is developed by Goldstein-Schlag \cite{goldstein}. See also \cite{bourgainschlag}. For the purpose of completeness, we include the proof here. First, we state the following lemma from \cite{goldstein}, which is called ``Avalanche Principle''.

\begin{lemma}\label{l.avlanche-principle}
Let $E^{(1)},\ldots, E^{(n)}$ be a finite sequence in $\mathrm{SL}(2,\R)$ satisfying the following condition
\begin{align}\label{condition-AP}
&\min_{1\le j\le n}\|E^{(j)}\|\ge \mu\ge n,\\
&\max_{1\le j<n}\left|\log \|E^{(j+1)}\|+\log \|E^{(j)}-\log\|E^{(j+1)}E^{(j)}\|\right|<\frac12\log\mu
\end{align}
then
\beq\label{avalanche-principle}
\left|\log\|E^{(n)}\ldots E^{(1)}\|+\sum_{2}^{n-1}\log\|E^{(j)}\|-\sum_{1}^{n-1}\log\|E^{(j+1)}E^{(j)}\|\right|\le C\frac{n}{\mu}.
\eeq
\end{lemma}
\noindent See \cite[Proposition 2.2]{goldstein} for a proof of Lemma~\ref{l.avlanche-principle}. It's probably interesting to point out that there is some intrinsic  relation between the $C^0$ version of Lemma~\ref{l.nstep-deri-estimate} (see Remark \ref{C0-or-C1}) and Lemma~\ref{l.avlanche-principle}. In fact, though taking different point of views, both of them deal with long finite concatenation of $\mathrm{SL}(2,\R)$ matrices, and the conditions assumed in both lemmas are similar. Let
$$
L_n(E)=\frac1n\int_{x\in\R/\Z}\log\|A^{(E-v)}_n(x)\|dx.
$$
Then, combining Lemma~\ref{l.avlanche-principle} and Theorem~\ref{t.large-deviation}, we get the following Lemma. Let us fix some $\e$ small. For instance, $\e=\frac1{100}$ will be enough for our purpose. Then we may also replace $\delta$ in Theorem~\ref{t.large-deviation} by $c$ since we may set $\delta=c\e=\frac c{100}$.

\begin{lemma}\label{l.well-approximate-LE}
Let $v,\a,\l$ be as in Theorem~\ref{t.large-deviation}, then for all $n\in\Z_+$ large enough
\beq\label{well-approximate-LE}
|L(E)+L_n(E)-2L_{2n}(E)|<Ce^{-cn^\sigma},
\eeq
where $c,C$ depend on $v,\a,\l$, and $\sigma$ on $\a$.
\end{lemma}

\begin{proof}
Apply Lemma~\ref{l.avlanche-principle} to
$$
n\le e^{\frac\delta{10}\ell^\sigma},\ E^{(j)}=A_\ell(x+(j-1)\ell\a).
$$
From (\ref{large-deviation}), with $i=\ell$, there is an exceptional set $\Omega\subset\R/\Z$ with
$$
\mathrm{Leb}(\Omega)<e^{-\frac14\delta\ell^\sigma}
$$
such that if $x\notin\Omega$ and $j=0,\ldots,n$ then
$$
(1-\e)\log\l<\frac1\ell\log\|A_\ell(x+j\ell\a)\|,\ \frac1{2\ell}\log\|A_{2\ell}(x+j\ell\a)\|<\log\l.
$$
Hence for $x\notin\Omega$
$$
\left|\log \|E^{(j+1)}\|+\log \|E^{(j)}-\log\|E^{(j+1)}E^{(j)}\|\right|\le 2\e\log\l<\frac{1-\e}{2}\log\l.
$$
Thus taking
$$
\mu=e^{\frac{(1-\e)\ell}{2}\l}
$$
condition~(\ref{condition-AP}), (81) of Lemma~\ref{l.avlanche-principle} are clearly fulfilled. For $x\notin\Omega$ the conclusion (\ref{avalanche-principle}) is that
$$
\left|\log\|A_{\ell n}(x)\|+\sum_{2}^{n-1}\log\|A_{\ell}(x+(j-1)\ell\a)\|-\sum_{1}^{n-1}\log\|A_{2\ell}(x+(j-1)\ell\a)\| \right|\le C\frac{n}{\mu}.
$$

Divide the above inequality by $\ell n$ and integrate it in $x\in\R/\Z$. Splitting the integration as $(\R/\Z)\setminus\Omega$ and $\Omega$, we get
$$
\left|L_{\ell n}(E)+\frac{n-2}{n}L_\ell(E)-\frac{2(n-1)}{n}L_{2\ell}(E)\right|<C\left(\mu^{-1}\ell^{-1}+\mathrm{Leb}(\Omega)\right)<Ce^{-\frac14\delta\ell^\sigma}.
$$
Note here $C$ depends on $\l$. Hence we obtain
\beq\label{avalanche-consequence1}
\left|L_{\ell n}(E)+L_\ell(E)-2L_{2\ell}(E)\right|<Ce^{-\frac14\delta\ell^\sigma}.
\eeq
Applying (\ref{avalanche-consequence1}) to $n=n_1=\ell$ large enough and $\log n_{2}\sim n_{1}^\sigma$ yields
\beq\label{avalanche-consequence2}
\left|L_{n_2}(E)+L_{n_1}(E)-2L_{2n_1}(E)\right|<Ce^{-\frac14\delta n_{1}^\sigma}.
\eeq
Applying (\ref{avalanche-consequence1}) to $n=n_1=\ell$ and $2n_2$ yields
$$
\left|L_{2n_2}(E)+L_{n_1}(E)-2L_{2n_1}(E)\right|<Ce^{-\frac14\delta n_{1}^\sigma}.
$$
Therefore, we get
\beq\label{avalanche-consequence3}
|L_{2n_2}(E)-L_{n_2}(E)|<Ce^{-\frac14\delta n_{1}^\sigma}.
\eeq
Clearly, in (\ref{avalanche-consequence2}) and (\ref{avalanche-consequence3}), we may replace $n_1$ and $n_2$ by any $n_s$ and $\log n_{s+1}\sim n_s^{\sigma}$. Thus we obtain
\begin{align*}
|L(E)-L_{n_2}(E)|&\le \sum_{s\ge2}|L_{n_{s+1}}(E)-L_{n_{s}}(E)|\\
&\le 2\sum_{s\ge2}\left(|L_{2n_{s}}(E)-L_{n_{s}}(E)|+Ce^{-\frac14\delta n_{s}^\sigma}\right)\\
&< 4\sum_{s\ge1}Ce^{-\frac14\delta n_{s}^\sigma}<Ce^{-\frac14\delta n^\sigma}.
\end{align*}
Thus, replacing $L_{n_2}(E)$ by $L(E)$ in (\ref{avalanche-consequence2}) yields (\ref{well-approximate-LE}).
\end{proof}

Then we are already to prove Theorem~\ref{t.continuity}

\begin{proof}({\bf Proof of Theorem~\ref{t.continuity}})
Clearly, we only need to show (\ref{log-holder}) for $E'$ and $E$ in the statement of Theorem~\ref{t.continuity} that are sufficiently close to each other. It is straightforward computation to see that $|L_n(E)-L_n(E')|<C^n|E-E'|$. So by (\ref{well-approximate-LE}), for all large $n$
$$
|L(E)-L(E')|<C^n|E-E'|+Ce^{-cn^\sigma}.
$$
Then, by choosing
$$
n=\left[\frac1{2\log C}\log\frac1{|E-E'|}\right],
$$
we clearly get Lyapunov exponents part of (\ref{log-holder})
\beq\label{log-holder-LE}
|L(E)-L(E')|<Ce^{-c(\log|E_1-E_2|^{-1})^\sigma}.
\eeq
This says that Lyapunov exponents is weak H\"older continuous with respect to the energy. Now for the IDS part of (\ref{log-holder}), by (\ref{log-holder-LE}) and the discussion following (\ref{thouless}), we obtain for $E,E'\in\R$,
$$
|N(E)-N(E')|<Ce^{-c(\log|E-E'|^{-1})^\sigma},
$$
concluding the proof of Theorem~\ref{t.continuity}.
\end{proof}

\appendix

\section{{\bf Proof of Lemma~\ref{l.reduce-form}--\ref{l.s-deri-resonance}}}\label{a}

\subsection{Polar decomposition of Schr\"odinger cocycles: Proof of Lemma~\ref{l.reduce-form} }\label{a.1}
\ \ \
In this subsection, we do the polar decomposition of the Schr\"odinger cocycles to reduce it to the form (\ref{polar-decom}). It is basically from Section 3.2, 4.2 and 5.3 of \cite{zhang}.

For $B\in \mathrm{SL}(2,\mathbb R)$, it is a standard result that we can decompose it as $B=U_1\sqrt{B^tB}$, where $U_1\in \mathrm{SO}(2,\mathbb R)$ and $\sqrt{B^tB}$ is a positive symmetric matrix. We can further decompose $\sqrt{B^tB}$ as $\sqrt{B^tB}=U_2\Lambda U_{2}^t$, where $U_2\in \mathrm{SO}(2,\mathbb R)$ and $\Lambda=\begin{pmatrix}\|B\|&0\\ 0&\|B\|^{-1}\end{pmatrix}$, thus $B=U_1U_2\Lambda U_2^t$.

Consider a map $B\in C^r(\mathbb R/\mathbb Z, SL(2,\mathbb R))$ for some $r\ge1$. Then, it can be decomposed as
$$
B(x)=U_1(x)U_2(x)\Lambda(x) U_2^t(x).
$$
By Lemma 10 of \cite{zhang}, $U_1(x)$, $U_2(x)$ and $\Lambda(x)$ are $C^r$ in $x$ as long
as $B(x)$ does not touch $\mathrm{SO}(2,\mathbb R)$. Hence, we have
$$
(U_1U_2)^t(x)B(x)(U_1U_2)(x-\alpha)=\Lambda(x)U(x),
$$
where $U(x)=U_2^t(x)(U_1U_2)(x-\alpha)\in \mathrm{SO}(2,\mathbb R)$. Let $c(x,t)$ be the upper left element of $U(x,t)$. \\

Now let us come back to the Schr\"odinger cocycles, we first use a simple trick to avoid that $A^{(E-\l v)}$ can always touch $\mathrm{SO}(2,\mathbb R)$ for $t=\frac{E}{\l}\in v(\mathbb R/\mathbb Z)$, which leads to the discontinuity of the polar decomposition. We instead consider $A^{(t,\l)}=TA^{(E-\l v))}T^{-1}$, where
$$
T=\begin{pmatrix}\sqrt{\lambda}^{-1}& 0\\0& \sqrt{\lambda}\end{pmatrix}
$$
This obviously does not change the dynamics. Thus
$$
A(x)=A^{(t,\lambda)}(x)=\begin{pmatrix}\lambda[t-v(x)]& -\lambda^{-1}\\\lambda & 0\end{pmatrix}.
$$
Let $r(x,t)=t-v(x)$. Then $r(x,t)$ is uniformly bounded on $\mathbb R/\mathbb Z\times\CI$, where $\CI\subset\R$ is any compact interval. If we set
$$
a=a(x,t,\lambda)=r^2+1+\frac{1}{\lambda^4}+\sqrt{(r^2+1+\frac{1}{\lambda^4})^2-\frac{4}{\lambda^4}},
$$
then obviously $\left|\frac{\partial^ma}{\partial^mx}\right|^{\pm}$ are uniformly bounded for all $(x,t,\lambda)\in\mathbb R/\mathbb Z\times\CI\times[c, \infty)$ and $m=0,1,3$. Then a direct computation shows that $\|A\|=\lambda\sqrt{\frac{a}{2}}$. Thus, it is clear that (\ref{norm1-deri-control}) holds for large $\l$.

A direct computation shows
$$
U_2=\frac{1}{\sqrt{(a-\frac{2}{\lambda^4})^2+\frac{4}{\lambda^4}r(x)^2}}\begin{pmatrix} a-\frac{2}{\lambda^4}& \frac{2}{\lambda^2}r(x)\\ -\frac{2}{\lambda^2}r(x)& a-\frac{2}{\lambda^4} \end{pmatrix}
$$
For simplicity let
$$
f(x,t,\lambda)=\left(\sqrt{(a-\frac{2}{\lambda^4})^2+\frac{4}{\lambda^4}r(x)^2}\right)^{-1}.
$$
Thus we get that the corresponding upper-left element of $U$ is
$$
c(x,t,\lambda,\alpha)=c_4\left\{r(x-\alpha)-\frac{2r(x)}{\lambda^2a(x)}+\frac{2r(x-\alpha)}{\lambda^4a(x)}-\frac{4r(x)}{\lambda^6 a(x-\alpha)a(x)}\right\},
$$
where
$$
c_4=\sqrt{\frac{2}{a(x-\alpha)}}f(x)f(x-\alpha)a(x)a(x-\alpha).
$$
Hence, we have $c(x,t,\infty,\alpha)=\frac{t-v(x-\alpha)}{\sqrt{(t-v(x-\alpha))^2+1}}$. Furthermore, it is not difficult to see that for any fixed $\alpha$,
\begin{center}
$c(x,t,\lambda,\alpha)\rightarrow c(x,t,\infty,\alpha)$ in $C^2(\mathbb R/\mathbb Z\times\CI,\mathbb R)$ as $\lambda\rightarrow\infty$.
\end{center}
Indeed, it is easy to see this reduces to the convergence of $a(x,t,\l)$ to $a(x,t,\infty)$ in $C^2$ topology, which is immediate.

In the proof of Theorem~\ref{t.iteration}, it is clear that the only important thing about the matrix $U=R_{\theta(x)}$ is the $C^2$ shape of the function $\theta$, see, for example, Corollary~\ref{c.general}. Thus, we could replace $c(x,t,\l,\a)$ by $c(x,t,\infty,\a)$ for large $\l$ since they are sufficiently close in $C^2$ norm. After a translation, we can of course replace $v(x-\a)$ by $v(x)$. This completes the proof.

\subsection{Nonresonance case: proof of Lemma~\ref{l.ess-change}--\ref{l.nstep-deri-estimate}}\label{b}

In this subsection, we will prove Lemma~\ref{l.ess-change}--\ref{l.nstep-deri-estimate}.
\vskip 0.3cm
\noindent Let us start with Lemma~\ref{l.ess-change}.
\begin{proof}(\textbf{Proof of Lemma~\ref{l.ess-change}})
It suffices to consider $s(x)$ since all estimates for $u(x)$ follow from the same way. By the polar decomposition procedure we have that $s(x)=\frac{\pi}{2}+\theta_2(x)$, where $\theta_2(x)$ is the eigen-direction of $E^t(x)E(x)$ corresponding to the eigenvalue $\|E(x)\|^2$. For simplicity, let us omit the dependence on $x$ in the following computation. A direct computation show that
$$
E^tE=\begin{pmatrix}e_1^2e_2^2\cos^2\t+e_1^2e_2^{-2}\sin^2\t& (e_2^{-2}-e_2^{2})\sin\t\cos\t\\(e_2^{-2}-e_2^{2})\sin\t\cos\t& e_1^{-2}e_2^{-2}\cos^2\t+e_2^2e_1^{-2}\sin^2\t\end{pmatrix}.
$$

Let $a=e_1e_2$ and $b=\frac{e_1}{e_2}$. Then it is easily calculated that
$$
\tan s(x)=-\cot\theta_2=\frac{w}{u},\mbox{ where } w=2(e_2^{2}-e_2^{-2})\sin\t\cos\t\mbox{ and }
$$
$$
u=\sqrt{[(a^2+a^{-2}) \cos^2\t+(b^2+b^{-2})\sin^2\t]^2-4}-[(a^2-a^{-2}) \cos^2\t+(b^{2}-b^{-2})\sin^2\t]
$$
Let $U=[(a^2-a^{-2}) \cos^2\t+(b^{2}-b^{-2})\sin^2\t]$, then it is easy to see that
$$
u=\sqrt{U^2+w^2}-U.
$$
Thus it is straightforward that
$$
\tan s(x)=\frac{w}{\sqrt{U^2+w^2}-U}=\frac{\sqrt{U^2+w^2}+U}{w}
$$
To simplify the above formula, we divide it into the following cases.\\

\noindent $(\mathbf{a})$ $e_1>e_2\gg1$. Then obviously $U>0$. Without loss of generality, assume $w>0$. Then we have
$$
\tan s(x)=\frac{\sqrt{U^2+w^2}+U}{w}=\sqrt{\frac{U^2}{w^2}+1}+\frac{U}{w}.
$$
After dropping some small terms that are not important in all the estimates, we get
$$
\frac{U}{w}\approx\frac12(e_1^2\cot\t+e_1^2e_2^{-4}\tan\t)=f_1(x),
$$
which implies (\ref{s-e1-great-e2}).\\

\noindent $(\mathbf{b})$ $e_2>e_1\gg1$. Without loss of generality, assume $w>0$. Then let us first assume that $U\ge0$, which approximately implies that
$$
|a^2\cos^2\t|\ge b^{-2}\Rightarrow|\cos\t|\ge a^{-2}b^{-2}=e_1^{-2}.
$$
Then after dropping some small terms and by the same argument as in the case $(a)$, we get that
$$
\tan s(x)=\frac{\sqrt{U^2+w^2}+U}{w}=\sqrt{\frac{U^2}{w^2}+1}+\frac{U}{w} \mbox{ with } \frac{U}{w}\approx\frac12(e_1^2\cot\t-\frac{1}{e_1^2\cot\t}).
$$
It is clear that we have
$$
\sqrt{\frac{U^2}{w^2}+1}=\frac12(e_1^2\cot\t+\frac{1}{e_1^2\cot\t}).
$$
Thus we get (\ref{s-e2-great-e1}) in this case.\\

Now if $U<0$, which implies that $|\cot\t|<e_1^{-2}$, then we have
\begin{align*}
\tan s(x)&=\frac{w}{\sqrt{U^2+w^2}-U}=\frac1{\sqrt{\frac{U^2}{w^2}+1}+\frac{-U}{w}}\\ &\approx\frac{1}{\frac12(e_1^2\cot\t+\frac{1}{e_1^2\cot\t})-\frac12(e_1^2\cot\t-\frac{1}{e_1^2\cot\t})}\\
&=e_1^2\cot\t
\end{align*}
\vskip 0.2cm

\noindent $(\mathbf{c})$ $e_2=e_1\gg1$. Then a direct computation shows that
$$
\tan s(x),\ \cot u(x)=\frac{\sqrt{(e_1^2+e_1^{-2})^2+4\tan^2\t}+e_1^2+e_1^{-2}}{2\tan\t},
$$
which clearly implies (\ref{s-e1-equal-e2}) after dropping some small terms.

This completes the proof.
\end{proof}

\vskip 0.6cm
\noindent Next we prove Lemma \ref{l.1step-deri-estimate}.
\begin{proof}(\textbf{Proof of Lemma~\ref{l.1step-deri-estimate}})
It is enough to consider only $s$ since $u$ can be treated similarly. By our assumption, we obviously have $|\cot\t|>ce_0^{-\eta}$. Then by  (\ref{s-e2-great-e1})-(\ref{s-e1-equal-e2}) of Lemma~\ref{l.ess-change}, it is a straightforward calculation to see that $\tan s$ is always dominated by $e_1^2\cot\t$ for $C^2$ estimate. In other words, we have
$$
\left|\frac{d^m s}{dx^m}\right|<C\left|\frac{d^m\tan^{-1}(e_1^2\cot\t)}{dx^m}\right|\mbox{ for } m=0,1,2.
$$
 Thus, from the case $m=0$, we get
 \beq\label{1step-deri-estimate1}
 \left|s-\frac\pi2\right|_{I}<Ce_1^{-(2-\eta)}.
 \eeq
 From $m=1$, we get
$$
\left|\frac{ds}{dx}\right|<\frac{C}{1+e_1^4\cot^2\t}\left|2e_1e'_1\cot\t-e_1^2\t'\cot^2\t-e_1^2\t'\right|.
$$
By our assumption, it is easy to see that we get
$$
\left|\frac{ds}{dx}\right|<C\left|\frac{2e'_1}{e_1^3\cot\t}\right|+C\left|\frac{\t'}{e_1^2}\right|+C\left|\frac{\t'}{e_1^2\cot^2\t}\right|.
$$
Hence, we have
\beq\label{1step-deri-estimate2}
\left|\frac{ds}{dx}\right|_{I}<Ce_1^{-(2-2\eta)}+Ce_1^{-(2-\eta)}+Ce_1^{-(2-3\eta)}<Ce_1^{-(2-3\eta)}.
\eeq
For $m=2$, we get
$$
\left|\frac{d^2 s}{dx^2}\right|<C\left|\frac{d^2\tan^{-1}(e_1^2\cot\t)}{dx^2}\right|.
$$
Thus, we have
$$
\left|\frac{d^2 s}{dx^2}\right|<C\left|\frac{1}{1+(e_1^2\cot\t)^2 }\frac{d^2(e_1^2\cot\t)}{dx^2}-\frac{2e_1^2\cot\t}{[1+(e_1^2\cot\t)^2]^2}\left[\frac{d(e_1^2\cot\t)}{dx}\right]^2\right|.
$$
A direct computation shows that, after dropping some relatively small terms, we have
$$
\left|\frac{d^2s}{dx^2}\right|<C\left|\frac{2e_1''}{e_1^3\cot\t}+\frac{4e_1'\t'}{e_1^3\cot^2\t}-\frac{6(e_1')^2}{e_1^4\cot\t}+\frac{2(\t')^2}{e_1^2\cot^3\t}
-\frac{2\t''}{e_1^2\cot^2\t}\right|.
$$
Hence, we obtain
\beq\label{1step-deri-estimate3}
\left|\frac{d^2s}{dx^2}\right|_{I}<Ce_1^{-(2-3\eta)}+Ce_1^{-(2-4\eta)}+Ce_1^{(2-5\eta)}<Ce_1^{-(2-5\eta)}.
\eeq

Clearly, (\ref{1step-deri-estimate1})--(\ref{1step-deri-estimate3}) imply (\ref{1step-deri-estimate4}). Now let us consider $e_3(x)$. Without loss of generality, assume $\cos\t>0$. Note we have $\cos\t>e_1^{-\eta}$ since $|\t-\frac\pi2|>e_1^{-\eta}$. Then by our assumption, after dropping some small term, we have
$$
e_3=e_1e_2\cos\t.
$$
Hence, we have
$$
\frac{de_3}{dx}=e_1'e_2\cos\t+e_1e_2'\cos\t-e_1e_2\t'\sin\t.
$$
It is clear that it is dominated by the first two terms. Hence, we have
$$
\left|\frac{de_3}{dx}\right|<C|e_1e_2^{1+\eta}\cos\t|+C|e_1^{1+\eta}e_1\cos\t|<C(e_2e_1\cos\t)^{1+\eta}<Ce_3^{1+\eta}.
$$
Then it is easy to see that after dropping some small term, we have
\begin{align*}
\left|\frac{d^2e_3}{dx^2}\right|<&C|e_1e_2''\cos\t|+C|e_1''e_2\cos\t|+Ce_1'e_2'\cos\t\\<&Ce_1e_2^{1+2\eta}\cos\t+Ce_1^{1+2\eta}e_2\cos\t+
Ce_1^{1+\eta}e_2^{1+\eta}\cos\t\\ <&C(e_1e_2\cos\t)^{1+2\eta}<Ce_3^{1+2\eta}.
\end{align*}
This concludes the proof.
\end{proof}

\vskip 0.2cm
Now we are ready to prove Lemma~\ref{l.nstep-deri-estimate}. In fact, it is an easy Corollary of Lemma~\ref{l.1step-deri-estimate}.

\begin{proof}({\bf Proof of Lemma~\ref{l.nstep-deri-estimate}}). Recall in Lemma~\ref{l.nstep-deri-estimate}, we consider a sequence of maps
$$
E^{(\ell)}\in C^2(I,\mathrm{SL}(2,\mathbb R)),\ 0\le \ell\le n-1.
$$
We have $s^{(\ell)}=s(E^{(l)})$, $u^{(\ell)}=s[(E^{(l)})^{-1}]$, $\l_\ell=\|E^{(\ell)}\|$, $\Lambda^{(\ell)}=\begin{pmatrix}\l_\ell&0\\0&\l_\ell^{-1}\end{pmatrix}$ and $E^{(\ell)}=R_{u^{(\ell)}}\Lambda^{(\ell)}R_{\frac\pi2-s^{(\ell)}}.$

Also $E_k(x)=E^{(k-1)}(x)\cdots E^{(0)}(x)$, $1\le k\le n$,
$$
s_{k}=s(E_{k}),\ u_{k}=s(E_{k}),\ l_{k}=\|E_{k}\|,\ L_{k}=\begin{pmatrix}l_{k}&0\\0&\l_{k}^{-1}\end{pmatrix},
$$
$E_{k}=R_{u_{k}}L_{k}R_{\frac\pi2-s_{k}}$ and $\l'=\min_{0\le\ell\le n-1}\{\l_\ell\}$. Then we have the following condition: $n<C\l'^{\frac12}$ and for any $x\in I$, $m=1,2$ and $0\le\ell\le n-1,$
$$
\left|\frac{d^m\l_\ell}{dx^m}(x)\right|< C\l_\ell^{1+m\eta};\ \left|\frac{d^ms^{(\ell)}}{dx^m}\right|,\ \left|\frac{d^mu^{(\ell)}}{dx^m}\right|<C\l'^{\eta},\ |s^{(\ell)}-u^{(\ell-1)}|>c\l'^{-\eta}.
$$
We want to show that
\begin{itemize}
\item $\left\|u^{(n-1)}-u_{n}\right\|_{C^2}<C\l_{n-1}^{-(2-5\eta)},\ \left\|s^{(0)}-s_n\right\|_{C^2}<C\l_0^{-(2-5\eta)}$;
\item $\left|\frac{d^ml_n}{dx^m}(x)\right|< Cl_n^{1+m\eta}$;
\item $l_n>\left(\prod^{n-1}_{\ell=0}\l_\ell\right)^{1-\eta}$.
\end{itemize}
Now applying Lemma~\ref{l.1step-deri-estimate} and our assumptions for $\ell=0$, $1$ to $E_{2}=E^{(0)}\cdot E^{(1)}$, we get
$$
\left\|s_{1}-s_2\right\|_{C^2}<C\l_{0}^{-(2-5\eta)},\ \left\|u^{(1)}-u_2\right\|_{C^2}<C\l_{1}^{-(2-5\eta)},
$$
$$
\left|\frac{dl_{2}}{dx}\right|<l_{2}^{1+\eta},\mbox{ and } \left|\frac{d^2l_{2}}{dx^2}\right|<l_{2}^{1+2\eta}.
$$
Clearly, together with our assumption, $\left\|u^{(1)}-u_2\right\|_{C^2}<C\l_{1}^{-(2-5\eta)}$ implies
$$
\left|\frac{d^mu_{2}}{dx^m}\right|<C\l'^{\eta},\ m=1,2\mbox{ and }|u_{2}-s^{(2)}|>c\l'^{-\eta}
$$
Combined with the assumption in case $\ell=2$, we are able to move to the next step $E_{3}=E^{(2)}\cdot E_2$. Since $n<C\l'^{\frac12}$, by induction, we get
$$
\left\|s^{(1)}-s_{n}\right\|_{C^2}<C\sum^{n}_{j=1}L_j^{-(2-5\eta)}<CL_1^{-(2-5\eta)}=C\l_0^{-(2-5\eta)},
$$
$$
\left\|u^{(n-1)}-u_n\right\|_{C^2}<C\l_{n-1}^{-(2-5\eta)}\mbox{ and } \left|\frac{d^ml_{n}}{dx^m}(x)\right|< Cl_{n}^{1+m\eta},\ m=1,2
$$
The last estimate is straightforward. Because by our assumption, we clearly have
$$
l_k>c\l_k\cdot l_{k-1}|\cos(s^{(k)}-u_k)|>c\l_k\cdot l_{k-1}\l'^{-\eta},\ 2\le k\le n.
$$
Since $\l'$ is sufficiently large and growth super-exponentially fast in our induction, $c$ can be absorbed into $\l'^{-\eta}$. Hence, we obtain
$$
l_n>\prod^{n-1}_{\ell=0}(\l'^{-\eta}\l_\ell)>\left(\prod^{n-1}_{\ell=0}\l_\ell\right)^{1-\eta},
$$
which completes the proof.
\end{proof}

\vskip 0.3cm
\subsection{Resonance case: proof of Lemma~\ref{l.norm-deri}--\ref{l.s-deri-resonance}}\label{c}

Let us first consider Lemma~\ref{l.norm-deri}.
\begin{proof}({\bf Proof of Lemma~\ref{l.norm-deri}}). Note we have $E(x)=E_2(x)\cdot E_1(x)$, $e_j=\|E_j(x)\|$, $e_3(x)=\|E(x)\|$, $e_0=\min\{e_1,e_2\}$, $\t(x)=s[E_2(x)]-u[E_1^{-1}(x)]$ and $0<\eta\ll 1$. In addition, we have $e_1\le e_2^\b$ or $e_2\le e_1^\b$ for some $0<\beta\ll1$ and for each $x\in I$, $j,m=1,2$,
$$
\left|\frac{d^me_j}{dx^m}(x)\right|< Ce_j^{1+2\eta};\ \left|\frac{d^m\t}{dx^m}\right|<Ce_0^{\eta}.
$$
Then we want to show that (\ref{norm-deri2})--(\ref{norm-deri1}), which are the following. For $m=1,2$
\begin{itemize}
\item $\left|\frac{d^ms[E(x)]}{dx^m}\right|<Ce_1^{2+4\eta},\ \left|\frac{d^mu[E(x)]}{dx^m}\right|<Ce_3^{-\frac32} \mbox{ if } e_1\le e_2^\b;$
\item $\left|\frac{d^mu[E(x)]}{dx^m}\right|<Ce_2^{2+4\eta},\ \left|\frac{d^ms[E(x)]}{dx^m}\right|<Ce_3^{-\frac32} \mbox{ if } e_2\le e_1^\b;$
\item $\left|\frac{d^me_3}{dx^m}(x)\right|<Ce_3^{1+m\eta+2m\eta\b}$.
\end{itemize}
For first two estimate, which are (\ref{norm-deri2}) and (\ref{norm-deri3}), it again suffices to consider $s$ since $u$ can be done similarly. Note the difference between the current situation and the one for (\ref{1step-deri-estimate4}) is that, here it is possible that $\t=\frac\pi2$. Now for (\ref{norm-deri2}), we apply the first estimate of (\ref{s-e2-great-e1}). For $m=1$, similar to the proof of (\ref{1step-deri-estimate2}), we get
$$
\left|\frac{ds}{dx}\right|<\frac{C}{1+e_1^4\cot^2\t}\left|2e_1e'_1\cot\t-e_1^2\t'\cot^2\t-e_1^2\t'\right|<Ce_1^{2+\eta},
$$
where the worst case happens when $\t=\frac\pi2$, hence, $\cot\t=0$. For $m=2$, similar to the proof of (\ref{1step-deri-estimate3}), we get
\begin{align*}
\left|\frac{d^2 s}{dx^2}\right|&<C\left|\frac{1}{1+(e_1^2\cot\t)^2 }\frac{d^2(e_1^2\cot\t)}{dx^2}-\frac{2e_1^2\cot\t}{[1+(e_1^2\cot\t)^2]^2}\left[\frac{d(e_1^2\cot\t)}{dx}\right]^2\right|\\
&<C\left|\frac{e_1^6(\t')^2\cot\t}{1+e_1^8\cot^4\t}\right|_{\cot\t\approx e_1^{-2}}\\&<Ce_1^{4+2\eta}.
\end{align*}

For (\ref{norm-deri3}), we apply the first part of (\ref{s-e1-great-e2}). It not difficult to see that, for $C^2$ estimate, we could use the following
$$
\left|\frac{d^ms}{dx^m}\right|<C\left|\frac{d^m\tan^{-1}(f)}{dx^m}\right|,
$$
where $f(x)=e_1^2\cot\t+e_1^2e_2^{-4}\tan\t$. Note $|f|\ge e_1^{2}e_2^{-2}\ge e_1^{2(1-\b)}$. The worst case again happens when $\t=\frac\pi2$ or $\tan\t=\infty$. Note $e_3>e_1e_2^{-1}$ and $e_2\le e_1^\b$ for some $\b\ll1$. Then for $m=1$, we get
\begin{align*}
\left|\frac{ds}{dx}\right|&<C\left|\frac{d\tan^{-1}(f)}{dx}\right|=C\frac{|f'|}{1+f^2}\\&<Ce_1^{-2}e_2^4\cdot|\t'|<C(e_1e_2^{-1})^{\frac32}\\&<Ce_3^{-\frac32}.
\end{align*}
For $m=2$, we get
\begin{align*}
\left|\frac{d^2s}{dx^2}\right|&<C\left|\frac{d^2\tan^{-1}(f)}{dx^2}\right|=C\left|\frac{f''}{1+f^2}-\frac{2f(f')^2}{(1+f^2)^2}\right|\\&
<Ce'_1e_1^{-3}e_2^{4}\cdot |\t'|<C(e_1e_2^{-1})^{\frac32}\\&<Ce_3^{-\frac32}.
\end{align*}

Now let us consider the function $e_3(x)$. It is clear that $e_3^2(x)+e_3^{-2}(x)=\mathrm{tr}(E^tE)$, where $tr$ stands for trace. Since $e_3>e_2e_1^{-1}\gg 1$, after dropping some small term, we get
$$
e_3^2\approx e_1^2e_2^2\cos^2\t+e_1^{-2}e_2^2\sin^2\t.
$$

The difference between the case in question and the case in Lemma~\ref{l.1step-deri-estimate} is that it is possible that $\cos\t=0$. In fact, this is the the worst case in the sense that the derivatives of $e_3$ may get large with respect to $e_3$ itself. Thus we only need to consider the case that $e_3^2$ is dominated by the second term. So we must have $\sin\t\approx\pm1$. Without loss of generality, assume $\sin\t>0$. Thus, we have $e_3\approx e_1^{-1}e_2\sin\t$ and
$$
\frac{de_3}{dx}\approx \frac{e_1'}{e_1^2}e_2\sin\t+e_1^{-1}e_2'\sin\t+e_1^{-1}e_2\cdot\t'\cos\t.
$$
By our assumption, it is easy to see that $\frac{de_3}{dx}$ is dominated by the second term. Thus we have
$$
\left|\frac{de_3}{dx}\right|<Ce_1^{-1}e_2'\sin\t<Ce_2^{1+\eta-\b}.
$$
Now we need to find a $\gamma_1$ so that
$$
\left|\frac{de_3}{dx}\right|<Ce_2^{1+\eta-\beta}<(e_1^{-1}e_2)^{1+\gamma_1}=e_2^{(1-\b)(1+\gamma_1)}<Ce_3^{1+\gamma_1}.
$$
 Then it is easy to see that it is enough to choose $\gamma_1=\eta+2\eta\b$. Similarly, we get
$$
\left|\frac{d^2e_3}{dx^2}\right|<Ce_1^{-1}e_2''\sin\t<Ce_2^{1+2\eta-\b}.
$$
By choosing $\gamma_2=\eta+4\eta\b$, we have
$$
\left|\frac{d^2e_3}{dx^2}\right|<Ce_2^{1+2\eta-\b}<Ce_2^{(1-\b)(1+\gamma_2)}=C(e_1^{-1}e_2)^{(1+\gamma_2)}<Ce_3^{1+\gamma_2}.
$$
Clearly, this concludes the proof.
\end{proof}
\vskip 0.2cm
Finally, let us show Lemma~\ref{l.s-deri-resonance}.
\begin{proof}({\bf Proof of Lemma~\ref{l.s-deri-resonance}})
Recall we have
\beq\label{difference}
f=\tan^{-1}(l^2[\tan f_1(x)])-\frac\pi2+f_2,
\eeq
where $f_1$ is of type $\mathrm{I}_+$ and $f_2$ of type $\mathrm{II}_-$. Furthermore, $f_1(0)=0$ and $f_2(d)=0$ with $d\ge 0$.
A direct computation shows that $$
\frac{df}{dx}(x)=\frac{l^2+l^2\tan^2[f_1(x)]}{1+l^4\tan^2[f_{1}(x)]}\cdot\frac{df_{1}}{dx}(x)+\frac{df_{2}}{dx}(x).
$$
We may drop the term $l^2\tan^2[f_1(x)]$ since it produces an error of order at most $l^{-2}$. So we basically have
\beq\label{difference-deri1}
\frac{df}{dx}(x)=\frac{l^2}{1+l^4\tan^2(f_{1})}\cdot\frac{df_{1}}{dx}+\frac{df_{2}}{dx}.
\eeq
After dropping some small terms, we also get
\beq\label{difference-deri2}
\frac{d^2f}{dx^2}=\frac{l^2}{1+l^4\tan^2(f_{1})}\cdot\frac{d^2f_{1}}{dx^2}+\frac{d^2f_{2}}{dx^2} -\frac{l^6\tan(f_{1})}{1+l^8\tan^4(f_{1})}\cdot\left(\frac{df_{1}}{dx}\right)^2.
\eeq

(\ref{difference}), (\ref{difference-deri1}) and (\ref{difference-deri2}) clearly imply that, $\|f-f_2\|_{C^1}<Cl^{-\frac32}$ outside $I\setminus B(0, Cl^{-\frac14})$. In the following discussion, let $r_j$, $1\le j\le 9$ be numbers such that $cr^2\le r_j\le Cr^{-2}$. First, by (\ref{difference-deri1}), it is straightforward calculation to see that $f$ has two critical points, $x_3$ and $x_4$, such that
$$
x_3=-r_1l^{-1}\mbox{ and } x_4=r_2l^{-1}.
$$
\vskip 0.2cm
Let us first consider the case $d\ge\frac r3$. By (\ref{difference-deri1}), the following holds true:
$$
f \mbox{ increases from } r_3l^{-1}+cr^3-\pi \mbox{ to } cr^3-r_4l^{-1} \mbox{ on } [x_3,x_4].
$$
Thus there is a zero contained in $B(0, Cl^{-\frac34})$, say $x_1$. Hence, $|x_1|<Cl^{-\frac34}$. We must have $B(d,\frac r4)\cap B(0, Cl^{-\frac14})=\varnothing$ since $d\ge \frac r3$. Hence, there is a zero contained in $B(d,Cl^{-\frac 34})$, say $x_2$. This implies that
$$
B(x_2,\frac r4)\cap B(x_1, Cl^{-\frac14})=\varnothing \mbox{ and } \|f-f_2\|_{C^1}<Cl^{-\frac32}\mbox{ on }B(x_2,\frac r4).
$$
Now we have
$$
\|f-f_2\|_{C^1}<Cl^{-\frac32} \mbox{ on }I\setminus B(x_1, Cl^{-\frac14}), \mbox{ and }|f_2|>cr^3 \mbox{ on }I\setminus B(x_2,\frac r4).
$$
This clearly implies that
$$
|f(x)|>cr^3 \mbox{ for all } x\notin B(x_1,Cl^{-\frac12})\cup B(x_2,\frac r4).
$$

If $d<\frac r3$, similar to the case $d\ge \frac r3$, we get
$$
|f(x)|>cr^3 \mbox{ for all } x\notin B(0,\frac r6)\cup B(d,\frac r6).
$$
Now let us focus the set $B(0,\frac r6)\cup B(d,\frac r6)$. To find $x_{1}$ and $x_{2}$, note $f$ is strictly increasing on $[x_3,x_4]=[-r_1l^{-1}, r_2l^{-1}]$. Then, by (\ref{difference}), we have
\beq\label{bifurcation}
f(x)=\begin{cases}r_3l^{-1}+r_5(d+r_1l^{-1})-\pi,\ &x=-r_1l^{-1}\\ -r_4l^{-1}+r_6(d-r_2l^{-1}),\ &x=r_2l^{-1}.\end{cases}
\eeq
Note $f(Cl^{-2})=\tan^{-1}(Cr_7)-\frac\pi2+f_2(Cl^{-2})<0$. So to solve the possible equation $f(x)=0$, we only need to consider $x>Cl^{-2}$. Then we may write $f(x)$ as
$$
f(x)=\tan^{-1}(l^2[\tan f_1(x)])-\frac\pi2+f_2(x)=-cl^{-2}(r_8x)^{-1}+cr_9(d-x).
$$

If $d>c(l\sqrt{r_8r_9})^{-1}$, we get two solutions $x_{1,2}=\frac{d\pm \sqrt{d^2-c(r_8r_9l^2)^{-1}}}{2}$ of $f(x)=0$. If $0\le d\le c(l\sqrt{r_8r_9})^{-1}$, we get
\begin{align*}
f(x)&=-cl^{-2}(r_8x)^{-1}+cr_9(d-x)=-cr_8^{-1}l^{-2}x^{-1}-cr_9x+cr_9d\\&\le -c\sqrt{\frac{r_9}{r_8}}\cdot l^{-1}+cr_9d.
\end{align*}
The equality holds when $x_{1,2}=c(l\sqrt{r_8r_9})^{-1}$. In any case, we get $|x_1|<Cl^{-\frac 34}$ and $|x_2-d|<Cl^{-\frac34}$. Hence, we get
$$
|f(x)|>cr^3\mbox{ for all } x\notin B(x_1,\frac r6)\cup B(x_2,\frac r6)
$$
in this case. From the discussion above, we clearly get the bifurcation as $d$ varying.

By choosing
$$
\eta_0=\max\{r_1,r_2\},\ \eta_1=r_3,\ \eta_2=c\sqrt{r_8r_9}^{-1},\ \eta_3=c\sqrt{\frac{r_9}{r_8}} \mbox{ and } \eta_4=cr_9,
$$
we get the corresponding estimates in Lemma~\ref{l.s-deri-resonance}.

From the graph of $f$ and the bifurcation procedure, it is also clear that either $0<x_1\le x_2<d$ if $f(x_1)=f(x_2)=0$; or $x_1=x_2$ if $f(x_1)=f(x_2)\neq 0$.

Finally, to estimate $|\frac{d^2f}{dx^2}|$ when $|\frac{df}{dx}|<r^2$. By (\ref{difference-deri1}), we only need to take care of the case when $l^{-\frac52}<\tan^2[f_1(x)]<l^{-\frac32}$. By (\ref{difference-deri2}), in this case, it is clear that $|\frac{d^2g_2}{dx^2}|$ is dominated by
$$
\frac{l^6\tan[f_1(x)]}{1+l^8\tan^4[f_1(x)]}\cdot[\frac{df_{1}}{dx}(x)]^2,
$$
which is of order at least
\beq\label{second-deri-notsmall}
l^{\frac14}[\frac{df_{1}}{dx}(x)]^2>cl^{\frac 14}r^{4}\gg c
\eeq
since $l\gg r^{-1}$. This completes the proof.
\end{proof}

\section{{\bf Applications of the proof of Theorem~\ref{t.main} and \ref{t.continuity}}}\label{application}

It is clear that the proof the Theorem~\ref{t.iteration}, hence of Theorem~\ref{t.main} and \ref{t.continuity}, can be applied to any one parameter family of cocycle maps  $B\in C^2(\CJ\times\R/\Z, \mathrm{SL}(2,\R))$ such that we could get started with our induction as case $(1)_\I$ and $(2)_\II$ in Section~3.1. Here $\CJ\subset\R$ is any compact interval of parameters. In particular, consider
$$
B^{(t,\l)}=\L(x)\circ R_{\psi(x,t)}=\begin{pmatrix}\l(x)&0\\0&\l^{-1}(x)\end{pmatrix}\cdot\begin{pmatrix}\cos\psi(x,t)&-\sin\psi(x,t)\\
\sin\psi(x,t)&\cos\psi(x,t)\end{pmatrix}
$$
with $\psi(t,x)\in C^2(\CJ\times\R/\Z,\R)$, $\l(x)\in C^2(\R/\Z,\R)$. Assume $\l(x)$ and $\psi(t,x)$ satisfying the following conditions.
\vskip0.2cm
First, $\l(x)>\l\gg 1$ and $\left|\frac{d^m\l(x)}{dx^m}\right|<C\l$ for each $x\in\R/\Z$ and $m=1,2$. For each $t$, $\psi(t;\R/\Z)\subset[0,\pi)$ in $\R\PP^1$.

\vskip0.2cm
Secondly, For each $t\in\CJ$, we have that the set $\CC_t:=\{x:\psi(x,t)=\frac\pi2\}=\{c_{t,1}, c_{t,2}\}$ with the possibility that $c_{t,1}=c_{t,2}$.

\vskip0.2cm
Finally, there exists a $r>0$ such that if we consider the interval $I(t)=I_1(t)\cup I_2(t)$ with $I_{j}(t)=B(c_{j,t},r)$, $j=1,2$, then we have the following.
\begin{itemize}
\item If $I_1(t)\cap I_2(t)=\varnothing$, then $\psi(t;\cdot)$ is of type $\I$ on $I_j(t)$, $j=1,2$. Moreover, if $\psi(t;\cdot)$ is of type $\I_-$ on $I_1(t)$ then it is of type $\I_+$ on $I_2(t)$, vice versa.
\item If $I_1(t)\cap I_2(t)\neq\varnothing$, then $\psi(t;\cdot)$ is of type $\II$ on $I(t)$.
\end{itemize}

Let $L(\a,B^{(t,\l)})$ be the Lyapunov exponents of the dynamical systems $(\a,B^{(t,\l)})$. Then we have that the follow corollary of the proof of Theorem~\ref{t.main} and ~\ref{t.iteration}.
\begin{corollary}\label{c.general}
For the given $B^{(t,\l)}$ as above, for each $\a\in DC_\tau$ with $\tau>2$ and each $\e>0$, there exists a $\l_0=\l_0(\a,B,\e)$ such that
$$
L(\a,B^{(t,\l)})>(1-\e)\log\l
$$
for all $(t,\l)\in\CJ\times[\l_0,\infty)$. Moreover, for any fixed $\l>\l_0$ and for all $t,t'\in\CJ$, it holds that
$$
|L(\a,B^{(t,\l)})-L(\a,B^{(t',\l)})|<Ce^{-c(\log|t-t'|^{-1})^\sigma},
$$
where $c,C>0$ depend on $\a,\psi,\l$, and $0<\sigma<1$ on $\a$.
\end{corollary}

Clearly, Corollary~\ref{c.partial} is a direct consequence of Lemma~\ref{l.reduce-form} and Corollary~\ref{c.general}. We may also apply Corollary~\ref{c.general} to the Szeg\H o cocycles which arise naturally in the study of orthogonal polynomial on the unit circle. See \cite{zhang} for a brief introduction. For detailed information, see \cite{simon1} and \cite{simon2}. In particular, the cocycle map $A^{(E,f)}:\R/\Z\rightarrow \mathrm{SU}(1,1)$ is given by
\beq\label{szego}
A^{(E,f)}(x)=(1-|f(x)|^2)^{-1/2}\begin{pmatrix}\sqrt E& \frac{-\overline{f(x)}}{\sqrt E}\\ f(x)\sqrt E&\frac{1}{\sqrt E} \end{pmatrix},
\eeq
where $E\in\partial\D$, $\D$ is the open unit disk in complex plane $\C$, and $f:\R/\Z\rightarrow\D$ is a measurable function satisfying
$$
\int_X\ln(1-|f|)d\mu>-\infty.
$$
$\mathrm{SU}(1, 1)$ is the subgroup of $\mathrm{SL}(2,\C)$ preserving the unit disk in $\C\PP^1=\C\cup\{\infty\}$ under M\"obius transformations. It is conjugate in $\mathrm{SL}(2,\C)$ to $\mathrm{SL}(2,\R)$ via
$$
Q=\frac{-1}{1+i}\begin{pmatrix}1& -i\\ 1& i \end{pmatrix}\in \mathbb U(2).
$$
In other words, $Q^*\mathrm{SU}(1,1)Q=\mathrm{SL}(2,\R)$. Now consider a function $\t\in C^2(\R/\Z,\R)$ such that $\t(\R/\Z)\subset [0,\frac12)$ and
for some \emph{Diophantine} $\a$, $\t(x)-\t(x-\a)$ is of the same type of function with $v$ in Theorem~\ref{t.main}.

One easy example is that $\t(x)=\frac12\cos(x)$, of which $\t(x)-\t(x-\a)$ is of the same type of function with $v$ for all irrational $\a$. Then, we have the follow corollary of Corollary~\ref{c.general}.
\begin{corollary}\label{c.szego}
Let $f=\l e^{2\pi i[\t(x)+kx]}$, $0<\l<1$, $k\in\Z$ with $\t$ satisfying the above conditions. Let $\a$ be a Diophantine number such that $\t(x)-\t(x-\a)$ is of the same type of function with $v$ in Theorem~\ref{t.main}. Then for each Diophantine $\a$ and each $\e>0$, there exists a $\l_0=\l_0(\t, \a,\e)\in (0,1)$ such that
$$
L(\a, A^{(E,f)})>-\frac12(1-\e)\log(1-\l)
$$
for all $(E,\l)\in \partial D\times[\l_0,1)$. Moreover, for any fixed $\l\in[\l_0,1)$ and for all $E, E'\in\partial\D$, it holds that
$$
|L(\a, A^{(E,f)})-L(\a, A^{(E',f)})|<Ce^{-c(\log|E-E'|^{-1})^\sigma},
$$
where $c,C>0$ depend on $\a,\t,\l$, and $0<\sigma<1$ on $\a$.
\end{corollary}
\begin{proof}
Transform $\mathrm{SU}(1,1)$ to $\mathrm{SL}(2,\R)$, set $E=e^{2\pi t}$ for $0\le t< 1$ and do the polar decomposition. We see that the cocycle map (\ref{szego}) can be transformed into the following form
$$
A^{(\l,t)}=\begin{pmatrix}\sqrt{\frac{1+\l}{1-\l}}&0\\0&\sqrt{\frac{1-\l}{1+\l}}\end{pmatrix}\cdot R_{\psi(x,t)},
$$
where $\psi(x,t)=\pi[\theta(x)-\theta(x-\alpha)+k\alpha+t]$. Clearly, $\psi(x,t)$ satisfies all the conditions of Corollary~\ref{c.general}. This concludes the proof.
\end{proof}

\begin{remark}\label{r.szego}
\cite[Theorem A]{zhang} constructed analytic Szeg\H o cocycles with uniformly positive Lyapunov exponents, which answers a question proposed in \cite[Section 10.16]{simon2} and \cite[Section 3]{damanik}. Clearly, Corollary~\ref{c.szego} is a smooth version, which to the best of our knowledge, is the first example of this kind.
\end{remark}

Finally, let us mention the following application of Theorem~\ref{t.main}. Denote by $\Sigma(H_{\a,v,x})$ the spectrum of the Schr\"odinger operator $H_{\a,v,x}$. Let
$$
\Sigma_{\a,v}=\bigcup_{x\in \R/\Z}\Sigma(H_{\a,v,x}).
$$
Define $L_+(\a,v)=\{E: L(\a,v,E)>0\}$ with $L(\a,v, E)$ the associated Lyapunov exponent.
Recently, Jitomirskaya and Mavi proved the following result in \cite{Jito-Mavi}.
\begin{prop}\label{Jito-Mavi}
For each irrational $\a$, there exists a sequence of rationals $\frac{p_n}{q_n}\rightarrow \a$ such that for any potential $v\in C^{\gamma}(\R/\Z,\R)$ with $\gamma>\frac 12$,
$$
\lim_{n\rightarrow\infty}\Sigma_{\frac{p_n}{q_n},v}\cap L_+(\a,v)=\Sigma_{\a,v}\cap L_+(\a,v).
$$
Moreover, in the \emph{Diophantine} case, the sequence $\frac{p_n}{q_n}$ is the full sequence of continued fraction approximants of $\a$.
\end{prop}
Here, the sequence of bounded measurable sets $B_n\subset\R$ converges to $B$ in the following sense.
\beq\label{set-convergence}
\limsup_{n\rightarrow\infty} B_n=\liminf_{n\rightarrow\infty}B_n=B.
\eeq
Note $\lim_{n\rightarrow\infty}\Sigma_{\a_n,v_n}=\Sigma_{\a,v}$ in Hausdorff metric if $\lim_{n\rightarrow\infty}(\a_n,v_n)=(\a,v)$ in $(\R/\Z,|\cdot|)\times C^0(\R/\Z,\R)$, see \cite[Lemma 12]{aviladamanikzhang}. Thus, the convergence in the sense of (\ref{set-convergence}) is more delicate. In particular, it implies that
$$
\lim_{n\rightarrow\infty}\mathrm{Leb}(B_n)=\mathrm{Leb}(B).
$$
Thus, we have the following immediate corollary of Theorem~\ref{t.main} and Proposition~\ref{Jito-Mavi}.
\begin{corollary}
 Let $\a$ and $v$ be given in Theorem~\ref{t.main}. Let $\frac{p_n}{q_n}$ be the sequence of continued fraction approximants of $\a$. Then there exists $\l_0>0$  such that for any $\l>\l_0$, it holds that
$$
\lim_{n\rightarrow\infty}\Sigma_{\frac{p_n}{q_n},\lambda v}=\Sigma_{\alpha,\lambda v},
$$
which implies that
$$
\lim_{n\rightarrow\infty}\mathrm{Leb}(\Sigma_{\frac{p_n}{q_n},\lambda v})=\mathrm{Leb}(\Sigma_{\alpha,\lambda v}).
$$
\end{corollary}

\end{document}